\documentclass[11pt,a4paper]{amsart}
\usepackage[T1]{fontenc}
\usepackage[utf8]{inputenc}
\usepackage{amsmath,amssymb,amsthm,mathrsfs,mathtools,microtype,booktabs,array}
\usepackage[margin=1.02in]{geometry}
\usepackage[svgnames]{xcolor}
\usepackage[unicode,colorlinks,linkcolor=FireBrick,citecolor=DarkGreen,urlcolor=MidnightBlue]{hyperref}
\hypersetup{pdftitle={A three-block higher-rank deformed Hermitian\textendash Yang\textendash Mills correspondence on P1 x P1}}
\allowdisplaybreaks
\newtheorem{theorem}{Theorem}[section]
\newtheorem{proposition}[theorem]{Proposition}
\newtheorem{lemma}[theorem]{Lemma}

\theoremstyle{definition}
\newtheorem{definition}[theorem]{Definition}

\theoremstyle{remark}

\title[A three-block dHYM correspondence]{A three-block higher-rank deformed Hermitian--Yang--Mills correspondence on \(\mathbb P^1\times\mathbb P^1\)}
\author{Yuqi Li}
\address{C. N. Yang Institute for Theoretical Physics, Stony Brook University, Stony Brook, New York 11794, USA}
\email{yuqi.li@stonybrook.edu}

\author{Hao-Yu Sun}
\address{Weinberg Institute, The University of Texas at Austin, 2515 Speedway, C1600, Austin, Texas 78712, USA}
\address{SpaceXAI, USA}
\email{hkdavidsun@utexas.edu}
\date{October 12, 2026}
\subjclass[2020]{Primary 53C07;
	Secondary 53C55, 32L05, 35J47}
\keywords{deformed Hermitian--Yang--Mills equation,
	equivariant $Z$-stability, holomorphic chains,
	quiver bundles, dimensional reduction,
	nonlinear elliptic systems}

\begin{document}

\begin{abstract}
We study the invariant rank-four deformed Hermitian--Yang--Mills equation
associated with the chain $\mathcal O(1)\to\mathcal O(2)^{\oplus2}\to\mathcal O(3)$ over $\mathbb P^1$, with vertical weights $4,2,0$ and a fixed product
K\"ahler class and phase. For arbitrary linear arrows $\psi=Bz$ and $\phi=(Az)^T$, smooth strictly calibrated reduced solutions exist exactly when $A$ is invertible and $z^TA^TBz$ is not identically zero. The same condition characterizes stability, semistability and polystability in the specified equivariant
coherent category. The classification includes all eleven nonempty
coefficient strata and both coincident and distinct divisor points on
the rank-one lower boundary. The analytic proof retains the ordered,
noncommuting middle equation. After constructing positive-parameter
solutions, it rules out unbounded metric values by combining the source
budgets with moving spectral tests; on the rank-one lower boundary,
the divisor mass must also be retained. Global value, gradient and symbol
estimates then give a smooth strictly calibrated endpoint on the original chain.
No equivalence with full coherent or derived-category stability, uniqueness,
or estimate uniform across changing chains is asserted.
\end{abstract}

\maketitle

\section*{Introduction}

The deformed Hermitian--Yang--Mills equation asks for a Hermitian metric
whose curvature has a prescribed complex phase and a positive calibration.
In the present compact surface construction, the equation is
\[
 \operatorname{Im}\!\left(e^{-i\theta}
          (\widetilde\omega I-F_D)^2\right)=0,\qquad
 \operatorname{Re}\!\left(e^{-i\theta}
          (\widetilde\omega I-F_D)^2\right)>0.
\]
Here \(D\) is the Chern connection of the induced holomorphic bundle;
the second inequality means positive definiteness of the endomorphism
obtained after division by the positive volume form. The raw-curvature
convention, product form and principal phase are fixed explicitly below.
Our question is whether existence of such a metric is characterized
by the phases of all equivariant coherent subobjects.

We consider the chain
\[
 \mathcal O(1)\xrightarrow{\ Bz\ }\mathcal O(2)^{\oplus2}
          \xrightarrow{\ (Az)^T\ }\mathcal O(3),
 \qquad A,B\in\operatorname{Mat}_2(\mathbb C).
\]
It induces a rank-four bundle on
\(\mathbb P^1\times\mathbb P^1\) with three vertical weights.
Even though the two end vertices have rank one, the middle curvature
equation is genuinely matrix-valued. Its two arrow products need not
commute. Removing an arrow, diagonalizing the metric by an extra ansatz,
or grouping two vertices would change the system rather than prove the
statement for these original coefficients.

The result has a simple coefficient description:
\begin{quote}
For every original pair \(A,B\), a smooth strictly calibrated reduced
solution exists if and only if the induced object is equivariantly
\(Z\)-stable, if and only if it is equivariantly \(Z\)-semistable or
\(Z\)-polystable. These equivalent conditions hold precisely when
\(\operatorname{rank}A=2\) and \(Q(z)=z^TA^TBz\not\equiv0\).
\end{quote}
The exact statement is Theorem~\ref{geo-main}. The word \emph{equivariant}
is essential: the category, quotient positivity condition, admissible
subobjects and principal arguments are given in the algebraic section.
The theorem includes the rank-one lower arrow, whose image has one
divisor zero, and does not identify it with a pointwise subbundle at that
zero. It also includes repeated composition roots in the full-rank
family. Metrics are allowed to vary over all six real metric fields
for each fixed chain; no extra symmetry on the base is imposed.

\subsection*{The prescribed-phase equation and dimensional reduction}

Calibrated geometry provides a motivating scalar model. In the local
semi-flat mirror construction of Leung, Yau and Zaslow, a special
Lagrangian section with a flat line connection transforms into a line
bundle satisfying dHYM \cite[\S\S2--3.1]{LYZ2000}.
Jacob and Yau formulate the line-bundle equation on compact K\"ahler
manifolds using a metric-independent curvature integral to specify the
target phase \cite[\S1; Definition 2.1]{JY2014}. These models explain
the prescribed-phase question; they do not determine the ordering,
phase branch or strict matrix calibration in our higher-rank equation.
The correspondence problem asks whether algebraic phase inequalities
detect existence of a metric satisfying that pointwise matrix condition.

In the scalar setting, Collins, Jacob and Yau prove estimates and existence
under the supercritical-phase and subsolution hypotheses of their
Theorems~1.1--1.2 \cite{CJY2015}. These are eigenvalue estimates for a real
\((1,1)\)-form, not estimates for a noncommuting bundle-curvature system.
The physics background also concerns a different setting: Mari\~no,
Minasian, Moore and Strominger derive nonlinear calibrated equations for
abelian worldvolume fields on Calabi--Yau cycles, including the background
\(B\)-field \cite[\S\S3.2.3--3.2.4]{MMMS2000}. Douglas, Fiol and
R\"omelsberger propose \(\Pi\)-stability using lifted central-charge grades
\cite[\S2]{DFR2000}. Neither statement identifies our principal-phase
equivariant coherent test with a physical BPS or derived-category
stability condition, or realizes this rank-four bundle as a nonabelian
brane configuration.

An endomorphism-valued higher-rank dHYM formulation is proposed by
Collins and Yau \cite[\S8.1]{CY2018}. Dervan, McCarthy and Sektnan place
dHYM within the more general \(Z\)-critical equations. Their
large-volume correspondence requires a simple, sufficiently smooth
slope-semistable bundle and uniformly bounded solutions for all
sufficiently large K\"ahler scalings; it uses asymptotic \(Z\)-stability
against subbundles \cite[Theorem 1.1; Definition 2.15; Theorem 4.1]{DMS2024}.
It is not a correspondence at our fixed product class and phase.
Keller and Scarpa prove necessity of \(Z\)-stability for indecomposable
rank-two bundles on compact K\"ahler surfaces carrying \(Z\)-positive
\(Z\)-critical metrics. Their subbundle conclusion extends to saturated
coherent subsheaves under the additional sign condition~(1.4)
\cite[Theorem 1.1]{KS2025}. Their positivity tests the formal derivative
of the charge form, rather than our positive calibrated curvature square;
no rank-four sufficiency theorem is imported from this result.

Two other higher-rank constructions clarify the dimensional-reduction
comparison. Pingali's vector-bundle Monge--Amp\`ere equation has a
rank-two, Griffiths-positive necessity result for MA-stability and a
vortex correspondence on a curve times \(\mathbb P^1\)
\cite[Theorems 1.2--1.3; \S4.4]{Pingali2022}. That vortex bundle has two
line summands and one arrow; its parameters \(r_1,r_2\) are twisting
integers, not vertex ranks. Takahashi proves a rank-two necessity result
for \(J\)-stability on surfaces and, in arbitrary rank, constructs dHYM
solutions at sufficiently small volume from a simple bundle already
carrying a \(J\)-positive solution and satisfying the cohomological
conditions~(1.4) \cite[Theorems 1.2 and 1.6; \S7]{Takahashi2023}.
The \(J\)-equation, its positivity and the small-volume limit do not
supply existence for our fixed-volume three-block equation.

The weighted chain belongs to the dimensional-reduction framework of
\'Alvarez-C\'onsul and Garc\'ia-Prada \cite[\S2.3, Theorem 2.5; \S5.2]{AGGP}:
equivariance along the projective-line factor records bundles as weighted
vertices and holomorphic arrows on the base. Collins and Zhang develop
the dHYM system for two weights and one map from a line bundle to a
vector bundle \cite{CZ2026}. Unlike classical vortex coupling, its
reduced equations contain quadratic covariant-arrow derivatives and
products of arrow endomorphisms with curvature. Here the middle vertex
both receives and sends an arrow. The two arrow products can fail to
commute, and both derivatives enter its equation. Thus the second arrow
changes the matrix analysis, rather than adding a scalar copy of the
one-arrow problem.

Collins and Zhang's correspondence does not supply a three-weight,
two-arrow theorem, and their proved implication from equivariant
Bridgeland stability is not a proved converse equivalence.
Charbonneau, Oliveira and Sena-Dias construct
irreducible higher-rank solutions on the full flag manifold in complex
dimension three \cite{COS2026}; those homogeneous constructions involve
a different curvature power and radius regime. Correa treats rank-one
equations on noncompact crepant resolutions using a Calabi ansatz
\cite{Correa2026}. Neither geometry is the compact three-block system
considered here. These are bounded comparisons with the specified
versions, not a claim of exhaustive prior-art coverage or global priority.

The coherent quiver equivalence just cited is used at its equivariant
coherent scope, not as a dHYM existence theorem. The analytic construction
uses bounded real-Banach compact-map degree and continuation
from Leray and Schauder \cite{LS1934}. Their required hypotheses and the
application are stated in the proof; neither source is invoked as an
existence theorem for this chain.

\subsection*{Outline of the proof}

Proposition~\ref{geo-algebraic-classification} classifies the
equivariantly stable, semistable and polystable locus. It begins with
all coherent subchains, including
divisor modifications and arbitrary finite-length middle quotients.
An explicit charge formula turns their phase comparisons into integral
degree inequalities. The incidence conditions control the unbounded
degree range, while separate zero-charge and negative-real-axis checks
keep the chosen phase branch visible. This yields the coefficient
criterion without an analytic existence premise.

The analytic classification has separate necessity and existence steps.
Proposition~\ref{geo-elementary-exclusions} excludes low upper rank
and a vanishing lower arrow. In the remaining composition-zero case,
the arrows form an Euler sequence. Testing the middle equation with
its metric-orthogonal image and using the endpoint rows gives an
integrated identity incompatible with strict calibration
(Proposition~\ref{geo-euler-exclusion}). These exclusions apply to
arbitrary smooth positive metric triples.

For the two surviving strata, Theorems~\ref{primary-existence} and
\ref{lower-existence} prove existence. Both use the common analytic
construction of Section~\ref{common-section}, which produces a
positive-parameter solution set at one finite auxiliary mass. The
reference inverse retains the coupled first-order terms and the ordered
Lyapunov operator, rather than replacing the middle equation by scalar
multiplication. A compact-map degree argument then produces the positive-stage
component without assuming an undamped endpoint.

The decisive compactness step is to rule out metric-value escape on the
chain. In the full-rank family, the nonzero composition restricts the
limiting eigendirections, while a moving spectral test retains the
upper-arrow kinetic terms. When $B$ has coefficient rank one, its image is $L(-p)$ inside
the saturated line $L$, and the composition has divisor $p+q_0$, also
when these points coincide. A subbundle argument would lose that divisor
mass. The lower-boundary proof therefore keeps atomic and diffuse source contributions,
the damping terms, and the coupled derivative and angular-flux terms.
Local convergence away from the divisor points is used inside this
contradiction, not as a substitute for global endpoint convergence.
Once values are bounded, the gradient estimates and every principal
symbol gap give strong global convergence to a smooth strictly calibrated
solution of the original equation.

Theorem~\ref{geo-main} joins the algebraic classification and these
analytic necessity and existence results on the same coefficient pairs.
They are complementary parts of the correspondence, not two independent
complete proofs of it. The seven sections follow this division: geometry,
algebraic classification, analytic exclusions and the correspondence,
then the common construction and the two value-bound and existence proofs.

\tableofcontents

\section{The three-block construction and its coefficient domain}
\label{geo-setup}

Let $X=\mathbb P^1$, with a fixed K\"ahler form $\omega_X$ satisfying
$\int_X\omega_X=2\pi$. On $\mathcal X=X\times\mathbb P^1$ fix
\[
 \widetilde\omega=p_1^*\omega_X+2p_2^*\omega_{FS},\qquad
 \int_{\mathbb P^1}\omega_{FS}=2\pi.
\]
Thus the vertical normalization is $\sigma=4\pi$ and the curvature
parameter is $q=4\pi/\sigma=1$. Fix a basis
$z=(z_0,z_1)^T$ of $H^0(\mathbb P^1,\mathcal O(1))$. For arbitrary
$A,B\in\operatorname{Mat}_2(\mathbb C)$ consider the original chain
\begin{equation}\label{geo-chain}
 E_2=\mathcal O(1)\xrightarrow{\ \psi=Bz\ }
 E_1=\mathcal O(2)^{\oplus2}\xrightarrow{\ \phi=(Az)^T\ }
 E_0=\mathcal O(3),\qquad Q(z)=z^TA^TBz.
\end{equation}
The transpose in this coefficient formula is complex transpose. The vertex
order $0,1,2$ has ranks $(1,2,1)$ and degrees $(3,4,1)$; the displayed
arrow order has degrees $(1,4,3)$. Coefficient rank records independence
of the global linear components, rather than pointwise bundle-map rank.
There is no imposed relation $\phi\psi=0$.

The primary family has $\operatorname{rank}A=\operatorname{rank}B=2$ and
$Q\not\equiv0$. In this splitting, the upper basis condition is equivalent
to pointwise surjectivity of $\phi$, and the lower basis condition to
pointwise injectivity of $\psi$. The lower arrow has ranks $1\to2$ and is
not an isomorphism. The coefficient domain also includes rank drops
and zero arrows.

\subsection{Invariant holomorphic structure and ordered curvature}

Put $V_j=p_1^*E_j\otimes p_2^*\mathcal O(2j)$ and
$\mathcal E_{A,B}=V_0\oplus V_1\oplus V_2$ as a smooth bundle.
Under the tensor identifications
\[
 \operatorname{Hom}(\mathcal O(2),\mathcal O)
 =\mathcal O(-2)=\operatorname{Hom}(\mathcal O(4),\mathcal O(2)),
\]
choose the same invariant $(0,1)$-form $\alpha$ on both adjacent arrows,
with
$\alpha\wedge\alpha^\dagger=2i\omega_{FS}$.
This equality of the chosen invariant bases is a normalization convention;
independent rescaling of the two bases would require inverse rescaling of
the maps. For smooth positive Hermitian metrics $h_j$ on $E_j$, use
$\widetilde h_j=p_1^*h_j\otimes p_2^*h_{FS}^{2j}$, where
$iF_{h_{FS}}=\omega_{FS}$. With pullbacks suppressed, set
$\beta_1=\phi\alpha$, $\beta_2=\psi\alpha$. The Dolbeault operator and
Chern connection are
\begin{equation}\label{geo-dolbeault}
 \bar\partial_{\mathcal E}=
 \begin{pmatrix}\bar\partial_0&\beta_1&0\\
 0&\bar\partial_1&\beta_2\\0&0&\bar\partial_2\end{pmatrix},\qquad
 D=\begin{pmatrix}D_0&\beta_1&0\\
 -\beta_1^\dagger&D_1&\beta_2\\0&-\beta_2^\dagger&D_2\end{pmatrix}.
\end{equation}
The form $\alpha$ is parallel for the homogeneous connections. Both maps
are holomorphic, and
$\beta_1\wedge\beta_2=(\phi\psi)\alpha\wedge\alpha=0$, since the fiber
is a curve. This proves integrability without a length-two relation.

We use the Chern-connection conventions of
\cite[Ch.~V, \S12]{Demailly2012}:
\[
 D_{h_j}^{1,0}=\partial+h_j^{-1}\partial h_j,\qquad
 F_{h_j}=\bar\partial(h_j^{-1}\partial h_j),\qquad
 H_j=i\Lambda_XF_{h_j},\qquad
 \xi_1=\nabla^{1,0}\phi,\quad\xi_2=\nabla^{1,0}\psi.
\]
The diagonal connections on $V_j$ have curvature
$F_{h_j}-2ji\omega_{FS}$. Multiplying the connection blocks in their
displayed order gives
\begin{equation}\label{geo-curvature}
\begin{aligned}
 F_{00}&=F_{h_0}-2i\phi\phi^\dagger\omega_{FS},\\
 F_{11}&=F_{h_1}-2i\omega_{FS}
       +2i(\phi^\dagger\phi-\psi\psi^\dagger)\omega_{FS},\\
 F_{22}&=F_{h_2}-4i\omega_{FS}+2i\psi^\dagger\psi\omega_{FS},\\
 F_{01}&=\xi_1\wedge\alpha,&F_{10}&=-\xi_1^\dagger\wedge\alpha^\dagger,\\
 F_{12}&=\xi_2\wedge\alpha,&F_{21}&=-\xi_2^\dagger\wedge\alpha^\dagger.
\end{aligned}
\end{equation}
The formal blocks $F_{02}=\beta_1\wedge\beta_2$ and
$F_{20}=\beta_2^\dagger\wedge\beta_1^\dagger$ vanish by vertical bidegree.
For a general tridiagonal connection they need not vanish.

Define the ordered endomorphisms
\begin{gather}
 P_1=\phi\phi^\dagger,\quad R_1=\phi^\dagger\phi,\quad
 P_2=\psi\psi^\dagger,\quad R_2=\psi^\dagger\psi,\label{geo-products}\\
 S_0=P_1,\qquad S_1=I-R_1+P_2,\qquad S_2=2I-R_2.\label{geo-S}
\end{gather}
The derivative endomorphisms are
\begin{equation}\label{geo-K}
\begin{aligned}
 K_0&=\frac i2\Lambda_X(\xi_1\wedge\xi_1^\dagger),\\
 K_1&=-\frac i2\Lambda_X(\xi_1^\dagger\wedge\xi_1)
       +\frac i2\Lambda_X(\xi_2\wedge\xi_2^\dagger),\\
 K_2&=-\frac i2\Lambda_X(\xi_2^\dagger\wedge\xi_2).
\end{aligned}
\end{equation}
In unitary local coordinates, with
$\omega_X=ig_{z\bar z}\,dz\wedge d\bar z$, these are
\[
 K_0=\tfrac12g^{z\bar z}\xi_{1,z}\xi_{1,z}^\dagger,\quad
 K_1=\tfrac12g^{z\bar z}
 (\xi_{1,z}^\dagger\xi_{1,z}+\xi_{2,z}\xi_{2,z}^\dagger),\quad
 K_2=\tfrac12g^{z\bar z}\xi_{2,z}^\dagger\xi_{2,z}.
\]
In particular $K_j\ge0$.

\begin{lemma}[Ordered square]\label{geo-square}
For $\Omega=\widetilde\omega I-F_D$,
\[
 \frac{\Omega^2}{\widetilde\omega^2}
 =\operatorname{diag}(\mathcal A_0+i\mathcal B_0,
                       \mathcal A_1+i\mathcal B_1,
                       \mathcal A_2+i\mathcal B_2),
\]
where
\begin{equation}\label{geo-AB}
 \mathcal A_j=I-\tfrac12(H_jS_j+S_jH_j)+K_j,\qquad
 \mathcal B_j=H_j+S_j.
\end{equation}
\end{lemma}
\begin{proof}
Here $\widetilde\omega^2=4\omega_X\wedge\omega_{FS}$ and
$\Omega_{jj}=(I+iH_j)\omega_X+2(I+iS_j)\omega_{FS}$. Thus the diagonal
square, before including adjacent off-diagonal products, is
\[
 \tfrac12\big[(I+iH_j)(I+iS_j)+(I+iS_j)(I+iH_j)\big]
 =I+i(H_j+S_j)-\tfrac12\{H_j,S_j\}.
\]
For $\gamma_i=\xi_i\wedge\alpha$, graded multiplication gives
\[
 \gamma_i\wedge\gamma_i^\dagger
 =-2i\xi_i\wedge\xi_i^\dagger\wedge\omega_{FS},\qquad
 \gamma_i^\dagger\wedge\gamma_i
 =2i\xi_i^\dagger\wedge\xi_i\wedge\omega_{FS}.
\]
After division by $\widetilde\omega^2$ these contribute exactly
\eqref{geo-K}. A pure base or fiber $(1,1)$-form wedges to zero with each
mixed $\xi_i\wedge\alpha$, and the two consecutive mixed forms wedge to
zero because $\alpha\wedge\alpha=0$; the adjoint products behave likewise.
Products involving both $R_1$ and $P_2$ vanish by fiber form degree. These
facts remove every off-diagonal square block and retain the ordered middle
anticommutator. No commutation of the endomorphisms has been used.
\end{proof}

For a fixed real phase $\theta$, the invariant equation
$\operatorname{Im}(e^{-i\theta}\Omega^2)=0$ and positive-definite
calibration are consequently equivalent to
\[
 \cos\theta\,\mathcal B_j=\sin\theta\,\mathcal A_j,\qquad
 \cos\theta\,\mathcal A_j+\sin\theta\,\mathcal B_j>0
 \quad(j=0,1,2).
\]
The reduction uses holomorphicity of the arrows, their parallel vertical
form, the curve identities, and the Chern conventions. It uses no rank,
injectivity, surjectivity, simplicity, compactness, or commutation
assumption. A merely smooth arrow would have extra derivative terms.
The zero-arrow checks retain the same shifts. If $\psi=0$, the upper
two-block connection and square decouple from the $\mathcal O(4)$ block;
in a rank-one middle block their terms agree with the two-block square.
If $\phi=0$, the lower pair is tensorized by the vertical
$\mathcal O(2)$: its curvature is shifted by $-2i\omega_{FS}I$,
so its square at the fixed K\"ahler form still contains the $I$ and $2I$
shifts of \eqref{geo-S}. It cannot be substituted into an unshifted
two-block existence theorem. With both arrows zero all blocks decouple,
and independent vertex scalars enlarge the chain stabilizer. Simplicity
is a chain-endomorphism property, not a claim about the kernel of the
linearized three-block equation.

\subsection{Charge, phase, and the exact equation}

For any coherent equivariant sheaf use
\begin{equation}\label{geo-charge-definition}
 Z(\mathcal F)=-\int_{\mathcal X}
       e^{-i\widetilde\omega/(2\pi)}\operatorname{ch}(\mathcal F).
\end{equation}
For comparison with normalized-curvature formulas, put
\(\omega_{\rm n}=\widetilde\omega/(2\pi)\) and
\(\mathcal F_D=iF_D/(2\pi)\). Then
\((\widetilde\omega I-F_D)^2=(2\pi)^2
(\omega_{\rm n}I+i\mathcal F_D)^2\).
The Chern--Weil charge forms in
\cite[\S2.3.1, Example 2.25]{DMS2024} and
\cite[\S2]{KS2025} use normalized curvature; any common charge rotation
must be tracked before comparing arguments. The minus sign in
\eqref{geo-charge-definition} and the phase below are fixed here.
The charge depends on the three filtration quotients, not on a holomorphic
splitting. More generally, for volume $V_X$, vertical parameter $\sigma$,
ranks $r_j$ and degrees $d_j$, their sum is
\begin{equation}\label{geo-charge-general}
\begin{split}
 Z(\mathcal E)={}&\frac{\sigma V_X}{4\pi^2}\sum_jr_j-2d_1-4d_2\\
 &+i\left[\frac{\sigma}{2\pi}\sum_jd_j
             +\frac{V_X}{\pi}r_1+\frac{2V_X}{\pi}r_2\right].
\end{split}
\end{equation}
Indeed set $x=p_1^*\omega_X/(2\pi)$, $y=p_2^*\omega_{FS}/(2\pi)$ and
$s_v=\sigma/(2\pi)$. Then $x^2=y^2=0$,
$\int_Xx=V_X/(2\pi)$, $\int_{\mathbb P^1}y=1$, and the $j$th quotient
has Chern character $(r_j+p_1^*c_1(E_j))(1+2jy)$. The degree-four part of
$-e^{-i(x+s_vy)}$ times this character integrates to
\[
 \frac{\sigma V_X}{4\pi^2}r_j-2jd_j
 +i\left(\frac{\sigma}{2\pi}d_j+\frac{jV_X}{\pi}r_j\right).
\]
Summing proves \eqref{geo-charge-general} by additivity of the Chern
character. In the present normalization the three charges are
\[
 2+6i,\qquad -4+12i,\qquad -2+6i,
 \qquad Z(\mathcal E_{A,B})=-4+24i.
\]
Fix the principal phase and its sine and cosine:
\begin{equation}\label{geo-phase}
 \theta_4=\pi-\arctan6\in(\pi/2,\pi),\qquad
 \chi=-\frac1{\sqrt{37}},\qquad\eta=\frac6{\sqrt{37}}.
\end{equation}
Write $\mathrm{Ex}(A,B)$ for existence of a smooth positive metric triple
on the original bundles satisfying
\begin{equation}\label{geo-equation}
 \chi\mathcal B_j=\eta\mathcal A_j,\qquad
 \mathcal C_j:=\chi\mathcal A_j+\eta\mathcal B_j>0
 \quad(j=0,1,2).
\end{equation}
All adjoints and covariant derivatives in this equation use that actual
metric triple. No auxiliary base symmetry restricts these unknowns.

\begin{lemma}[Coefficient strata and original divisors]\label{geo-domain}
There are precisely eleven nonempty triples
$(a,b,\epsilon)=(\operatorname{rank}A,\operatorname{rank}B,
\mathbf1_{\{Q\not\equiv0\}})$. They and exact representatives are listed
below, with $O=0$, $D=\operatorname{diag}(1,0)$,
$D^\perp=\operatorname{diag}(0,1)$ and
$J=\left(\begin{smallmatrix}0&1\\-1&0\end{smallmatrix}\right)$.
\[
\begin{array}{c|c|c|c}
 (a,b,\epsilon)&A&B&Q(z)\\\hline
 (0,0,0)&O&O&0\\(0,1,0)&O&D&0\\(0,2,0)&O&I&0\\
 (1,0,0)&D&O&0\\(1,1,0)&D&D^\perp&0\\
 (1,1,1)&D&D&z_0^2\\(1,2,1)&D&I&z_0^2\\
 (2,0,0)&I&O&0\\(2,1,1)&I&D&z_0^2\\
 (2,2,0)&I&J&0\\(2,2,1)&I&I&z_0^2+z_1^2
\end{array}
\]
If $a=2,b=1$, the original lower map has the form
$\psi=v\sigma$, with saturated line $L=\mathcal O(2)v$,
$\operatorname{im}\psi=L(-p)$, and
$\operatorname{div}(\phi|_L)=q_0$ for simple divisor points $p,q_0$.
Then $\operatorname{div}Q=p+q_0$, allowing $p=q_0$.
\end{lemma}
\begin{proof}
Put $C=A^TB$. Expanding $z^TCz$ shows that $Q\equiv0$ exactly when
$C+C^T=0$. Every skew $2\times2$ matrix is $\lambda J$, of rank zero
if $\lambda=0$ and rank two otherwise. If a matrix has rank zero the
composition vanishes. If the coefficient ranks are $(1,2)$ or $(2,1)$,
$C$ has rank one, so its quadratic form is nonzero. In the $(1,1)$ case
write $A=u\alpha_c^T$, $B=v\beta_c^T$; then
\[
 C=\alpha_c(u^Tv)\beta_c^T,\qquad
 Q=(u^Tv)(\alpha_c^Tz)(\beta_c^Tz).
\]
Since the polynomial ring has no zero divisors, $Q=0$ precisely when
$u^Tv=0$. If both ranks are two, $C$ is invertible and can be skew or
have a nonzero symmetric part. This proves exhaustion, and the table
proves nonemptiness. The seven impossible triples are the five zero-rank
pairs with $\epsilon=1$, together with $(1,2,0)$ and $(2,1,0)$.

For $a=2,b=1$, factor the actual matrix $B=v\beta_c^T$ with nonzero
vectors. The nonzero section $\sigma=\beta_c^Tz$ has one simple zero
$p$, so $\psi=v\sigma$ is sheaf-injective and has image $L(-p)$; its
fiber map does vanish at $p$. Also
$\phi|_L=(A^Tv)^Tz\ne0$, since $A$ is invertible. This section of
$\operatorname{Hom}(L,E_0)=\mathcal O(1)$ has one simple zero $q_0$.
Their product is nonzero with divisor $p+q_0$. For any prescribed
$p,q_0$, choose nonzero linear sections $\sigma,\tau$ with those
divisors and put
$v=(A^T)^{-1}\operatorname{coeff}(\tau)$,
$B=v\operatorname{coeff}(\sigma)^T$. Thus no ordered pair is omitted.
For $A=I$, $v=e_0$, $\sigma=z_0$ gives $p=q_0=[0:1]$, while
$\sigma=z_1$ gives $p=[1:0]\ne q_0=[0:1]$.
Inverse scalar rescaling of $v,\sigma$ does not alter the original map
or the divisors. No metric splitting along $L$ or base change is used.

Repeated primary roots also remain: $A=I$ and
$B=\left(\begin{smallmatrix}1&1\\-1&0\end{smallmatrix}\right)$ have
both ranks two and $Q=z_0^2$.
\end{proof}

\subsection{Isomorphisms and transport of arbitrary metrics}

Allowed chain isomorphisms cover $\operatorname{id}_X$ and have the form
$(g_0,g_1,g_2)=(a,G,c)$, with nonzero scalars $a,c$ and
$G\in GL_2(\mathbb C)$. Their coefficient action is
\begin{equation}\label{geo-coefficient-action}
 (A,B)\longmapsto(aG^{-T}A,c^{-1}GB),\qquad
 A'^TB'=(a/c)A^TB.
\end{equation}
Consequently the primary orbit invariant is $[A^TB]$. Normalize by
$G=A^T$, $a=c=1$ to obtain $(I,A^TB)$. Transformations preserving the
upper normalized matrix have $G=aI$ and multiply the lower matrix by
$a/c$. Thus the primary orbit set is precisely
\[
 \mathcal U_4=\{[C]\in\mathbb P^3:\det C\ne0,\ C+C^T\ne0\}.
\]
This asserts an orbit set, rather than a universal family or separated
coarse moduli space. Every primary chain is simple: a chain endomorphism
satisfies $f_0\phi=\phi f_1$ and $f_1\psi=\psi f_2$; since the upper
components form a basis, the first relation forces $f_1=f_0I$, and the
nonzero lower arrow then forces $f_2=f_0$.

The primary family is not one base-isometry orbit. The matrices
\[
 B_0=\begin{pmatrix}1&1\\-1&0\end{pmatrix},\qquad
 B_1=\begin{pmatrix}0&2\\-1&0\end{pmatrix}
\]
have symmetric parts of ranks one and two. A base change acts on the
projective composition by $[C]\mapsto[T^TCT]$, so these ranks separate
even the $PGL_2$ orbits, and hence the $PSU(2)$ orbits. Moreover for
$B_*=\left(\begin{smallmatrix}1&0\\1&1\end{smallmatrix}\right)$,
at $z=(1,0)^T$ with standard local metrics, $R_1$ projects onto the first
coordinate while $P_2$ has image spanned by $(1,1)^T$; their commutator is
nonzero. The middle equation therefore cannot be scalarized by a
commuting representative.

\begin{lemma}[Metric covariance]\label{geo-metric-covariance}
An allowed holomorphic chain isomorphism gives a bijection of all smooth
positive metric triples, preserves the ordered equation and strict
calibration, and also intertwines the invariant equation on $\mathcal X$.
\end{lemma}
\begin{proof}
Transport metrics by
$h'_j=(g_j^{-1})^*h_jg_j^{-1}$, with inverse
$h_j=g_j^*h'_jg_j$. Then $g_j:(E_j,h_j)\to(E_j,h'_j)$ is an isometry.
For an arrow $\Phi:E_s\to E_t$,
\[
 \Phi'=g_t\Phi g_s^{-1},\qquad
 (\Phi')^{\dagger_{h'}}=g_s\Phi^{\dagger_h}g_t^{-1}.
\]
Hence each of $P_1,R_1,P_2,R_2$ is conjugated at its own vertex. The
Chern connections and curvatures obey
\[
 D'_j=g_jD_jg_j^{-1},\quad F'_j=g_jF_jg_j^{-1},\quad
 H'_j=g_jH_jg_j^{-1},\quad
 \xi'_i=g_{i-1}\xi_i g_i^{-1}.
\]
The last formula follows by applying the transported source and target
connections; their derivative terms cancel. The allowed automorphisms
are constant, so no local $\partial g_j$ term is needed. Wedge
composition and metric adjoints give
\[
 K'_j=g_jK_jg_j^{-1},\quad S'_j=g_jS_jg_j^{-1},\quad
 \mathcal A'_j=g_j\mathcal A_jg_j^{-1},\quad
 \mathcal B'_j=g_j\mathcal B_jg_j^{-1}.
\]
In particular the entire anticommutator is conjugated, with no change in
order. The calibration is also conjugated, and for $w'\ne0$,
\[
 h'_j(\mathcal C'_jw',w')
 =h_j(\mathcal C_jg_j^{-1}w',g_j^{-1}w').
\]
This proves positivity in the transported metric.
The induced $\widetilde g=\bigoplus_jp_1^*g_j\otimes I_{\mathcal O(2j)}$
intertwines \eqref{geo-dolbeault}, commutes with the vertical $SU(2)$
action, and fixes the K\"ahler form. It conjugates $F_D$ and
$(\widetilde\omega I-F_D)^2$, including their Hermitian real and imaginary
parts, proving invariant transfer of the equation and calibration.
\end{proof}

For $(a,b,\epsilon)=(2,2,0)$, $A^TB=\lambda J$ with $\lambda\ne0$.
Taking $(g_0,g_1,g_2)=(1,A^T,\lambda)$ gives $(A',B')=(I,J)$, namely
\[
 \mathcal O(1)\xrightarrow{(z_1,-z_0)^T}\mathcal O(2)^2
       \xrightarrow{(z_0,z_1)}\mathcal O(3).
\]
The lower image is the upper kernel. There is no holomorphic or metric
splitting implicit here. Pullback of arbitrary representative metrics is
$h_0=h'_0$, $h_1=(A^T)^*h'_1A^T$, $h_2=|\lambda|^2h'_2$.
Thus a result for arbitrary metrics transfers to every coefficient pair;
a result for a restricted metric class would not suffice. These
fixed-chain bijections give no uniform estimate as $A$ loses conditioning
or $\lambda\to0$. A common $\mathcal O(2)$ tensor shift likewise preserves
only orbit geometry, since it changes degrees, charge, phase, and the
curvature terms of the equation.

\section{All equivariant coherent subobjects and their phases}

\subsection{The category and the compact-to-complex bridge}

The category is $\mathcal D_K=\operatorname{Coh}^{SU(2)}(\mathcal X)$,
with vertical action only. Set
\[
 \mathbb H_{\rm res}=\{\operatorname{Im}z>0\}
      \cup\{\operatorname{Im}z=0,\operatorname{Re}z<0\}.
\]
For either $\mathcal D_K$ or the separately defined ordinary coherent
category, the torsion class is
\[
 \mathcal T_{\mathcal D}=
 \{F:F\text{ is torsion}\}\ \cup\
 \{F:\operatorname{Im}Z(Q')>0\text{ for every nonzero quotient }
                                      F\twoheadrightarrow Q'\text{ in }\mathcal D\}.
\]
For an object with charge in $\mathbb H_{\rm res}$, stability means
membership in this class and
$\operatorname{Arg}Z(S)<\operatorname{Arg}Z(E)$ for every nonzero proper
exact subobject $S$ whose charge lies in $\mathbb H_{\rm res}$.
For its nonzero quotient in $0\to S\to E\to Q'\to0$, this is
equivalently $\operatorname{Arg}Z(E)<\operatorname{Arg}Z(Q')$.
Arguments are principal in $(0,\pi]$. Semistability uses $\le$;
polystability means semistability together with a finite direct sum in
the same category into nonzero stable objects of the same phase. The
ordinary coherent version remains separate throughout.

The algebraic input we use is the coherent-category equivalence of
\'Alvarez-C\'onsul and Garc\'ia-Prada, in
arXiv:math/0112160v2, \S2.3, Theorem~2.5 (the correspondence between
equivariant coherent sheaves and quiver sheaves), together with the
one-projective-line construction in \S1.6.2, ``Homogeneous bundles on
products of complex projective lines,'' and its chain-filtration
specialization in \S5.2, equation~(5.6) (TeX label \texttt{eq:hfil-P\^{}1})
\cite{AGGP}. In its notation, $G$ is connected, simply
connected, complex semisimple; $P$ is parabolic; $X$ is compact K\"ahler
with trivial $G$-action; and equivariant coherent actions are holomorphic.
The theorem is an equivalence of abelian categories between coherent
$G$-equivariant sheaves on $X\times G/P$ and finite-support coherent
quiver sheaves on $X$, and it matches locally free objects and their
filtration quotients.

For our application $G=SL_2(\mathbb C)$ and
$P=\left\{\left(\begin{smallmatrix}*&0\\ *&*\end{smallmatrix}\right)\right\}$,
so $G/P=\mathbb P^1$. Its nilradical has dimension one. The arrows lower
weights by two, and the relation space is zero because the second exterior
power of the nilradical is zero. The weights $0,2,4$ therefore give exactly
coherent three-vertex chains, with filtration quotients
$p_1^*S_j\otimes p_2^*\mathcal O(2j)$. Exactness preserves kernels,
cokernels and weight support. Thus no weight outside these three can enter
a subobject of the present bundle. The source proof uses restriction and
induction to the $P$-slice, coherent Levi multiplicity sheaves and the
nilradical representation; its coherent invariant-section and presentation
lemmas are part of the imported theorem. Its cited Roberts foundations
are not separately re-proved here.

\begin{lemma}[Compact and complex equivariant subobjects]\label{geo-bridge}
Restriction from $SL_2(\mathbb C)$ to $SU(2)$ gives a bijection on coherent
equivariant subobjects of every bundle \eqref{geo-dolbeault}. The underlying
image sheaf, monomorphism, quotient, nonsaturation, and nonreduced structure
are preserved.
\end{lemma}
\begin{proof}
An equivariant monomorphism is determined by its invariant image
$\mathcal S\subset\mathcal E$: injectivity forces its linearization to
be the restricted one. An invariant image and its quotient inherit their
linearizations. It is therefore enough to extend invariance of
$\mathcal S$ from $K=SU(2)$ to its complexification $G=SL_2(\mathbb C)$.

For $\xi\in\mathfrak{su}(2)$, consider $\exp(z\xi)$, $z\in\mathbb C$.
Since $\mathcal X$ is projective, GAGA
\cite[\S12, Theorems 2--3]{SerreGAGA} places the coherent quotient
$q:\mathcal E\twoheadrightarrow\mathcal E/\mathcal S$ at a point of the
projective Quot scheme with its Hilbert polynomial, by Grothendieck's
construction \cite[Theorem 5.1]{NitsureQuot}. The given
$G$-linearization induces a holomorphic action on this scheme. The orbit
map $z\mapsto\exp(z\xi)\cdot[q]$ is constant on the real axis by
$K$-invariance. The holomorphic identity theorem, applied in local
coordinates and then using separatedness, makes it constant on
$\mathbb C$. Equality of Quot points gives equality of kernels. Thus
$\mathcal S$ is invariant under $\exp(i\xi)$ for every $\xi$.
Polar decomposition $G=K\exp(i\mathfrak{su}(2))$ proves $G$-invariance.
The inherited $G$-linearization is unique and gives the quotient
linearization in the unchanged exact coherent sequence. No saturation or
local freeness on the surface was required.

One can also see the extension in a finite-dimensional representation.
Take the naturally linearized ample bundle
$L_a=p_1^*\mathcal O(1)\otimes p_2^*\mathcal O(1)$. For $n$ sufficiently
large, Serre global generation, as recalled in
\cite[\S2, Lemma 2.1(c)]{NitsureQuot}, gives
\[
 W=H^0(\mathcal X,\mathcal S\otimes L_a^n)
       \subset H^0(\mathcal X,\mathcal E\otimes L_a^n)=V.
\]
This is a $K$-invariant complex subspace of a finite-dimensional
holomorphic $G$-representation. Differentiating real one-parameter
invariance gives $d\rho(\xi)W\subset W$; complex linearity gives
$d\rho(i\xi)W\subset W$. Polar decomposition then makes $W$
$G$-invariant. The equivariant evaluation map
$W\otimes L_a^{-n}\to\mathcal E$ has image exactly $\mathcal S$.
Both arguments preserve all finite-support quotient data and apply to
every boundary map, whose induced linearization is unchanged.
\end{proof}

Combining this lemma with the precise external equivalence identifies the
required subobjects with all coherent subchains. It does not identify
equivariant and ordinary coherent stability.

\subsection{Complete sheaf inventory and quotient positivity}

Write a subchain as $S=(S_0,S_1,S_2)$, with ranks $\rho_j$ and degrees
$e_j$; absent components have rank and degree zero. All its components
are subsheaves of vector bundles on the smooth curve, hence torsion-free
and locally free. A nonzero rank-zero vertex torsion sheaf cannot inject.
Every possible component is of the following form:
\begin{itemize}
 \item $S_0=\mathcal O(3)(-D_0)$ if $\rho_0=1$, with any effective
 divisor $D_0$, so $e_0=3-\deg D_0$;
 \item $S_2=\mathcal O(1)(-D_2)$ if $\rho_2=1$, so
 $e_2=1-\deg D_2$;
 \item if $\rho_1=1$, its saturation is an embedded line
 $\mathcal O(k)\subset\mathcal O(2)^2$, with $k\le2$, and
 $S_1=\mathcal O(k)(-D_1)$, so $e_1=k-\deg D_1$;
 \item if $\rho_1=2$, it is the kernel of an arbitrary finite-length
 quotient $E_1\twoheadrightarrow T_1$, of length $\ell_1=4-e_1$.
\end{itemize}
The divisors and finite quotients allow arbitrary multiplicity and
nonreduced support. The exact, necessary and sufficient incidences are
\begin{equation}\label{geo-sheaf-incidences}
 \phi(S_1)\subset S_0,\qquad\psi(S_2)\subset S_1.
\end{equation}
For clarity, their coefficient refinements are as follows. If $a=2$,
$K_A=\ker\phi=\mathcal O(1)$; if $a=1$, then
$K_A=\mathcal O(2)$ and $\operatorname{im}\phi=\mathcal O(3)(-p_A)$ for
one simple point $p_A$. If $a=0$, the upper arrow imposes no condition.
If $b=2$, $I_B=\operatorname{im}\psi=\mathcal O(1)$ is saturated; if
$b=1$, it is $M_B(-p_B)$ inside the constant summand
$M_B=\mathcal O(2)$; if $b=0$, there is no lower condition.
When $a>0$, a middle line killed by $\phi$ is exactly a divisor modification of
$K_A$. Otherwise, if its saturated upper restriction has image divisor
$D_\phi$, its actual image is
$\mathcal O(3)(-D_\phi-D_1)$, and the first incidence says
$D_0\le D_\phi+D_1$. When $S_2\ne0$, a middle line is
$I_B(-D_1)$ with $D_1\le D_2$ for $b=2$, or
$M_B(-D_1)$ with $D_1\le p_B+D_2$ for $b=1$.
For rank-two $S_1$, the finite quotient $E_1/S_1$ must annihilate
$\psi(S_2)$, and $E_0/S_0$ must annihilate $\phi(S_1)$.
These are sheaf incidences, retaining non-diagonal quotients, support and
all multiplicities, rather than mere numerical bounds.

At the generic point the complete nonzero rank-pattern lists are
\begin{equation}\label{geo-all-patterns}
\begin{array}{c|l}
 \phi,\psi\ne0, \phi\psi\ne0&100,010,110,120,111,121\\
 \phi,\psi\ne0, \phi\psi=0&100,010,110,120,111,121,011\\
 \phi\ne0, \psi=0&010,011,001,100,101,110,111,120,121\\
 \phi=0, \psi\ne0&010,020,011,021,100,110,120,111,121\\
 \phi=\psi=0&\text{every nonzero vector in }
                    \{0,1\}\times\{0,1,2\}\times\{0,1\}.
\end{array}
\end{equation}
Here the digits are subobject ranks, not coefficient strata. If
$\rho_0=0$ and $\phi\ne0$, the middle component lies in its generic
kernel and has rank at most one. A nonzero composition then forbids
$\rho_2=1$. If $\rho_1=0$ and $\psi\ne0$, the bottom component must
vanish. These restrictions, or their removal when an arrow vanishes,
give precisely the lists above; the incidences refine them at every
point. The pattern $121$ includes full-rank finite-colength subchains
as well as the improper whole chain.

\begin{lemma}[Every quotient has positive imaginary charge]\label{geo-torsion}
Every coefficient pair belongs to $\mathcal T_{\mathcal D_K}$, including
all nonsaturated and mixed coherent quotients.
\end{lemma}
\begin{proof}
The charge computation above, applied to coherent filtration quotients,
gives
\begin{equation}\label{geo-subobject-charge}
 Z(S)=2(A_Z+iB_Z),\qquad
 A_Z=\rho_0+\rho_1+\rho_2-e_1-2e_2,\qquad
 B_Z=e_0+e_1+e_2+\rho_1+2\rho_2.
\end{equation}
For the ambient chain these numbers are $(-2,12)$. The component bounds
are $e_0\le3$, $e_2\le1$, and $e_1\le2$ or $4$ according as its rank is
one or two. Omitting any ambient rank reduces the maximum $B_Z$ from
$3+6+3=12$ to at most $9$. For full ranks, put
$e_0=3-\ell_0$, $e_1=4-\ell_1$, $e_2=1-\ell_2$; then
$B_Z=12-\ell_0-\ell_1-\ell_2$. Equality forces all quotient lengths
zero and the entire chain. Thus every proper subchain has $B_Z<12$ and
every nonzero quotient has
$\operatorname{Im}Z(\mathcal E/S)=2(12-B_Z)>0$.
This argument includes arbitrary finite-colength and nonreduced quotients.
\end{proof}

\subsection{The strict chamber and the exceptional phase cases}

Rotation by the fixed phase gives
\begin{equation}\label{geo-phase-functional}
 \operatorname{Im}(e^{-i\theta_4}Z(S))
 =-\frac2{\sqrt{37}}F(S),\qquad
 F(S)=B_Z+6A_Z
 =e_0-5e_1-11e_2+6\rho_0+7\rho_1+8\rho_2.
\end{equation}
For charges in $\mathbb H_{\rm res}$, principal phase decrease is
equivalent to $F>0$. Indeed the phase difference lies in $(-\pi,\pi)$,
so its sine has its sign. Equality $F=0$ is the only same-phase wall.
The sink $(E_0,0,0)$ has $F=9$, fixing the orientation.

\begin{proposition}[Equivariant algebraic classification]
\label{geo-algebraic-classification}
For every original coefficient pair, equivariant stability,
semistability and polystability coincide, and hold exactly when
$a=2$, $b>0$, $Q\not\equiv0$.
\end{proposition}
\begin{proof}
Assume first these three conditions. The complete rank list is the first
row of \eqref{geo-all-patterns}, with $K_A=\mathcal O(1)$. If the pattern
is $010$, then $S_1\subset K_A$ and $e_1\le1$. For $110$ either the
same inclusion holds, or the nonzero upper restriction gives $e_0\ge e_1$.
For $120$, surjectivity of $\phi$ induces a surjection
$E_1/S_1\twoheadrightarrow E_0/S_0$; hence
$\ell_0\le\ell_1$, or $e_0\ge e_1-1$. In pattern $111$ both restrictions
are nonzero, so $e_0\ge e_1\ge e_2$, $e_2\le1$ and $e_1\le2$.
If $b=2$ its middle saturation is $\mathcal O(1)$ and $e_1\le1$; if
$b=1$ the saturation is $M_B=\mathcal O(2)$, retaining the length-one
torsion between $\operatorname{im}\psi$ and $M_B$. For $121$ the same
upper quotient surjection gives $\ell_0\le\ell_1$.

For $B_Z>0$, these actual incidences give the following bounds:
\begin{equation}\label{geo-strict-bounds}
\begin{array}{c|l|c}
 \rho&\text{incidence or admitted-charge condition}&F(S)\\\hline
 100&e_0>0&e_0+6\ge7\\
 010&e_1\le1&7-5e_1\ge2\\
 110, S_1\subset K_A&e_1\le1, e_0\ge-e_1&F\ge13-6e_1\ge7\\
 110, S_1\not\subset K_A&e_0\ge e_1, e_1\le2&F\ge13-4e_1\ge5\\
 120&e_0\ge e_1-1, e_1\le4&F\ge19-4e_1\ge3\\
 111&e_0\ge e_1\ge e_2, e_1\le2, e_2\le1&F\ge2\\
 121&\ell_0\le\ell_1&F=-\ell_0+5\ell_1+11\ell_2>0.
\end{array}
\end{equation}
For $111$, the estimate is $F\ge21-4e_1-11e_2\ge2$.
For proper $121$, $F\ge4\ell_1+11\ell_2$; equality would give
$\ell_0=\ell_1=\ell_2=0$ and the whole chain.

The admitted negative real axis has not been omitted. If $B_Z=0$, the
same incidences give $A_Z>0$: for $100,010$ it is respectively $1,2$;
for the killed $110$ case $A_Z=2-e_1\ge1$, and otherwise
$e_0=-e_1-1\ge e_1$ forces $e_1\le-1$. In $120$,
$e_0=-e_1-2\ge e_1-1$ forces $e_1\le-1$, so $A_Z\ge4$.
For $111$, $e_0+e_1+e_2=-3$ and $e_0\ge e_2$ imply
$e_1+2e_2\le-3$, hence $A_Z\ge6$. For $121$,
$\ell_0+\ell_1+\ell_2=12$ and $\ell_0\le\ell_1$ imply
$\ell_1+\ell_2\ge6$, so $A_Z=-2+\ell_1+2\ell_2\ge4$.
None of these charges is negative real. Lemma~\ref{geo-torsion} and
\eqref{geo-strict-bounds} prove stability, including every middle
embedding and nonsaturated modification.

Every other coefficient stratum has the following strict destabilizer:
\begin{equation}\label{geo-destabilizers}
\begin{array}{c|c|c|c}
 \text{condition}&S&(A_Z,B_Z)&F(S)\\\hline
 a=0&(0,E_1,0)&(-2,6)&-6\\
 a=1&(0,K_A,0),\ K_A=\mathcal O(2)&(-1,3)&-3\\
 a=2, b=0&(0,0,E_2)&(-1,3)&-3\\
 (a,b,\epsilon)=(2,2,0)&(0,K_A,E_2)&(-1,5)&-1.
\end{array}
\end{equation}
In the last row $\operatorname{im}\psi=\ker\phi=\mathcal O(1)$,
so it is a legal subchain. Each charge has positive imaginary part and
phase strictly above $\theta_4$. The eleven-stratum exhaustion in
Lemma~\ref{geo-domain} shows that these rows cover every other pair.
They exclude semistability. A stable object is polystable through its
one-summand decomposition; conversely the definition of polystability
includes semistability. The three loci therefore coincide.
\end{proof}

For completeness, the same-phase and zero-charge cases admit a complete
degree and sheaf description. For every nonzero rank vector the wall
functional is
\begin{equation}\label{geo-all-walls}
\begin{array}{c|c@{\qquad}c|c}
100&6+e_0&010&7-5e_1\\
020&14-5e_1&001&8-11e_2\\
101&14+e_0-11e_2&110&13+e_0-5e_1\\
120&20+e_0-5e_1&011&15-5e_1-11e_2\\
021&22-5e_1-11e_2&111&21+e_0-5e_1-11e_2\\
121&28+e_0-5e_1-11e_2&&
\end{array}
\end{equation}
Solving $F=0$ with the component-degree bounds, integral degrees, and
$B_Z>0$ gives precisely
\begin{equation}\label{geo-wall-tuples}
\begin{array}{c|c|c}
\rho&(e_0,e_1,e_2)&\text{realizable coefficient strata}\\\hline
120&(0,4,0)&a=0\\
111&(0,2,1)&(0,0,0),(0,1,0),(1,0,0),(1,1,0)\\
121&(-2,3,1)&a=0.
\end{array}
\end{equation}
All have charge $2(-1+6i)=Z(\mathcal E)/2$; the only additional tuple is
$(3,4,1)$ in rank $121$, namely the improper whole chain.
One way to check the unbounded degree calculation is to eliminate $e_0$.
For $120$, $B_Z=6e_1-18>0$ forces $e_1=4$.
For $111$, $B_Z=6e_1+12e_2-18>0$ forces
$e_1+2e_2\ge4$, whose upper bounds force $(e_1,e_2)=(2,1)$.
For $121$, $B_Z=6e_1+12e_2-24>0$ forces
$e_1+2e_2\ge5$, leaving $(e_1,e_2)=(3,1),(4,1)$.
In $100,010,020,001,101,110$, their one-variable wall equations give
no integral tuple with positive $B_Z$ and the stated upper bounds.
For $011$, $5e_1+11e_2=15$ gives $e_1=3-11k$, $e_2=5k$;
$e_1\le2,e_2\le1$ are incompatible. For $021$,
$5e_1+11e_2=22$ gives $e_1=11k$, $e_2=2-5k$, again incompatible
with $e_1\le4,e_2\le1$. Thus no finite degree window is being used.

The first wall is realized by $(\mathcal O,E_1,0)$ whenever $a=0$.
It cannot occur for a nonzero upper arrow, whose image has colength at
most one, whereas $S_0=\mathcal O$ has colength three.
For the second, take a constant $\mathcal O(2)$ middle line killed by
$\phi$. When $b=1$ this line is forced to be $M_B$, and with $a=1$
the required incidence is $M_B=K_A$, exactly zero composition. If $b=2$
the saturated lower image has degree one and cannot lie in such a
different saturated line; if $a=2$ the upper kernel has degree one.
These observations give exactly the four listed strata. For the third,
take a colength-one middle subbundle containing $\psi(E_2)$ and an
arbitrary $S_0=\mathcal O(-2)$ of colength five when $a=0$.
Such a middle component is obtained by taking the inverse image of
the kernel of a length-one quotient of $E_1/\operatorname{im}\psi$
when $b=2$, or of $E_1/M_B$ when $b=1$; when $b=0$ any length-one
quotient of $E_1$ suffices.
For $a\ne0$, a length-one modification can make the upper image have
colength at most two, so the required colength five is impossible.
Every wall-bearing stratum already has a strict destabilizer in
\eqref{geo-destabilizers}. Thus equality restores no semistability;
there is no proper wall on either stable stratum.

The solutions of $A_Z=B_Z=0$ with the same component bounds are
\begin{equation}\label{geo-zero-tuples}
\begin{array}{c|c|l}
\rho&(e_0,e_1,e_2)&\text{realizable coefficient strata}\\\hline
101&(-3,0,1)&B=0\\
110&(-3,2,0)&a\le1\\
120&(-5,3,0)&a=0\\
111&(-5,1,1)&a=0; (1,0,0),(1,1,0),(2,0,0),(2,2,0)\\
121&(-8,4,0),\ (-7,2,1)&a=0.
\end{array}
\end{equation}
Indeed $A_Z=0$ says $e_1+2e_2=\rho_0+\rho_1+\rho_2$ and
$B_Z=0$ determines $e_0$; the upper degree bounds give precisely the
displayed tuples, and none for the other ranks. Their realization also
follows from \eqref{geo-sheaf-incidences}. In $101$ the bottom cannot be
present unless $B=0$. In $110$ the degree-two middle line must be killed
by $\phi$, available exactly when $a\le1$. In $120$, a length-one middle
modification cannot map into the colength-eight top subsheaf unless
$a=0$. In $111$ the middle line has degree one and contains the entire
lower image. For $a=0$ this is possible for every lower map. For $a>0$
it must be killed by $\phi$, since the top degree is $-5$; this gives
precisely the four additional strata shown. For $121$, the first tuple
has the whole middle bundle and the second a length-two middle
modification containing $\psi(E_2)$. They are realizable for $a=0$:
for the second, take the inverse image of the kernel of a length-two
quotient of $E_1/\operatorname{im}\psi$ when $b=2$, or of $E_1/M_B$
when $b=1$; when $b=0$ take any length-two quotient of $E_1$.
For $a>0$, their upper-image colengths are at most one and three,
respectively, insufficient for the top colengths eleven and ten.
Each tuple is outside $\mathbb H_{\rm res}$ and lies on an already
strictly unstable stratum. Also $F=0=B_Z$ would force zero charge;
there is no admitted zero-imaginary same-phase wall or phase ambiguity.

The two initial filtration charges provide useful checks on signs:
\[
 Z(E_0,0,0)=2+6i,\quad Z(\mathcal E/(E_0,0,0))=-6+18i,
\]
\[
 Z(E_0,E_1,0)=-2+18i,\quad Z(0,0,E_2)=-2+6i.
\]
Their phase cross-products against $-4+24i$ are $-72,-24$.
For primary maps, $E_1/\psi(E_2)$ is a degree-three line, with charge
imaginary part $8$ in the middle weight. The quotients by
$(0,K_A,0)$, $(E_0,K_A,0)$ and $(E_0,\operatorname{im}\psi,0)$ have
imaginary charges $20,14,14$. A finite-length vertex quotient of length
$\ell$ has imaginary charge $2\ell>0$.
These checks are consequences of the complete inventory, rather than
replacements for it.

\section{Metric positivity and the excluded coefficient strata}

\subsection{Pointwise and integrated identities}

\begin{lemma}[Lyapunov identities and the symbol]\label{geo-lyapunov}
At a solution of \eqref{geo-equation}, put
$T_j=\chi I+\eta S_j$ and $L_j=\chi I+\eta H_j$. Then
\begin{equation}\label{geo-two-lyapunov}
 \tfrac12\{T_j,\mathcal C_j\}=I+S_j^2+K_j>0,\qquad
 \tfrac12\{T_j,L_j\}=I+\eta^2K_j>0,
 \qquad T_j>0,\quad L_j>0.
\end{equation}
The metric-linearized principal symbol is, up to the common curvature
sign, $g^{z\bar z}\zeta_z\bar\zeta_z\{T_j,u_j\}/2$ at each vertex.
Every strictly calibrated solution is therefore positively strongly
elliptic.
\end{lemma}
\begin{proof}
The phase equation gives $\mathcal A_j=\chi\mathcal C_j$ and
$\mathcal B_j=\eta\mathcal C_j$. Substitute these into
\eqref{geo-AB}, retaining both orders of $H_jS_j$; this gives the two
identities in \eqref{geo-two-lyapunov}. Testing the first on an eigenvector
of the Hermitian $T_j$ shows that each eigenvalue is positive because
$\mathcal C_j>0$. The second then gives positivity of $L_j$ by the
inverse Lyapunov formula
\[
 L_j=2\int_0^\infty e^{-tT_j}(I+\eta^2K_j)e^{-tT_j}\,dt.
\]
For a metric variation $h_j(t)(v,w)=h_j(e^{tu_j}v,w)$,
$u_j=u_j^\dagger$, the curvature convention gives
$\dot H_j=-g^{z\bar z}\partial_z\partial_{\bar z}u_j$ plus lower-order
terms. Adjoint and $S_j$ variations are zero order; derivative and
$K_j$ variations are at most first order. Thus the second-order part of
$\chi\mathcal B_j-\eta\mathcal A_j$ is
$\{T_j,\dot H_j\}/2$. If the eigenvalues of $T_j$ are $t_a$, its
matrix-entry multipliers are $(t_a+t_b)/2$. Ellipticity is equivalent
to absence of zero pair sums; positive strong ellipticity is equivalent
to $T_j>0$. The identity at an actual calibrated solution supplies that
positivity. Calibration alone away from the equation need not do so.
This is a symbol statement, rather than a global-kernel or inverse theorem.
\end{proof}

This reduced symbol calculation is specific to the ordered three-block
equation. It does not identify strict matrix calibration with the
formal-derivative \(Z\)-positivity of
\cite[Definition 2.6]{KS2025}.

In particular,
\begin{equation}\label{geo-positive-buffers}
 -I+6P_1>0,\quad5I-6R_1+6P_2>0,\quad11I-6R_2>0,
 \qquad -I+6H_j>0,
\end{equation}
and the nonscalar middle identity is
\[
 \tfrac12\{5I-6R_1+6P_2,-I+6H_1\}=37I+36K_1.
\]
The Chern--Bochner identities with the conventions above
\cite[Ch.~VII, \S1]{Demailly2012} are
\begin{equation}\label{geo-bochner}
\begin{aligned}
 \int_X\operatorname{tr}(H_0P_1-H_1R_1)\omega_X
     &=\int_X|\xi_1|^2\omega_X,\\
 \int_X\operatorname{tr}(H_1P_2-H_2R_2)\omega_X
     &=\int_X|\xi_2|^2\omega_X.
\end{aligned}
\end{equation}
They follow by integration by parts for the holomorphic arrows with
the induced source-target Chern connections. Tracing the top and bottom
equations and using these identities gives
\begin{equation}\label{geo-endpoint-integrals}
\begin{aligned}
 -\frac{18\pi}{\sqrt{37}}
  &=-\int_X\operatorname{tr}(L_1R_1)\omega_X
     -\frac3{\sqrt{37}}\int_X|\xi_1|^2\omega_X,\\
 \frac{6\pi}{\sqrt{37}}
  &=\int_X\operatorname{tr}(L_2R_2)\omega_X
     +\frac3{\sqrt{37}}\int_X|\xi_2|^2\omega_X.
\end{aligned}
\end{equation}
Both are compatible with positive $L_j$. The averaged integrated traces
of $L_j$ are respectively $17/\sqrt{37},22/\sqrt{37},5/\sqrt{37}$,
since $\int\operatorname{tr}H_j=2\pi d_j$ and
$\int\operatorname{tr}I=2\pi r_j$. When $\operatorname{rank}B=2$, the
exact sequence $0\to E_2\xrightarrow{\psi}E_1\to Q_\psi\to0$
defines a quotient line $Q_\psi$ of degree $d_1-d_2=3$.
Equip it with the quotient metric induced by $h_1$, let $D_{Q_\psi}$
be its Chern connection, and put $H_{Q_\psi}=i\Lambda F_{D_{Q_\psi}}$.
Its normalized average is
$ (2\pi)^{-1}\int_X(\chi+\eta H_{Q_\psi})\omega_X=17/\sqrt{37}$.
Independently, when $\operatorname{rank}A=2$, the sequence
$0\to\ker\phi\to E_1\xrightarrow{\phi}E_0\to0$
gives the degree-three upper quotient $Q_\phi=E_1/\ker\phi\simeq E_0$;
its quotient-metric Chern connection has the same normalized average.
These two quotients agree when $\operatorname{im}\psi=\ker\phi$,
but are not identified in general.
When $\operatorname{rank}B=1$, instead
$\operatorname{im}\psi=L(-p)$ with saturation $L\simeq\mathcal O(2)$.
The sheaf $E_1/\operatorname{im}\psi$ has length-one torsion
$L/\operatorname{im}\psi\simeq\mathbb C_p$ and torsion-free quotient
$E_1/L\simeq\mathcal O(2)$, of degree two; its quotient-metric
Chern connection gives normalized average $11/\sqrt{37}$.
These statements concern induced quotient curvature, not the projection
of $L_1$, which retains the second-fundamental-form correction in
\eqref{geo-euler-kinetics} and \eqref{lower-gauss-codazzi}.
The top Chern--Bochner sign is
$\chi d_0-\eta r_0=-9/\sqrt{37}$; the lower normalized endpoint value is
$3/\sqrt{37}$ (integer numerator $12$ in the charge normalization).

For arbitrary metrics there is also the compatible-target identity
\begin{equation}\label{geo-target-compatibility}
 \sum_{j=0}^2\int_X\operatorname{tr}
              (\chi\mathcal B_j-\eta\mathcal A_j)\omega_X=0.
\end{equation}
Indeed the integrated sums of $\operatorname{tr}\mathcal B_j$ and
$\operatorname{tr}\mathcal A_j$ are $24\pi$ and $-4\pi$, respectively,
by Chern--Weil and \eqref{geo-bochner}; the resulting equality is
$24\chi+4\eta=0$. Thus the global linearized target must respect this
codimension-one compatibility. A positive-symbol metric can be chosen
formally by normalizing the upper basis to common nonzero eigenvalue
$P_1,R_1=1/2$ and scaling the bottom metric to make the common nonzero
eigenvalue of $P_2,R_2$ less than $11/6$. The three $T_j$ are then
positive, but this supplies neither a subsolution nor existence.
Equivalently the invariant topological check is
\[
 \int_{\mathcal X}\frac{\operatorname{tr}(\Omega^2)}{2(2\pi)^2}
       =-4+24i=Z(\mathcal E_{A,B}).
\]
No connection-dependent derivative or commutator contribution changes
this Chern character.

\begin{proposition}[Elementary all-metric exclusions]
\label{geo-elementary-exclusions}
There is no smooth strictly calibrated solution if
$\operatorname{rank}A\le1$ or $B=0$, including their intersections.
\end{proposition}
\begin{proof}
When the upper coefficient rank is at most one, its two linear components
have a common zero, including the identically zero case. At this point
$P_1=S_0=0$ and $T_0=\chi I<0$, contradicting
Lemma~\ref{geo-lyapunov}. If $B=0$, then $R_2=K_2=0$ and $S_2=2I$.
The scalar bottom equation is
$-(H_2+2)=6(1-2H_2)$, so $H_2=8/11$ everywhere. Its integral is
$16\pi/11$, whereas Chern--Weil on $E_2=\mathcal O(1)$ gives $2\pi$.
This second contradiction does not require calibration. Both arguments
apply to every smooth positive metric triple.
\end{proof}

\begin{proposition}[Euler-sequence all-metric exclusion]
\label{geo-euler-exclusion}
No original pair in the stratum $(2,2,0)$ admits a smooth strictly
calibrated solution.
\end{proposition}
\begin{proof}
Suppose such a solution exists. Full pointwise rank and zero composition
give the actual holomorphic Euler sequence
\[
 0\longrightarrow I_b:=\operatorname{im}\psi=\ker\phi\simeq\mathcal O(1)
 \longrightarrow E_1\longrightarrow\mathcal Q\simeq\mathcal O(3)
 \longrightarrow0.
\]
Let $P$ be the actual $h_1$-orthogonal projection onto $I_b$ and
$P^\perp=I-P$. For smooth positive functions $x,y$ the rank-one arrow
products are
\begin{equation}\label{geo-euler-products}
 P_1=xI_{E_0},\quad R_1=xP^\perp,\qquad
 P_2=yP,\quad R_2=yI_{E_2}.
\end{equation}
Thus $S_1=(1+y)P+(1-x)P^\perp$. The quotient eigenvalue of $T_1$ is
$(5-6x)/\sqrt{37}$, so its positivity implies
\begin{equation}\label{geo-euler-x}
 0<x<5/6.
\end{equation}
This is the metric-orthogonal splitting, not a holomorphic splitting or
an assumption that $H_1$ preserves $I_b$.

Let $\beta=P^\perp D_1^{1,0}P|_{I_b}$ be the second fundamental form.
The Chern connection in this smooth splitting is
\[
 D_1=\begin{pmatrix}D_{I_b}&-\beta^\dagger\\
                         \beta&D_{\mathcal Q}\end{pmatrix}.
\]
Gauss--Codazzi in the fixed curvature convention
\cite[Ch.~V, \S14, equations (14.6)--(14.7)]{Demailly2012} gives
$P^\perp H_1P^\perp=H_{\mathcal Q}-\beta\beta^\dagger$, and hence
\begin{equation}\label{geo-euler-gauss}
 \int_X\operatorname{tr}(P^\perp H_1)\omega_X
       =6\pi-\int_X|\beta|^2\omega_X.
\end{equation}
Differentiating $\phi P=0$ and $P^\perp\psi=0$ yields
$\xi_1P=-\phi\beta$ and $P^\perp\xi_2=\beta\psi$. By
\eqref{geo-euler-products},
\begin{equation}\label{geo-euler-derivatives}
 |\xi_1P|^2=x|\beta|^2,\qquad
 |P^\perp\xi_2|^2=y|\beta|^2.
\end{equation}

Set $V=2\pi$, $\ell_Q=\operatorname{tr}(P^\perp L_1)$ and
$k_Q=\operatorname{tr}(P^\perp K_1)$. The upper identity in
\eqref{geo-endpoint-integrals}, with $R_1=xP^\perp$, is
\begin{equation}\label{geo-euler-top}
 \frac{9V}{\sqrt{37}}
       =\int_Xx\ell_Q\omega_X+\frac\eta2\int_X|\xi_1|^2\omega_X.
\end{equation}
Project the second identity of \eqref{geo-two-lyapunov} at vertex one
onto $P^\perp$. Although $L_1$ need not preserve the splitting, $T_1$
does; cyclicity of trace gives
\begin{equation}\label{geo-euler-middle}
 [\chi+\eta(1-x)]\ell_Q=1+\eta^2k_Q,\qquad
 x\ell_Q=\tfrac56\ell_Q-\tfrac1\eta-\eta k_Q.
\end{equation}
Both kinetic halves are retained:
\begin{equation}\label{geo-euler-kinetics}
 k_Q=\tfrac12\bigl(|\xi_1P^\perp|^2+|P^\perp\xi_2|^2\bigr),\qquad
 \int_X\ell_Q\omega_X
       =\frac{17V-6\int_X|\beta|^2\omega_X}{\sqrt{37}},
\end{equation}
where the last identity uses \eqref{geo-euler-gauss} and
$L_1=\chi I+\eta H_1$.

Substitute \eqref{geo-euler-middle} and \eqref{geo-euler-kinetics} into
\eqref{geo-euler-top}, splitting
$|\xi_1|^2=|\xi_1P|^2+|\xi_1P^\perp|^2$. The quotient half of the
upper derivative cancels its actual half in $k_Q$. The remaining equality
is
\[
 9V=8V-5\int_X|\beta|^2\omega_X
           +3\int_X|\xi_1P|^2\omega_X
           -3\int_X|P^\perp\xi_2|^2\omega_X.
\]
Using \eqref{geo-euler-derivatives} gives
\begin{equation}\label{geo-euler-contradiction}
 2\pi=\int_X(3x-3y-5)|\beta|^2\omega_X.
\end{equation}
Since $x<5/6$ and $y>0$, the coefficient is strictly less than $-5/2$.
The right side is nonpositive, also when $\beta=0$, while the left side
is positive. This excludes arbitrary metric triples, without a
representative normalization or base automorphism.
\end{proof}

\section{The correspondence on the original coefficient closure}

\begin{theorem}\label{geo-main}
For every $A,B\in\operatorname{Mat}_2(\mathbb C)$, on the fixed original
chain and background of \eqref{geo-chain}, the following are equivalent:
\begin{align*}
 &\mathrm{Ex}(A,B),\\
 &\mathcal E_{A,B}\text{ is }SU(2)\text{-equivariantly }Z\text{-stable},\\
 &\mathcal E_{A,B}\text{ is }SU(2)\text{-equivariantly }Z\text{-semistable},\\
 &\mathcal E_{A,B}\text{ is }SU(2)\text{-equivariantly }Z\text{-polystable},\\
 &\operatorname{rank}A=2\quad\text{and}\quad Q\not\equiv0.
\end{align*}
Thus all three stability predicates and smooth strictly calibrated
existence hold on $(2,1,1)$ and $(2,2,1)$, and fail on each of the nine
other nonempty triples of Lemma~\ref{geo-domain}. In particular there
are no strictly semistable objects in this fixed family.
\end{theorem}
\begin{proof}
The algebraic result, Proposition~\ref{geo-algebraic-classification}, was
proved solely from the coherent equivalence, the compact-to-complex
bridge, charge, full subobject inventory and phase inequalities.
Since $Q\not\equiv0$ implies $B\ne0$, it identifies all three predicates
with $\operatorname{rank}A=2$, $Q\not\equiv0$.

We establish the analytic criterion on the same original inputs. If
$\operatorname{rank}A=2$, $Q\not\equiv0$, then $B$ has rank one or two.
In rank two this is exactly the full primary family; the original-chain
existence theorem \ref{primary-existence} applies, including repeated
composition roots. In rank one, Lemma~\ref{geo-domain} factors the actual
map as $v\sigma$ with saturation $\mathcal O(2)v$, its lower zero $p$,
and upper restriction divisor $q_0$. Both $p=q_0$ and $p\ne q_0$ belong
to the original-chain existence theorem \ref{lower-existence}. Hence a
smooth strictly calibrated metric triple exists for every pair in
either stratum.

Conversely, if such a metric triple exists,
Proposition~\ref{geo-elementary-exclusions} forces
$\operatorname{rank}A=2$ and $B\ne0$. If $Q$ vanished identically,
rank-one $B$ would be impossible by its actual factorization in
Lemma~\ref{geo-domain}. Rank-two $B$ would give $(2,2,0)$, excluded by
Proposition~\ref{geo-euler-exclusion}. Thus $Q\not\equiv0$.
This proves the analytic equivalence independently of stability; joining
the two same-input equivalences proves the theorem and every boundary
row. The eight low-rank or zero-lower triples are excluded by the first
proposition and the remaining composition-zero triple by the second.
\end{proof}

In the table below, the shared status applies to stability,
semistability, polystability and smooth strictly calibrated existence,
which are equivalent by Theorem~\ref{geo-main}. Each row quantifies over
all coefficient pairs in its stratum. The analytic column names the
existence or exclusion result; Proposition~\ref{geo-algebraic-classification}
supplies the algebraic classification in every row. The digits have no
relation to the subobject rank patterns in \eqref{geo-all-patterns}.
\begin{equation}\label{geo-outcomes}
\begin{array}{c|c|c}
 (a,b,\epsilon)&\text{shared status}&\text{analytic result}\\\hline
 (0,0,0)&\text{no}&\text{Proposition \ref{geo-elementary-exclusions}}\\
 (0,1,0)&\text{no}&\text{Proposition \ref{geo-elementary-exclusions}}\\
 (0,2,0)&\text{no}&\text{Proposition \ref{geo-elementary-exclusions}}\\
 (1,0,0)&\text{no}&\text{Proposition \ref{geo-elementary-exclusions}}\\
 (1,1,0)&\text{no}&\text{Proposition \ref{geo-elementary-exclusions}}\\
 (1,1,1)&\text{no}&\text{Proposition \ref{geo-elementary-exclusions}}\\
 (1,2,1)&\text{no}&\text{Proposition \ref{geo-elementary-exclusions}}\\
 (2,0,0)&\text{no}&\text{Proposition \ref{geo-elementary-exclusions}}\\
 (2,1,1)&\text{yes}&\text{Theorem \ref{lower-existence}}\\
 (2,2,0)&\text{no}&\text{Proposition \ref{geo-euler-exclusion}}\\
 (2,2,1)&\text{yes}&\text{Theorem \ref{primary-existence}}
\end{array}
\end{equation}

Each existence proof treats all six metric fields for one fixed chain
and one finite auxiliary mass. The value bounds proved in the last two
sections allow the common endpoint theorem to produce a smooth strictly
calibrated metric; punctured convergence inside an escape contradiction
does not suffice. These analytic proofs do not assume stability, and
the coherent subobject classification does not assume existence.
The correspondence concerns the prescribed invariant construction and
equivariant category, not ordinary coherent or derived stability or
arbitrary non-invariant product metrics. Its estimates are fixed-chain
estimates: it asserts neither varying-chain uniformity, a common metric
for different chains, uniqueness nor a moduli theorem.

\section{Analytic construction and compactness}
\label{common-section}

The auxiliary parameter supplies a starting solution and permits continuation
on every closed interval of positive parameters. Estimates on such an
interval need not remain bounded at zero. The construction below therefore
has two distinct uses: positive-stage estimates justify continuation,
whereas estimates conditioned on a uniform metric-value bound justify
the strong endpoint limit. The missing value bounds are proved separately
in Theorems~\ref{primary-value-bound} and~\ref{lower-value-bound}, for the
primary family and the rank-one lower boundary respectively.

\subsection{Reference metrics and the six metric variables}

We use two fixed-chain settings separately.  In the primary setting both
component pairs of \(\phi\) and \(\psi\) are bases and
\(\phi\psi\not\equiv0\).  In the lower-boundary setting the upper pair is a
basis and \(\psi=v\sigma\), where \(v\ne0\) is constant and \(\sigma\) is a
nonzero linear section.  Its zero \(p\) and the zero \(q_0\) of
\(\phi|_{\mathcal O(2)v}\) may coincide.  All constants below are for one
fixed chain, one actual background form, and one fixed reference; they need
not be uniform as the chain changes.

Write \(\omega_X=ig\,dz\wedge d\bar z\) and
\(\Delta=g^{-1}\partial_z\partial_{\bar z}\).  Fix the actual line metric
\(\ell\) with \(iF_\ell=\omega_X\), and a fixed positive determinant-one
constant matrix \(G\).  Set
\[
 \alpha=\frac16,\qquad c=\frac56,\qquad d=\frac{11}{6},\qquad
 A_0=\frac{37}{36},\qquad p_* =\frac9{14},\qquad W=\frac{10}{3}.
\]
For an actual metric, the residual endomorphisms are
\[
 F_j(h)=\tfrac12\{S_j-\alpha I,H_j-\alpha I\}-A_0I-K_j.
\]
Expanding this anticommutator and using the fixed phase in
\eqref{geo-phase} shows that \(F_j(h)=0\) is exactly
\eqref{geo-equation}'s phase equation.  Here \(F_j\) is an endomorphism
residual, not the raw curvature two-form \(F_{h_j}\).
In holomorphic frames, \(\ell=L\), and the upper row is denoted by
\(\phi\).  The reference metrics are
\begin{equation}\label{common-reference}
 h_{*,1}=L^2G,\qquad
 h_{*,0}=\frac{p_*L^2}{\phi G^{-1}\phi^*},\qquad h_{*,2}=s\ell.
\end{equation}
The quotient form in the denominator is everywhere positive because the
upper pair is a basis.  These formulas define metrics on the original
bundles; varied metrics will not be required to be quotient metrics.

Before choosing \(s\), compute the actual lower norm and kinetic term with
domain metric \(\ell\):
\[
 \rho_{\rm ref}=|\psi|^2,\qquad k_{\rm ref}=\tfrac12|D\psi|^2,
 \qquad j=\tfrac56\rho_{\rm ref}+k_{\rm ref},\qquad
 s=\max\{1,4\max_X j\}.
\]
The function \(\rho_{\rm ref}\) is positive in the primary setting and
vanishes at \(p\) in the lower-boundary setting.  In both settings it is
smooth, nonnegative and not identically zero.  Constant domain scaling
leaves the lower Chern connection unchanged and divides its actual squared
norms by \(s\).  Thus
\[
 H_{*,1}=2I,\quad H_{*,2}=1,\quad
 r_* =\rho_{\rm ref}/s\le\frac3{10},\quad K_{*,2}=k_{\rm ref}/s.
\]
The induced middle End connection is flat.  With
\(N_\phi=\phi G^{-1}\phi^*\), the quotient calculation gives
\[
 H_{*,0}=2+n_\phi,\quad
 n_\phi=g^{-1}\partial_z\partial_{\bar z}\log N_\phi\ge0,
 \quad K_{*,0}=\frac{p_*}{2}n_\phi.
\]
Indeed the reference derivative is
\(\xi_* =\phi_z-(N_\phi)_z\phi/N_\phi\); its squared norm gives the last
identity, and Cauchy--Schwarz gives \(n_\phi\ge0\).

Use the notation \(R_* =\phi^\dagger\phi\) and
\(P_* =\psi\psi^\dagger\) in reference-unitary frames.  Then
\[
 q_* =p_* -\alpha=\frac{10}{21},\qquad
 \beta_* =d-r_*\ge\frac{23}{15},\qquad
 V_* =cI-R_*+P_*\ge(c-p_*)I=\frac4{21}I.
\]
Only \(R_*\le p_*I\) and \(P_*\ge0\) were used; their commutator may be
nonzero.  The reference residual, in the convention of the reduced
equation, is
\begin{equation}\label{common-reference-force}
 \begin{aligned}
 F_{*,0}&=\frac{13}{84}(n_\phi-1),\\
 F_{*,1}&=dV_* -A_0I-\frac{\xi_*^*\xi_*+y_*y_*^*}{2g},\\
 F_{*,2}&=\frac12-\frac{j}{s}\ge\frac14,
 \end{aligned}
\end{equation}
where \(y_*\) is the actual reference derivative of \(\psi\).  In
particular, for \(0\le\tau\le1\),
\begin{equation}\label{common-force-bounds}
 C_0^\tau:=A_0+\tau F_{*,0}\ge\frac{55}{63}=:c_0,
 \qquad C_2^\tau:=A_0+\tau F_{*,2}\ge A_0.
\end{equation}
The middle force is recomputed at this reference, including its actual
lower derivative; its sign is not assumed on the full parameter interval.

\begin{definition}[Metric variables and auxiliary equation]
\label{common-fields}
After the fixed \(G\)-conjugation, use the real bundle
\(\mathscr H=\mathbb R\oplus\operatorname{Herm}_2\oplus\mathbb R\), with
inner product
\(aa'+\operatorname{tr}(MM')+bb'\).  Every positive metric triple has
unique reference-selfadjoint logarithmic variables \(w=(a,M,b)\):
\[
 h(w)=h_*e^w,\qquad f=e^M,\quad S=e^{M/2}.
\]
Square-root transport defines
\(\widehat F_j(w)=e^{w_j/2}F_j(h(w))e^{-w_j/2}\).  The transported
arrow products and symbols are
\begin{equation}\label{common-symbols}
 \begin{gathered}
 R=e^aS^{-1}R_*S^{-1},\qquad P=e^{-b}SP_*S,
 \qquad p=\operatorname{tr}R,\quad r=\operatorname{tr}P,\\
 q=p-\alpha,\qquad \beta=d-r,\qquad V=cI-R+P,\qquad
 \mathcal L_VH=\tfrac12(VH+HV).
 \end{gathered}
\end{equation}
The legal domain is \(q>0,\ \beta>0,\ V>0\).  Put
\[
 B(w)w=(qa,\mathcal L_VM,\beta b),\qquad
 v(w)=qa+\operatorname{tr}(VM)+\beta b.
\]
For one fixed prospective mass \(\mu_s>0\), the auxiliary equation is
\begin{equation}\label{common-auxiliary}
 \widehat F(w)+\tau\mu_s B(w)w=\tau F_*,
 \qquad 0<\tau\le1.
\end{equation}
Write its auxiliary residual as
\[
 \mathcal E(\tau,w)=\widehat F(w)+\tau\mu_sB(w)w-\tau F_*.
\]
Thus the auxiliary equation is exactly \(\mathcal E(\tau,w)=0\).
With \(t=\tau\mu_s\), its full rows are
\begin{equation}\label{common-positive-rows}
 \begin{aligned}
 H_0&=\alpha-ta+(C_0^\tau+K_0)/q,\\
 \widehat H_1&=\alpha I-tM+\mathcal L_V^{-1}(C_1^\tau+K_1),\\
 H_2&=\alpha-tb+(C_2^\tau+K_2)/\beta,
 \end{aligned}\qquad C_1^\tau=A_0I+\tau F_{*,1}.
\end{equation}
All variables are unrestricted real scalars or full Hermitian matrices.
\end{definition}

The notation \(\mu_s\) anticipates the single mass chosen in
Lemma~\ref{common-anchor}.  Until that choice, estimates stated with
\(\mu>0\) apply to any finite mass in the same auxiliary equation, with
\(t=\tau\mu\).  After the choice, continuation and endpoint arguments
use this one fixed value \(\mu=\mu_s\).

The ordered inverse used here is
\begin{equation}\label{common-lyapunov-inverse}
 \mathcal L_V^{-1}N=2\int_0^\infty e^{-uV}Ne^{-uV}\,du.
\end{equation}
For each positive \(V\) it preserves positivity and is selfadjoint for the
trace pairing.  Its trace is \(\operatorname{tr}(V^{-1}N)\).  This formula
is not multiplication by \(V^{-1}\).

\begin{lemma}[The single constant normalization]
\label{common-normalization}
The total residual integral vanishes for every metric triple.  Every exact
positive-parameter solution satisfies \(N(w):=\int_Xv(w)\omega_X=0\).
There is precisely one common constant to remove in the reference chart.
\end{lemma}
\begin{proof}
The actual arrow Bochner identities are
\[
 \Delta p=2K_0-pH_0+\operatorname{tr}(R\widehat H_1),\qquad
 \Delta r=2K_2-\operatorname{tr}(P\widehat H_1)+rH_2,
 \qquad\operatorname{tr}K_1=K_0+K_2.
\]
Substitution into the residual gives
\begin{equation}\label{common-total-trace}
 \operatorname{tr}_{\rm tot}\widehat F
 =-\alpha H_0+c\operatorname{tr}\widehat H_1+dH_2
   -\Delta(p+r)-\frac{14}{3}.
\end{equation}
Here \(q+\operatorname{tr}V+\beta=W\), and
\(\alpha W+4A_0=14/3\).  The degrees \((3,4,1)\) imply that its integral
is \((-3\alpha+4c+d-14/3)\operatorname{Vol}(X)=0\).
Transport preserves trace.  Integrating \eqref{common-auxiliary} gives
\(tN(w)=0\), hence the assertion for \(t>0\).

For \(T(k)=(k,kI,k)\), common constant translation leaves the arrows,
curvature, residual and symbols unchanged and
\(N(w+T(k))=N(w)+kW\operatorname{Vol}(X)\).  Let
\[
 A_*U=(q_*a,\mathcal L_{V_*}M,\beta_*b),\quad
 L_0(U)=\int_X\operatorname{tr}_{\rm tot}(A_*U)\omega_X,
 \quad X'=\ker L_0,
\]
and let \(Y\) be the space of sources of zero total integral.  Then
\[
 \mathcal P(u)=u-T\!\left(\frac{N(u)}{W\operatorname{Vol}(X)}\right)
\]
maps \(X'\) bijectively to \(N=0\), with inverse
\(w\mapsto w-T(L_0(w)/(W\operatorname{Vol}(X)))\).  Moreover
\(D\mathcal P(0)|_{X'}=I\), since \(DN(0)=L_0\).  These are closed
codimension-one spaces.  No nonconstant scalar, relative scalar, or matrix
direction is removed, and no separate vertex mean is imposed.
\end{proof}

\subsection{Estimates which precede continuation}

\begin{lemma}[The radial trace comparison]\label{common-trace-comparison}
For every finite field, not merely for solutions, put
\[
 T=-\alpha a+c\operatorname{tr}M+db+p+r-p_*-r_*(x),\qquad
 v=qa+\operatorname{tr}(VM)+\beta b.
\]
Then
\begin{equation}\label{common-radial-inequalities}
 T\ge q_*a+\operatorname{tr}(V_*M)+\beta_*b,
 \qquad v\ge T.
\end{equation}
Every auxiliary solution satisfies the exact scalar equation
\begin{equation}\label{common-trace-equation}
 \Delta T=(1-\tau)f_*+tv,\qquad
 f_*=\operatorname{tr}_{\rm tot}F_* ,\quad t=\tau\mu.
\end{equation}
These statements allow semidefinite \(P_*\), including \(P_*(x)=0\).
\end{lemma}
\begin{proof}
Cyclic trace, without commuting the two arrows, gives along the radial field
\[
 p(sw)=\operatorname{tr}(R_*e^{s(aI-M)}),\qquad
 r(sw)=\operatorname{tr}(P_*e^{s(M-bI)}).
\]
The spectral inequality \(e^A\ge I+A\), paired with the fixed positive
semidefinite weights, proves the first inequality.  The scalar function
\(k(y)=ye^y-e^y+1\) is nonnegative: its derivative is \(ye^y\) and its
unique minimum is zero at the origin.  Spectral calculus therefore gives
\[
 v-T=\operatorname{tr}(R_*k(aI-M))+
       \operatorname{tr}(P_*k(M-bI))\ge0.
\]
Neither calculation assumes that intermediate radial fields have positive
symbols.  Substitution of the full trace identity from
Lemma~\ref{common-normalization}, with
\(\operatorname{tr}\widehat H_1=4-\Delta\operatorname{tr}M\), proves
\eqref{common-trace-equation}.  The variable baseline \(r_*(x)\) is kept;
its derivatives are part of \(f_*\).
\end{proof}

\begin{lemma}[Independent positive-parameter size estimate]\label{common-positive-size}
For every actual positive-symbol solution with \(0<\tau\le1\) and
\(C_1^\tau\ge0\),
\begin{equation}\label{common-size-estimate}
 t\max(\|a\|_\infty,\|b\|_\infty,\|M\|_{\infty,\mathrm{op}})
 \le C_{\rm size}.
\end{equation}
The constant depends only on the fixed reference and background, not on
\(\mu,\tau\), a symbol minimum, or the existence of any other solution.
For the primary reference this is the primary estimate; for the boundary
reference it is a separate estimate with recomputed fixed-data constants.
\end{lemma}
\begin{proof}
Set \(C_L=\sup(H_{*0}-\alpha)=d+\max n_\phi\) and
\(F_T=\|f_*\|_\infty\).  At a minimum of \(a\),
\(H_0\le H_{*0}\); the positive top source in
\eqref{common-positive-rows} gives \(a\ge-C_L/t\).  At a minimum of
\(b\), the bottom row similarly gives \(b\ge-c/t\).
Let \(m=\min_X\lambda_{\min}M\), and take a constant reference-unit
vector \(e\) at an attaining eigenline.  The globally defined function
\(e^*f^{-1}e\) has its maximum there, also for repeated spectra.  Direct
ordered differentiation yields at that point
\[
 \Delta(e^*f^{-1}e)
 =e^{-m}(e^*\widehat H_1e-2)
 +g^{-1}e^{-2m}e^*f_zf^{-1}f_{\bar z}e.
\]
The last term is nonnegative, so \(\widehat H_{1,ee}\le2\).
The positive Lyapunov inverse in \eqref{common-lyapunov-inverse} gives
\(\widehat H_1\ge\alpha I-tM\), hence \(m\ge-d/t\).
This argument needs semipositive middle force, not a positive lower-arrow
norm or a uniform inverse bound.

At a maximum of \(T\), \eqref{common-trace-equation} and \(v\ge T\)
give \(T\le F_T/t\).  The shifted fields
\(a+C_L/t,b+C_L/t,M+(C_L/t)I\) are nonnegative.  The reference weights
are at least \(\delta=4/21\), and their sum is \(W\).  Thus
\[
 \delta\bigl(a+C_L/t+\operatorname{tr}(M+(C_L/t)I)+b+C_L/t\bigr)
 \le (F_T+C_LW)/t.
\]
For example \(C_{\rm size}=\max(C_L,(F_T+C_LW)/\delta)\) suffices.
Only semidefiniteness of \(P_*\) was used, so the boundary proof applies
at its zero as well.  No normalization premise was needed: an exact
positive-parameter solution is normalized by its total integral.
\end{proof}

\begin{lemma}[Positive-stage integral estimate]\label{common-positive-energy}
In the primary setting, fix \(\tau_->0\), a finite \(\mu\), and a closed
interval \([\tau_-,T_+]\) with \(C_1^\tau\ge c_1I>0\).
There are fixed-stage constants \(B,E_B\) for all solutions such that
\[
 |a|,|b|,\|M\|_{\mathrm{op}}\le B,\qquad
 \int_X(|DM|^2+K_0+K_2+\operatorname{tr}V^{-1})\,\omega_X\le E_B.
\]
These constants are not asserted uniform as \(\tau_-\downarrow0\).
\end{lemma}
\begin{proof}
The preceding lemma already gives the required \(B\).  The independent
entropy route behind the same stage estimate can also be recorded without
any symbol buffer.  Put \(C=C_L/(\tau_-\mu)\),
\(\lambda=\lambda_{\max}M\), and \(E=\log(112/27)\).
The top norm gives \(a\le\lambda+E\).  If \(m_i\) are the eigenvalues
of \(M\), and \(c_i\) the spectral weights of its reference top unit
vector, let \(d_i=c_ie^{-m_i}/\sum c_je^{-m_j}\), omitting zero
supports.  The relative entropy \(D_- =\sum d_i\log(d_i/c_i)\ge0\)
follows from \(\log x\le x-1\), and expansion gives
\[
 a-\operatorname{tr}(RM)/p=\log(p/p_*)+D_-.
\]
Together with the lower bounds, \(r+\beta=d\), and
\(p\log(p/p_*)\ge-p_*/e\), this implies
\[
 v\ge \tfrac23\lambda_++\beta b_+-K,
 \qquad K=CW+\alpha E+p_*/e.
\]
Normalization gives \(\int\lambda_+\le(3/2)K\operatorname{Vol}X\)
and \(\int\beta b_+\le K\operatorname{Vol}X\).
If \(\beta<d/2\), then
\(r\le r_{*,\max}e^{\lambda-b}\) and \(r>d/2\) imply \(b<\lambda\);
otherwise \(b_+\le(2/d)\beta b_+\).  These facts and
\(\operatorname{tr}M_+\le2\lambda_+\) bound the unweighted
\(L^1\) norm of all six fields.  The trace equation then has
\(\Delta T\ge-D_0\) and \(\int|T|\le A\), with fixed-stage constants.
The actual chordal Green identity \eqref{common-green-kernel} gives
\(T\le m_\omega A/(2\pi)+n_\omega D_0\), where
\(m_\omega=\sup\omega_{\rm FS}/\omega_X\) and
\(n_\omega=\sup\omega_X/\omega_{\rm FS}\).
Finally
\(-\alpha a+c\operatorname{tr}M+db
 \ge(2/3)\lambda-(c+d)C-\alpha E\)
recovers all upper values.  This argument uses no weighted trace with an
assumed \(V\)-gap.

For the integral assertion differentiate \(h=f^{-1}\) in its actual
order.  The middle row is exactly
\begin{equation}\label{common-inverse-trace-row}
 \Delta h=g^{-1}hf_zhf_{\bar z}h-dh-tMh
       +S^{-1}\mathcal L_V^{-1}(C_1^\tau+K_1)S^{-1}.
\end{equation}
On the proved value range the first trace term controls \(|DM|^2\):
the divided differences of \(D\exp_M\) lie between \(e^{-B}\) and
\(e^B\), including repeated spectra.  Moreover \(V\le(c+d)I\), and
trace self-adjointness gives
\[
 \operatorname{tr}(S^{-1}\mathcal L_V^{-1}(K_1)S^{-1})
 \ge \frac{e^{-B}}{c+d}(K_0+K_2),\qquad
 \operatorname{tr}(S^{-1}\mathcal L_V^{-1}(C_1^\tau)S^{-1})
 \ge c_1e^{-B}\operatorname{tr}V^{-1}.
\]
The negative terms are bounded below by \(-2e^B(d+\mu B)\).
Thus, for a fixed \(a'_B>0\),
\(\Delta\operatorname{tr}h\ge a'_B(|DM|^2+K_0+K_2+
\operatorname{tr}V^{-1})-2e^B(d+\mu B)\).
Integration proves the claim.  This is an integral assertion only; the
pointwise gradients and gaps require the following full face argument.
\end{proof}

For later use the actual derivative formulas, with
\(\Theta=f^{-1}f_z\), are
\begin{equation}\label{common-ordered-jets}
 \begin{gathered}
 X=e^{a/2}(\xi_*+a_z\phi-\phi\Theta)S^{-1},\qquad
 Y=e^{-b/2}S(y_*+\Theta\psi-b_z\psi),\\
 K_0=XX^*/(2g),\qquad K_1=(X^*X+YY^*)/(2g),\qquad K_2=Y^*Y/(2g),\\
 H_0=H_{*,0}-\Delta a,\quad H_2=1-\Delta b,\quad
 \widehat H_1=2I-g^{-1}S\partial_{\bar z}(f^{-1}f_z)S^{-1}.
 \end{gathered}
\end{equation}
At a lower-arrow zero \(y_*\) is still retained.  None of these
coefficients contains \(1/r_*,1/\psi\), or a division by the composition.

\subsection{The reference inverse and strong damping}

The local kernel argument starts from the classical Green--Riesz and
Poisson representations \cite[Ch.~I, \S4.B]{Demailly2012}.
We give the estimates used by this system explicitly below.

\begin{lemma}[Local potential estimates]
\label{common-local-potentials}
The component Poisson and complex \(\bar\partial\) estimates used below
hold on nested coordinate disks.  A bounded two-sided Poisson source and
bounded values give compactness in \(C^1\), with gradient modulus
\(Ch(1+|\log h|)\).  A bounded complex function with bounded
\(\bar\partial\) source is locally precompact in \(C^\nu\) for every
\(\nu<1\).  A \(C^\gamma\) Poisson source gives a \(C^{2,\gamma}\)
solution.  A nonnegative entire distributionally superharmonic function
on the plane is constant.
\end{lemma}
\begin{proof}
For the Euclidean equation use
\(\Gamma(x)=-(2\pi)^{-1}\log|x|\) and a cutoff source; subtracting its
logarithmic potential leaves a distributionally harmonic function on a
smaller disk.  Mollification and radial mean values show this remainder
is smooth, and the circle formula bounds its interior derivatives by its
local value norm.  The gradient kernel is \(O(|x|^{-1})\).  Splitting the
difference of its integrals into disks of radius \(2h\) and their
complement gives \(Ch+Ch\int_{2h}^{R}dr/r\).  The same calculation with
the Cauchy kernel \(1/z\), and the holomorphic remainder's circle formula,
gives the stated \(\bar\partial\) compactness.

The Hessian kernel is
\[
 H_{ij}(x)=-\frac1{2\pi}
 \left(\frac{\delta_{ij}}{|x|^2}-\frac{2x_ix_j}{|x|^4}\right),
 \qquad
 \partial_i\partial_j\Gamma
 =\operatorname{pv}H_{ij}-\tfrac12\delta_{ij}\delta_0.
\]
Its circular mean is zero.  Subtract the source value in the singular
integral.  The near integrand is bounded by
\(C[F]_\gamma |x|^{\gamma-2}\).  For two centers distance \(h\) apart,
the near contribution is \(O(h^\gamma)\); outside \(2h\), the kernel
difference is \(Ch|x|^{-3}\), giving
\(Ch[F]_\gamma\int_{2h}^Rr^{\gamma-2}dr=O(h^\gamma)\).
The constant-source kernel integral is uniformly bounded by circular
cancellation: a shifted excluded ball changes only an \(O(h)\)-annulus.
Cutoff terms are smooth.  This proves the Hessian estimate without a
logarithmic loss.

For the last assertion convolve the nonnegative superharmonic function
with a nonnegative smooth compactly supported probability kernel.  A
positive mass of its negative Laplacian in one disk makes its circular
mean derivative at most \(-m/(2\pi R)\) on every larger circle.
Integration in \(\log R\) contradicts nonnegativity.  Thus every
convolution is harmonic.  The differentiated Poisson formula bounds the
gradient of a nonnegative entire harmonic function at a center by its
value times \(C/R\).  Letting \(R\to\infty\) makes it constant, and
approximate identities give the distributional assertion.
\end{proof}

\begin{lemma}[Full nonsymmetric reference inverse]
\label{common-reference-inverse}
For each of the two fixed-chain settings above and \(0<\gamma<1\), the
operator \(J_*+\mu A_*\), where \(J_*=D\widehat F(0)\), is an
isomorphism \(C^{2,\gamma}(X;\mathscr H)\to C^{0,\gamma}(X;\mathscr H)\)
for every sufficiently large finite \(\mu\).  It restricts to an
isomorphism \(X'\to Y\).
\end{lemma}
\begin{proof}
The principal part is coercive, but the full linearization also has
coupled first-order terms and need not be symmetric. We keep those terms
in the calculation: a sufficiently large finite mass absorbs them in the
energy estimate, after which local elliptic estimates give the stated
inverse. The normalization is handled by restricting this unrestricted
inverse, rather than by dropping a scalar equation in advance.

The full ordered derivative is recorded explicitly to fix its lower-order
terms.  For \(U=(a,M,b)\), put \(Z=M-aI\), \(Y_0=M-bI\),
\(c_\phi=\phi Z\phi^*\), \(c_\psi=\psi^*Y_0\psi\).  All quantities in
this paragraph are at the fixed reference.  Square-root transport gives
\(J_*U=\dot F[U]+[U,F_*]/2\), including the off-shell commutator.
The changing-adjoint variations are
\[
 \dot q=-c_\phi,\quad\dot\beta=-c_\psi,\quad
 \dot V=\tfrac12(\{Z,R_*\}+\{Y_0,P_*\}),\quad
 \widehat{\dot\xi}=-\phi Z_z-\tfrac12\xi_*Z,\quad
 \widehat{\dot y}=Y_{0,z}\psi+\tfrac12Y_0y_*.
\]
Consequently
\begin{align*}
 \kappa_0&=-g^{-1}\operatorname{Re}(\phi Z_z\xi_*^*)
             -(2g)^{-1}\xi_*Z\xi_*^*,\\
 \kappa_2&=g^{-1}\operatorname{Re}(y_*^*Y_{0,z}\psi)
             +(2g)^{-1}y_*^*Y_0y_*,\\
 \mathcal K_1[U]&=-(2g)^{-1}
 (Z_{\bar z}\phi^*\xi_*+\xi_*^*\phi Z_z)
 -(4g)^{-1}\{Z,\xi_*^*\xi_*\}\\
 &\hspace{1em}+(2g)^{-1}
 (Y_{0,z}\psi y_*^*+y_*\psi^*Y_{0,\bar z})
 +(4g)^{-1}\{Y_0,y_*y_*^*\}.
\end{align*}
The curvature variations are \(-\Delta a,-\Delta M,-\Delta b\), because
the middle End connection is flat and its curvature is central.  The
resulting rows are
\begin{equation}\label{common-reference-linearization}
 \begin{aligned}
 (J_*U)_0&=-q_*\Delta a-(d+n_\phi)c_\phi-\kappa_0,\\
 (J_*U)_1&=-\mathcal L_{V_*}\Delta M
       +\tfrac d2(\{Z,R_*\}+\{Y_0,P_*\})-\mathcal K_1[U],\\
 (J_*U)_2&=-\beta_*\Delta b-c\,c_\psi-\kappa_2.
 \end{aligned}
\end{equation}
Thus \(J_*=-A_*\Delta+B_*D+C_*\), with smooth real coefficient
operators in all six directions.  In particular this is
\(-A_*\Delta\), not \(-\Delta(A_*\,\cdot)\).  The variable \(r_*(x)\)
contributes to \(DA_*\).  At a lower-arrow zero the terms involving
\(D\psi\) remain smooth; no norm reciprocal was used in this calculation.

Normalize the flat gradient by
\(-\int\langle U,\Delta U\rangle=\int|DU|^2\), and put
\[
 \delta=\frac4{21},\quad K=\|DA_*\|_\infty+\|B_*\|_\infty,
 \quad G_0=\|C_*\|_\infty+K^2/(2\delta),\quad
 \mu_0=(G_0+1)/\delta.
\]
Since \(\operatorname{tr}(V_*M^2)\ge\delta\operatorname{tr}(M^2)\),
\(A_*\ge\delta I\).  Integration by parts, including the derivative
of \(A_*\), and Young's inequality give for \(\mu\ge\mu_0\)
\begin{equation}\label{common-coercivity}
 \int_X\langle U,(J_*+\mu A_*)U\rangle\omega_X
 \ge\frac\delta2\|DU\|_2^2+\|U\|_2^2.
\end{equation}
No favorable sign is assigned to the coupled lower-order terms.

For completeness this also proves surjectivity, rather than only a
kernel assertion.  On the real Hilbert space \(H^1\) use the bounded form
\[
 b(U,Z)=\int\langle D(A_*Z),DU\rangle
       +\int\langle Z,B_*DU+C_*U+\mu A_*U\rangle.
\]
Its diagonal is bounded below by
\(c_H\|U\|_{H^1}^2\), \(c_H=\min(\delta/2,1)\).
The representing Hilbert operator is therefore bounded below and has
closed range.  A vector orthogonal to its range obeys \(b(U,Z)=0\) for
all \(U\); setting \(U=Z\) makes it zero.  Its range is also dense, hence
surjective.  Every \(L^2\) source has a unique weak \(H^1\) solution,
with norm at most \(c_H^{-1}\) times the source norm.  Neither this form
nor \(J_*\) is assumed symmetric.

Multiplication after the operator by the smooth \(A_*^{-1}\) gives scalar
principal part \(-\Delta\) with smooth coupled lower-order terms.
Initially the local Euclidean Poisson source is \(L^2\).  A cutoff
logarithmic potential has bounded values, and its gradient kernel is
locally \(L^s\) for \(s<2\).  For any finite \(p>2\), use
\(s=2p/(p+2)\) in convolution to obtain local \(W^{1,p}\).
The now \(L^p\) source gives \(C^{1,\nu}\),
\(\nu=1-2/p\), by the near/far gradient split.  Choose
\(p=4/(1-\gamma)\), so \(\nu=(1+\gamma)/2>\gamma\).  The source becomes
\(C^\gamma\), and the Hessian calculation of
Lemma~\ref{common-local-potentials} gives \(C^{2,\gamma}\).
A finite nested cover, including infinity, gives a norm bound by the
\(C^\gamma\) source and the initial \(H^1\) bound.  This proves the
unrestricted bounded inverse.

Differentiating the universal total integral gives
\(\int\operatorname{tr}_{\rm tot}J_*U=0\).  Hence
\[
 \int\operatorname{tr}_{\rm tot}(J_*+\mu A_*)U=\mu L_0(U).
\]
The full inverse sends \(Y\) into \(X'\), proving exactly the normalized
restriction without adding vertex conditions or an unproved cokernel
assumption.
\end{proof}

\begin{lemma}[A mass-uniform inverse]
\label{common-mass-inverse}
For each fixed reference and \(0<\gamma<1\) there are finite
\(t_0,C_I\), independent of \(t\ge t_0\), for which
\(L_t=J_*+tA_*\) is an unrestricted isomorphism and
\begin{equation}\label{common-weighted-inverse}
 N_t(U):=t\|U\|_\gamma+\sqrt t\|DU\|_\gamma+\|D^2U\|_\gamma
 \le C_I\|L_tU\|_\gamma.
\end{equation}
Here and below \(\|\cdot\|_\gamma=\|\cdot\|_{C^{0,\gamma}}\) in a fixed
finite precompact atlas and the fixed flat real fiber frame.
\end{lemma}
\begin{proof}
Multiply after \(L_t\) by \(A_*^{-1}\).  The resulting operator is
\(-\Delta+\widetilde B D+\widetilde C+tI\), with smooth full coupled
lower coefficients.  In a holomorphic chart its scalar principal
coefficient is \(a(x)=1/(4g(x))\).  The fixed atlas gives a positive
compact range for \(a\); its coordinate weighted norms and \(N_t\) are
uniformly equivalent for \(t\ge1\).

For constant \(a>0\) in that range use the explicit Gaussian resolvent
\[
 H_{au}(x)=(4\pi au)^{-1}e^{-|x|^2/(4au)},\qquad
 R_t^aF=\int_0^\infty e^{-tu}H_{au}*F\,du.
\]
Differentiation and \(x=\sqrt u\,z\) give
\(\|D^kH_{au}\|_1\le C_ku^{-k/2}\) and, for \(k\ge1\),
\(\int|D^kH_{au}(x)||x|^\gamma dx\le C u^{(\gamma-k)/2}\).
Derivative kernels have integral zero.  Consequently
\[
 t\|R_t^aF\|_\gamma+\sqrt t\|DR_t^aF\|_\gamma
   +\|D^2R_t^aF\|_\gamma\le C\|F\|_\gamma,
 \qquad \|D^2R_t^aF\|_\infty\le C t^{-\gamma/2}[F]_\gamma.
\]
For the Hessian seminorm use cancellation for both two and three Gaussian
derivatives.  The difference at distance \(h\) is bounded by
\(C[F]_\gamma\min(u^{\gamma/2-1},h u^{(\gamma-3)/2})\).
Integrating separately over \((0,h^2)\) and \((h^2,\infty)\) gives
\(C[F]_\gamma h^\gamma\); both integrals converge for \(0<\gamma<1\).
The truncated resolvent satisfies
\((t-a\Delta_{\mathbb R^2})R_{t,\varepsilon}^aF
=e^{-t\varepsilon}H_{a\varepsilon}*F\).
As \(\varepsilon\downarrow0\), derivatives through order two converge
uniformly, the omitted Hessian integral being \(O(\varepsilon^{\gamma/2})\).
This proves the strong equation on big Hölder spaces.  No density of smooth
functions in their Hölder norm is asserted.

Refine the atlas by grids of a radius \(r\), with a partition \(\chi_j\),
cutoffs \(\psi_j=1\) on their supports, and centers \(x_j\).
The \(\psi_j\) supports have a fixed overlap bound, independent of \(r\),
and derivatives \(O(r^{-k})\).  With \(a_j=a(x_j)\), extend
\(\chi_jF\) by zero in its coordinate plane and put
\[
 S_tF=\sum_j\psi_jR_t^{a_j}(\chi_jF)_{\rm ext}.
\]
The cutoff product bound is
\(\|(\chi_jF)_{\rm ext}\|_\gamma\le Cr^{-\gamma}\|F\|_\gamma\).
For fixed \(r\), \(N_t(S_tF)\le K_r\|F\|_\gamma\).
Writing the full local operator as
\(-a\Delta_{\mathbb R^2}+b^k\partial_k+c_0+t\), its error on the
\(j\)-th summand is exactly
\begin{align*}
 E_{t,j}F={}&\psi_j(a_j-a)\Delta u_j
 +\psi_jb^k\partial_ku_j+\psi_jc_0u_j
 -2a\partial_k\psi_j\partial_ku_j\\
 &-a(\Delta\psi_j)u_j+(b^k\partial_k\psi_j)u_j,
 \qquad u_j=R_t^{a_j}(\chi_jF)_{\rm ext}.
\end{align*}
All but the first term are bounded by
\(C_r(t^{-1/2}+t^{-1})\|F\|_\gamma\).
For \(k_j=\psi_j(a_j-a)\),
\(\|k_j\|_\infty\le Cr\) and \([k_j]_\gamma\le Cr^{1-\gamma}\).
The split Hölder product estimate and the refined Hessian supremum bound
therefore give
\[
 \|k_j\Delta u_j\|_\gamma
 \le C\bigl(r^{1-\gamma}+r^{1-2\gamma}t^{-\gamma/2}\bigr)\|F\|_\gamma.
\]
At any pair of points only the supports containing one of the two points
contribute, at most twice the overlap bound.  The total patch count thus
does not multiply this estimate.  Hence
\begin{equation}\label{common-parametrix-error}
 \|E_t\|\le C_0r^{1-\gamma}
      +C_r(t^{-\gamma/2}+t^{-1/2}+t^{-1}).
\end{equation}
Choose \(r\) first so that the first term is at most \(1/4\); keep it
fixed, and choose \(t_0\ge\max(1,\mu_0)\) so the remaining terms are at
most \(1/4\).  The series \((I+E_t)^{-1}=\sum_{m\ge0}(-E_t)^m\)
then has norm at most two.  It gives a strong six-field right inverse
\(S_t(I+E_t)^{-1}A_*^{-1}\), with the stated weighted bound.
The coercivity calculation \eqref{common-coercivity}, with \(\mu=t\),
proves injectivity.  This proves \eqref{common-weighted-inverse} with a
constant independent of \(t\).  The normalization restriction follows
as before from \(\int\operatorname{tr}_{\rm tot}L_tU=tL_0(U)\).
\end{proof}

\subsection{Value-conditioned gradients and all symbol gaps}

\begin{theorem}[The full gradient and gap estimate]\label{common-gradient-gaps}
Fix either the primary reference or the separate lower-boundary reference,
a finite mass \(\mu\), and independently given \(B,c_1>0\).
In the primary setting the parameter range is \(0\le\tau\le1\).
In the lower-boundary setting it is the specified positive-force slab
\(0<\tau\le\tau_b\le T_+\), at the same fixed mass.
Every actual normalized solution in its stated range, with positive symbols,
\(|a|,|b|,\|M\|_{\mathrm{op}}\le B\), and \(C_1^\tau\ge c_1I\)
satisfies
\[
 \|Dw\|_\infty\le G_*,\qquad
 q\ge\delta_q>0,\quad\beta\ge\delta_\beta>0,\quad V\ge\delta_VI>0.
\]
The constants are uniform over this solution set, but depend on the fixed
setting, reference, background, \(B,c_1,\mu\).  No existence or value
estimate is part of the conclusion.
\end{theorem}
\begin{proof}
The proof separates the gradient bound from the symbol gaps. A bounded
scalar potential first localizes the energy near a putative gradient
blow-up. On each boundary face, rescaling then forces a constant value
limit. To contradict blow-up, however, the argument must also preserve a
nonzero derivative at the rescaling center; the exact projected equations
and scalar transforms below provide that step, including both extreme
mixed ratios. Only after gradients are bounded do we estimate the gaps
at their separate global minima.

On bounded values, \((a,M,b)\) and \((a,f,b)\) have uniformly
equivalent first-derivative norms.  Set
\[
 \rho=\operatorname{tr}f^{-1},\quad U=\operatorname{tr}V^{-1},\quad
 W_1=\operatorname{tr}(V^{-1}K_1),\quad G_M=2\|M_z\|^2/g.
\]
The ordered identity \eqref{common-inverse-trace-row} gives
\[
 \Delta\rho\ge a_MG_M+e^{-B}W_1+c_1e^{-B}U-D_M,
 \qquad a_M=e^{-5B}/2.
\]
Write \(\Delta a=A_a-(C_0+K_0)/q\) and
\(\Delta b=A_b-(C_2+K_2)/\beta\), with bounded \(A_a,A_b\).
In the primary setting \(r\ge e^{-2B}\min r_*>0\).  The identities
\[
 \begin{gathered}
 x^\sharp=e^{a/2}\xi_*S^{-1},\qquad
 y^\sharp=e^{-b/2}Sy_*,\\
 X=x^\sharp+a_zu^*-u^*Z,\qquad Y=y^\sharp+Zv-b_zv,\\
 Z=S^{-1}f_zS^{-1},\qquad
 u=e^{a/2}S^{-1}\phi^*,\qquad v=e^{-b/2}S\psi.
 \end{gathered}
\]
then reconstruct both endpoint derivatives from the arrow kinetics and
\(DM\).  Here \(\xi_*,y_*\) remain the fixed reference derivatives,
whereas \(x^\sharp,y^\sharp\) depend on the current metric variables.
Set
\[
 A_x=\sup_X|\xi_*|^2/g,\qquad A_y=\sup_X|y_*|^2/g,
\]
the fixed intrinsic reference contractions.
For example
\(|\nabla a|^2\le12K_0/\alpha+6e^{2B}A_x/\alpha+3e^{4B}G_M\).
Writing \(r_{\min}=e^{-2B}\min r_*\), the same three-term square
inequality in \(b_zv=y^\sharp+Zv-Y\) gives explicitly
\[
 |\nabla b|^2\le12K_2/r_{\min}
                  +6e^{2B}A_y/r_{\min}+3e^{4B}G_M.
\]
The factors use \(\|Z\|\le e^{2B}\|M_z\|\),
\(p>\alpha\), and \(r\ge r_{\min}\).
Thus \(\Psi=\rho-a-b\) satisfies \(\Delta\Psi\ge aF-D\), where
\[
 F=|\nabla a|^2+G_M+|\nabla b|^2+U+W_1
      +1/q+K_0/q+1/\beta+K_2/\beta.
\]
Indeed subtracting the two endpoint rows gives
\(\Delta\Psi\ge a_MG_M+e^{-B}W_1+c_1e^{-B}U+c_0/q+K_0/q
 +A_0/\beta+K_2/\beta-D\).
Put \(C_g=1+6e^{4B}\), \(C_k=\max(12/\alpha,12/r_{\min})\),
and \(C_{\rm ref}=6e^{2B}(A_x/\alpha+A_y/r_{\min})\).
Since \(W_1\ge(K_0+K_2)/(c+d)\), the full Dirichlet density is at most
\(C_gG_M+C_k(c+d)W_1+C_{\rm ref}\).
The positive minimum of
\(a_M/C_g,e^{-B}/(C_k(c+d)+1),c_1e^{-B},c_0,A_0,1\)
therefore supplies \(a>0\), enlarging \(D\) by \(aC_{\rm ref}\).

In the boundary setting that bottom reconstruction is unavailable.
Instead use the genuine scalar potential \(\Psi_B=\rho-a+e^{-b}\).
Its exact bottom contribution is
\[
 \Delta e^{-b}=e^{-b}\left[-A_b+(C_2+K_2)/\beta+|b_z|^2/g\right].
\]
Consequently
\[
 \begin{aligned}
 \Delta\Psi_B\ge{}& a_MG_M+e^{-B}W_1+c_1e^{-B}U
 +(c_0+K_0)/q\\
 &{}+e^{-B}(A_0+K_2)/\beta
 +(e^{-B}/2)|\nabla b|^2-D,
 \quad c_0=55/63.
 \end{aligned}
\]
Only the top reconstruction is needed, and this again bounds the entire
\(F\) with a positive coefficient.  The oscillation of \(\Psi_B\) is
at most \(3(e^B-e^{-B})+2B\).  No lower norm is inserted.

Here and below the exact kernel device is
\begin{equation}\label{common-green-kernel}
 G_0(x,y)=\frac1{2\pi}\log
 \frac{|x_0y_1-x_1y_0|^2}{|x|^2|y|^2}\le0,
 \qquad \Delta_{{\rm FS},y}G_0\,\omega_{\rm FS}
 =\delta_x-\omega_{\rm FS}/(2\pi).
\end{equation}
Off the pole this follows by differentiation, and the logarithmic flux
gives its delta mass.  The smooth actual density converts the identity to
\(u(x)-\overline u_{\rm FS}=\int G_0(x,y)\Delta_Xu(y)\omega_X(y)\).
Also \(\int(-G_0)\omega_{\rm FS}=1\).
For the appropriate potential \(\Phi\),
\(m=\Delta\Phi+D\ge aF\ge0\) has
\(\int(-G_0)m=\bar\Phi_{\rm FS}-\Phi(x)+D\int(-G_0)\omega_X\).
The bounded oscillation (in the primary case at most
\(2(e^B-e^{-B})+4B\)) therefore gives,
uniformly for every center,
\begin{equation}\label{common-log-localization}
 \int(-G_0(x,y))F(y)\omega_X(y)\le H,
 \qquad \int_{B_s(x)}F\,\omega_X\le\frac{\pi H}{\log(1/s)}.
\end{equation}
The endpoints immediately give their two-sided component-source bounds.
For the matrix use the full ordered equation
\[
 \Delta f=g^{-1}f_{\bar z}f^{-1}f_z+df+tfM
                         -S\mathcal L_V^{-1}(C_1+K_1)S.
\]
Its first term has norm at most \(e^{3B}G_M/2\); the last term is
positive Hermitian before its minus sign and has norm at most
\(e^B(\|C_1\|_\infty U+W_1)\), by positivity and trace
self-adjointness.  Thus the full two-sided component-source bound is
\(|\Delta a|+|\Delta b|+\|\Delta f\|\le C(1+F)\).

Suppose the full gradient maxima \(G_n\) tend to infinity, attained at
\(z_n\).  A fixed chart and the scale \(y=G_n(z-z_n)\) give bounded
values and first derivatives on disks exhausting the plane.  By
\eqref{common-log-localization}, the Dirichlet energy on every such disk
tends to zero.  Every locally uniform value limit is therefore constant.
This is only a value assertion.  Derivative convergence below will always
come from a two-sided Poisson bound or a bounded \(\bar\partial\) source,
using Lemma~\ref{common-local-potentials}.

The same-point rank-one facts
\[
 V\ge(c-p)I=(2/3-q)I,\qquad\lambda_{\max}V\ge c
\]
give an exhaustive list: strict interior; strict \(V\) with one or both
scalar losses; matrix loss with strict \(\beta\); mixed matrix/bottom
loss.  A \(q\)-face cannot meet a matrix face.  At the boundary lower
zero, bounded values give \(P\to0\), hence \(\beta\to d\); that center
cannot be a bottom or mixed face.

\paragraph{Scalar faces with strict matrix symbol.}
Let \(T_R=\mathcal L_V^{-1}R,T_P=\mathcal L_V^{-1}P\).
The entire ordered scalar rows are
\begin{align}
 \Delta q={}&(2-p/q)K_0+\operatorname{tr}(T_RK_1)
 +\operatorname{tr}(T_RC_1)+t[pa-\operatorname{tr}(RM)]-pC_0/q,
 \label{common-q-row}\\
 \Delta\beta={}&-(2+r/\beta)K_2+\operatorname{tr}(T_PK_1)
 +\operatorname{tr}(T_PC_1)+t[rb-\operatorname{tr}(PM)]-rC_2/\beta.
 \label{common-beta-row}
\end{align}
The second row follows from
\(\Delta r=2K_2+rH_2-\operatorname{tr}(P\widehat H_1)\), with
trace self-adjointness, and has no \(q\) denominator.
Moreover \(q_z=Xu\) and \(\beta_z=-v^*Y\).
Strict \(V\) gives bounded two-sided first-scale sources for \(f\),
so its derivatives tend to zero.  Every strict endpoint does likewise.
One collapsing endpoint therefore retains a derivative witness:
\(q_z/G_n=pa_z/G_n+o(1)\), or
\(\beta_z/G_n=rb_z/G_n+o(1)\).
At a bottom face \(r\to d>0\); this local fact suffices even in the
boundary setting and at simultaneous scalar loss.

For \(q\le1/18\), \(V\ge11I/18\), \(p\le2/9\), and
\(0\le T_R\le p/(c-p)I\).  Decompose
\(K_0=|q_z|^2/(2gp)+K_{\perp}\).
For
\[
 F_q(q)=\int_0^q\frac{\sqrt s}{\alpha+s}\,ds,\qquad
 F_b(\beta)=\int_0^\beta\frac{\sqrt s}{d-s}\,ds,
\]
the logarithmic derivative ratios are
\(F_q''/F_q'=1/(2q)-1/p\) and
\(F_b''/F_b'=1/(2\beta)+1/r\).
Their chain terms cancel the singular parallel kinetics in the respective
rows.  In the top row the remaining upper perpendicular coefficient is
at most \(2-p/q+p/(c-p)\le-18/11\); in the bottom row the remaining
bottom perpendicular coefficient and bare bottom force are nonpositive.
The bounded other kinetic terms yield
\[
 \Delta F_q\le C\sqrt q(1+G_n^2),\qquad
 \Delta F_b\le C\sqrt\beta(1+G_n^2)
\]
in their respective neighborhoods.  If the selected scalar is \(s\),
put \(\epsilon_n=s(z_n)\) and scale by \(\epsilon_n/G_n\).
Then \(Q_n=s/\epsilon_n\) is positive, locally Lipschitz and centered
at one.  The divided transforms converge to positive multiples of
\(Q^{3/2}\); their upper sources tend to zero on every fixed disk,
including possible zero points.  The nonnegative entire superharmonic
limit is constant by Lemma~\ref{common-local-potentials}; hence \(Q=1\).
On the resulting disks \(Q_n\ge1/2\), the exact scalar row has bounded
two-sided normalized source.  The other scalar denominator is absent,
and \(V\) is buffered.  Poisson \(C^1\) compactness contradicts the
preserved derivative.  This excludes both single and simultaneous scalar
losses.  Strict interior is the direct componentwise Poisson case.

\paragraph{The full separated matrix row.}
Put \(\lambda=\lambda_{\min}V,\nu=\lambda_{\max}V\).
If \(\lambda<c/4\), then \(q\ge11/24,\nu\ge c\), and
\(\nu-\lambda\ge3c/4\).  In a separated eigenframe \(e,e'\), write
\[
 \rho_1=R_{ee},\quad\pi_1=P_{ee},\quad
 \zeta=R_{ee'}=P_{ee'},\quad s_1=\rho_1+\pi_1=c-\lambda+2\pi_1,
 \quad D=C_1+K_1.
\]
With \(X=(x_1,x_2),Y=(y_1,y_2)^T\), the eigenvalue concavity term is
\[
 C_{\rm eig}=
 \frac{|u_1x_2-y_1\bar v_2|^2+|u_2x_1-y_2\bar v_1|^2}
 {g(\nu-\lambda)}\ge0.
\]
Both squares are retained.  The ordered covariant row
\(\Delta_\nabla V=(YY^*-X^*X)/g+H_0R+H_2P-
\mathcal L_{R+P}\widehat H_1\) gives the all-point identity
\begin{align}
 \Delta\lambda={}&2(|y_1|^2-|x_1|^2)/(2g)
 +\rho_1(C_0+K_0)/q+\pi_1(C_2+K_2)/\beta-s_1D_{ee}/\lambda
 \notag\\[-2pt]
 &-4\operatorname{Re}(\zeta D_{e'e})/(\lambda+\nu)
 +t[\mathcal L_{R+P}M-aR-bP]_{ee}-C_{\rm eig},
 \qquad\lambda_z=-u_1x_1+y_1\bar v_1.
 \label{common-lambda-row}
\end{align}
It is not restricted to a minimum or to \(\lambda_z=0\).

For pure matrix loss, in the six real variables \(\mathbf w=(a,f,b)\)
the exact singular image is
\[
 \Delta\mathbf w=-mN_f+E,\qquad
 m=D_{ee}/\lambda,\quad N_f=(0,S\Pi_eS,0),\quad
 |E|\le C(1+G_n^2).
\]
Choose five smooth uniformly conditioned covectors annihilating \(N_f\)
at every current state and basepoint.  Their projected complex derivatives
have bounded \(\bar\partial\) source, as verified below for the
two-normal case.  Constant value limits imply weak limit zero; Cauchy
compactness then makes the projections uniformly zero in the limit.
Square-root transport gives \(d\lambda(N_f)=s_1\ge c-\lambda\).
Thus a nonzero full-gradient witness forces
\(|\lambda_z(z_n)|/G_n\ge k>0\).
The singular kinetic in \eqref{common-lambda-row} is
\(-s_1(|x_1|^2+|y_1|^2)/(2g\lambda)\), and
\(|\lambda_z|^2\le s_1(|x_1|^2+|y_1|^2)\).
All other coefficients are bounded, apart from nonpositive concavity.
Hence
\[
 \Delta\lambda\le-|\lambda_z|^2/(2g\lambda)+C(1+G_n^2).
\]
The transform \((2/3)\lambda^{3/2}\) cancels this derivative term.
The amplitude scale \(\lambda(z_n)/G_n\), entire superharmonic limit,
and exact two-sided row on \(Q_n\ge1/2\) contradict the witness exactly
as above.  Since \(s_1\ge c-\lambda\) even when \(P=0\), this includes
the boundary lower zero.

\paragraph{Moving projections at the mixed face.}
On \(\lambda<c/4,\beta<d/2\), the full equation is
\begin{equation}\label{common-mixed-image}
 \Delta\mathbf w=-mN_f-nN_b+E,
 \quad n=(C_2+K_2)/\beta,\quad N_b=(0,0,1),
 \quad |E|\le C(1+|D\mathbf w|^2).
\end{equation}
The whole kinetic singular image is included before projection.
Four smooth conditioned covectors \(\theta^j\) annihilating both normals
are obtained by applying
\(I-N(N^TN)^{-1}N^T\) to a constant four-frame with an invertible
limiting minor.  A smaller neighborhood bounds its Gram inverse and minor.
These forms need not be exact, and no common scalar first integrals are
asserted.  For \(W_n(y)=\mathbf w_n(z_n+y/G_n)\), put
\(j_n^j=\theta^j(W_n,z_n+y/G_n)\partial_yW_n\).
Direct differentiation, using \(\Delta=g^{-1}\partial_z\partial_{\bar z}\), gives
\begin{equation}\label{common-moving-annihilator}
 \partial_{\bar y}j_n^j
 =\theta^j g_nE_n/G_n^2
 +(D_w\theta^j)[\partial_{\bar y}W_n,\partial_yW_n]
 +(D_x\theta^j)[1/G_n,\partial_yW_n].
\end{equation}
The singular sources vanish exactly at moving states, not at a frozen
center.  The nonexact quadratic curl terms remain bounded.  Thus both real
divergence and curl are controlled through the complex source.  Cauchy
compactness and the distributionally zero derivatives of the constant
value limit make all four \(j_n^j\) locally uniformly tend to zero.
There is no error subsequently divided by \(\lambda\) or \(\beta\).

For completeness the differential conditioning retains transport.  If
\(\eta=S^{-1}\delta fS^{-1},T=\delta SS^{-1}\), and
\(\Omega=(T-T^*)/2\), then \(T+T^*=\eta\) and
\[
 \delta R=-\mathcal L_\eta R+[\Omega,R]+R\delta a,
 \qquad\delta P=\mathcal L_\eta P+[\Omega,P]-P\delta b.
\]
Pairing \(\delta V\) with \(\Pi_e\) kills only its commutator.
The normal derivative matrix is exactly
\begin{equation}\label{common-normal-conditioning}
 \begin{pmatrix}d\lambda(N_f)&d\lambda(N_b)\\
 d\beta(N_f)&d\beta(N_b)\end{pmatrix}
 =\begin{pmatrix}s_1&-\pi_1\\-\pi_1&r\end{pmatrix},
 \quad s_1r-\pi_1^2\ge(3c/4)(d/2)>0.
\end{equation}
Indeed the determinant is \((c-\lambda)r+2\pi_1r-\pi_1^2\).
The four annihilators together with \(d\lambda,d\beta\) form a bounded
invertible six-frame.  No integrability or bracket direction is lost.

\paragraph{Positive combined normals and all three ratios.}
For \(1<A<2\), set \(t_A=\lambda+A\beta\),
\(a_N=s_1-A\pi_1\), and \(b_N=Ar-\pi_1\).
Trace self-adjointness in \eqref{common-beta-row} gives the entire row
\begin{align}
 \Delta t_A={}&-a_N(C_1)_{ee}/\lambda-b_NC_2/\beta
 -a_N(K_1)_{ee}/\lambda-b_NK_2/\beta
 \notag\\
 &+2(|y_1|^2-|x_1|^2)/(2g)-2AK_2+\rho_1(C_0+K_0)/q
 +AP_{e'e'}D_{e'e'}/\nu
 \notag\\
 &+4(A-1)\operatorname{Re}(\zeta D_{e'e})/(\lambda+\nu)
 +t\{[\mathcal L_{R+P}M-aR-bP]_{ee}
       +A[rb-\operatorname{tr}(PM)]\}-C_{\rm eig}.
 \label{common-combined-row}
\end{align}
Here \(a_N\ge3c/4=5/8\) and \(b_N\ge(A-1)d/2\).
Both singular forces and kinetics are negative; the mixed entries are
bounded above, not assigned a false sign.  Since
\((t_A)_z=-u_1x_1+(1-A)\bar v_1y_1-A\bar v_2y_2\),
the three-term square inequality gives
\[
\begin{gathered}
 |(t_A)_z|^2\le K_At_A
 \left[a_N(|x_1|^2+|y_1|^2)/\lambda
       +b_N(|y_1|^2+|y_2|^2)/\beta\right],\\
 K_A=3\max\left(\frac{64}{15},\frac{(A-1)^2d}{5/8},
                              \frac{2A}{A-1}\right).
\end{gathered}
\]
Thus \(\Delta t_A\le-|(t_A)_z|^2/(2gK_At_A)+C(1+G_n^2)\).
For \(A_-=5/4,A_+=7/4\), a common \(K=30\) suffices.
The power \(F_A=t_A^{61/60}\) has logarithmic derivative ratio
\(1/(60t_A)\), and therefore
\begin{equation}\label{common-combined-power}
 \Delta F_A\le Ct_A^{1/60}(1+G_n^2).
\end{equation}

Let \(\delta_n=\lambda(z_n)+\beta(z_n)\) and use the shared scale
\(\ell_n=\delta_n/G_n\).  The functions
\(Q_\lambda=\lambda/\delta_n,Q_\beta=\beta/\delta_n\) are locally
uniformly Lipschitz.  Dividing \eqref{common-combined-power} by
\(\delta_n^{61/60}\) makes its upper rescaled source at most
\(C\delta_n(Q_\lambda+AQ_\beta)^{1/60}(1+G_n^{-2})\), which tends
to zero.  The two nonnegative entire transformed limits are constant.
The two distinct \(A\)'s force
\(Q_\lambda\to L,Q_\beta\to B_0\), with \(L,B_0\ge0,L+B_0=1\).
Also
\(\widetilde W_n=(\mathbf w_n(z_n+\ell_ny)-\mathbf w_n(z_n))/\delta_n\)
is Lipschitz and zero at the origin.  The four projected derivatives
vanish on these nested disks, and the two scalar gradients vanish
distributionally.  Coefficients freeze since state changes are
\(O(\delta_n)\).  The invertible frame
\eqref{common-normal-conditioning} forces every weak derivative of a
limit of \(\widetilde W_n\) to vanish; hence its value limit is zero.
Derivative preservation still requires the following three separate bridges.

If \(L,B_0>0\), both relative denominators are buffered.  The source
factor in the equations for \(\widetilde W_n\) is \(\delta_n/G_n^2\).
All six component Laplacians are bounded two-sided, so Poisson \(C^1\)
convergence contradicts the full-gradient witness at the origin.

If \(L=0,B_0=1\), choose five covectors annihilating only \(N_f\).
For \(J=\vartheta(\mathbf w,x)\mathbf w_z/G_n\) on the shared scale,
\[
 \partial_{\bar y}J=(\delta_n/G_n^2)\vartheta g_n\Delta\mathbf w
       +\delta_n O(1)+\ell_nO(1).
\]
The whole middle singular image cancels exactly, while the bottom source
is bounded because \(\delta_n/\beta\le2\) on fixed disks.
The weak zero limit supplied by \(\widetilde W_n\to0\), followed by
Cauchy compactness, gives \(J\to0\) uniformly.  The witness is therefore
in \(N_f\), and \(|\lambda_z(z_n)|/G_n\ge k>0\).
Now use the smaller scale \(\epsilon_n/G_n\),
\(\epsilon_n=\lambda(z_n)\).  The bottom symbol remains comparable
to \(\delta_n\), because its variation is \(O(\epsilon_n)\) and
\(\epsilon_n/\delta_n\to0\).  The exact matrix row yields
\[
 \Delta\lambda\le-|\lambda_z|^2/(2g\lambda)
          +C(1+G_n^2)(1+1/\beta).
\]
The transformed amplitude source is
\(C\epsilon_n\sqrt Q(1+G_n^{-2})(1+1/\beta)\to0\).
Entire Liouville gives \(Q=1\).  On \(Q_n\ge1/2\), the exact row has
bounded two-sided source; its bottom terms carry
\(\epsilon_n/\beta\to0\).  Poisson compactness contradicts the witness.
No unspecified \(o(1)/\lambda\) error is used.

If \(B_0=0,L=1\), the five components \((a,f)\) have bounded
two-sided shared-scale sources: they contain no bottom singularity and
\(\lambda/\delta_n\ge1/2\).  Their derivative limits are zero.
The witness is in \(b\), and
\(\beta_z=rb_z-v^*Zv-v^*y^\sharp\), with \(r\to d\), preserves a nonzero
\(\beta_z/G_n\).  On the smaller scale
\(\epsilon_n/G_n,\epsilon_n=\beta(z_n)\), the matrix symbol remains
comparable to \(\delta_n\).  The exact bottom transform gives
\(\Delta F_b\le C\sqrt\beta(1+G_n^2)(1+1/\lambda)\).
Its amplitude source tends to zero through
\(\epsilon_n/\lambda\to0\); Liouville gives \(Q=1\), and the exact
two-sided bottom row on \(Q_n\ge1/2\) contradicts the witness.
These three ratios exhaust the mixed face.  At a boundary mixed center
\(r\ge d/2\) locally, so the lower zero is already excluded without a
global lower norm.  The gradient bound is proved in both settings.

\paragraph{Separate actual minima.}
All kinetics are now bounded.  At an actual small global \(q\)-minimum,
the same-point \(V\) buffer and \eqref{common-q-row} give
\(0\le\Delta q\le C-\alpha c_0/q\), proving the top gap.
First minimize \(\lambda+(3/2)\beta\).  A minimum tending to zero lies
in a separated mixed chart.  The exact combined row gives there
\[
 0\le\Delta(\lambda+(3/2)\beta)
 \le-(5/8)c_1/\lambda-(d/4)A_0/\beta+C,
\]
a contradiction.  Hence this total normal is at least \(\delta_t>0\).
At an actual \(\lambda\)-minimum
\(\epsilon<\min(c/4,\delta_t/2)\), its same-point bottom symbol is
at least \(\delta_t/3\).  The full matrix row gives
\(0\le\Delta\lambda\le-s_1c_1/\epsilon+C\), so \(V\) has a gap.
At an actual \(\beta\)-minimum below
\(\min(d/2,\delta_t/3)\), its same-point matrix symbol is at least
\(\delta_t/2\); the ordered probe is bounded and
\(0\le\Delta\beta\le-(d/2)A_0/\beta+C\).  This gives the bottom gap.
The minimizing points were never synchronized.  Repeated eigenvalues
cause no differentiability issue because every small matrix-normal minimum
has \(\nu-\lambda\ge3c/4\).  At a lower zero \(\beta=d\).
\end{proof}

\subsection{A single strong mass, the all-solution anchor, and degree}

The exact scalar-principal form is useful both near the reference and on
a buffered legal domain.  Define
\[
 E_M(H)=S^{-1}D\exp_M(H)S^{-1},\qquad
 Q_M=g^{-1}S^{-1}\bigl[D^2\exp_M(M_{\bar z},M_z)
                         -f_{\bar z}f^{-1}f_z\bigr]S^{-1}.
\]
The integral exponential formula makes these smooth without choosing
spectral frames.  In an eigenframe the multipliers of \(E_M\) are
\(\sinh((m_i-m_j)/2)/((m_i-m_j)/2)\), with value one at equality.
Thus it is invertible, including repeated spectra, and
\(E_M(I)=I,E_M(M)=M\).  Direct differentiation gives
\(\widehat H_1=2I-E_M(\Delta M)-Q_M\).
The exact six equations are consequently
\begin{equation}\label{common-scalar-principal}
 \begin{aligned}
 (-\Delta+t)a&=(C_0^\tau+K_0)/q-(H_{*0}-\alpha),\\
 (-\Delta+t)M&=-dI+E_M^{-1}\mathcal L_V^{-1}(C_1^\tau+K_1)
                         +E_M^{-1}Q_M,\\
 (-\Delta+t)b&=(C_2^\tau+K_2)/\beta-c.
 \end{aligned}
\end{equation}
The Lyapunov inverse acts first, \(E_M^{-1}\) second.  On bounded
buffered values the right side is a smooth function
\(G(x,w,Dw,\tau)=H(x,w,\tau)+B_1(x,w)[Dw]+Q_1(x,w)[Dw,Dw]\)
with no explicit \(t\)-dependence.  Actual reference-arrow derivatives,
including \(D\psi\) at its zero in the boundary setting, remain in these
coefficients.  There is no \(1/r\), normalized lower arrow, or composition
division.

\begin{lemma}[Signed upper branch and all-solution anchor]\label{common-anchor}
For each of the two fixed settings, there is one finite mass \(\mu_s\),
chosen independently of every later lower cutoff \(\tau_-\), with a smooth
\(C^{2,\gamma}\)-valued auxiliary solution curve on \([\tau_b,1]\),
through \((1,0)\).  At \(\tau_b>0\) its solution is unique among all
positive-symbol solutions and has an invertible unrestricted derivative.
Its normalized derivative is also an isomorphism.
\end{lemma}
\begin{proof}
There are two tasks at the anchor parameter: construct a nearby
solution with invertible derivative, and show that no other
positive-symbol solution occurs there. The latter is needed to identify
the degree of the whole bounded solution domain with the nonzero local
index. Neither task asserts uniqueness below the anchor.

Let \(L=\|F_{*1}\|_{\infty,\mathrm{op}}\), let
\(T_+=1\) if \(L=0\) and
\(T_+=\min(1,A_0/(2L))\) otherwise, and set \(\tau_b=T_+/2\).
Then \(C_1^\tau\ge A_0I/2\) on \([0,T_+]\); no such sign is needed
on the whole upper interval.  Write the entire remainder as
\[
 \mathcal R_t(w)=\widehat F(w)-F_*-J_*w+t(B(w)-A_*)w.
\]
The equation is
\(L_tw=-(1-\tau)F_*-\mathcal R_t(w)\).
On a fixed small value neighborhood all finite coefficient derivatives
of the exponential, inverse, square root, arrows, and transported curvature
are bounded.  The remainder consists of: coefficient variation times
\(D^2w\); quadratic first derivatives; variations multiplying a single
first derivative; the zeroth Taylor remainder; and mass variation.
Hölder multiplication on a weighted radius \(R\), with \(R/t\le r_0\),
gives term by term
\[
 \|\mathcal R_t(w)\|_\gamma\le C_NR^2/t,
 \qquad
 \|\mathcal R_t(w)-\mathcal R_t(z)\|_\gamma
 \le C_NR/t\,\|w-z\|_{t,\gamma}.
\]
For instance the first product is
\(O(\|w\|_\gamma\|D^2w\|_\gamma)\), the second
\(O(\|Dw\|_\gamma^2)\), and the mass product
\(O(t\|w\|_\gamma^2)\); each is \(O(R^2/t)\).
Difference products give the second inequality in the same weighted norm.
This does not require small unweighted second derivatives.

Take \(C_I\ge1\),
\[
 R=2C_I(\|F_*\|_\gamma+1),\qquad
 t_{\rm req}=\max(t_0,1,R/r_0,4C_IC_NR).
\]
For \(t\ge t_{\rm req}\),
\(-L_t^{-1}[(1-\tau)F_*+\mathcal R_t(w)]\) maps the complete
radius-\(R\) ball to itself and contracts by at most \(1/4\).
Geometric iteration constructs its unique fixed point.  At \(\tau=1\)
that point is zero.  Value closeness gives
\(q\ge5/21,\beta\ge23/30,V\ge2I/21\) after enlarging the threshold.
For a fixed finite mass the weighted norms are uniformly equivalent on
\([\tau_b,1]\).  The resolvent identity and uniform contraction give
continuity, and the factorization
\[
 D_w\mathcal E=L_t[I+L_t^{-1}D\mathcal R_t(w_\tau)]
\]
has a norm-convergent inverse.  Difference quotients and repeated
differentiation give the smooth curve, with one-sided endpoint derivatives.
Exact total integration gives \(tN(w)=0\) after construction.

It remains to place every solution at the anchor in that ball.
On the positive-force slab Lemma~\ref{common-positive-size} gives
\(\|w\|_\infty\le C_W/t\), independently of any solution selection.
For large \(t\) this buffers all symbols near the reference.
In a chart set \(y=\sqrt t(x-x_0),W=tw\).  The complete normalized
equation \eqref{common-scalar-principal} is
\[
 \Delta_yW=4g(x)\{W-H(x,W/t,\tau)
       -t^{-1/2}B_1(x,W/t)[D_yW]
       -t^{-1}Q_1(x,W/t)[D_yW,D_yW]\}.
\]
If \(G_n=\sqrt{t_n}\|Dw_n\|_\infty\to\infty\), rescale again
by \(\sqrt{t_n}G_n\) at an attaining point.  The rescaled values are
bounded, all gradients are bounded with a nonzero central witness, and
the full Laplacian is
\[
 4g_n\{(\widetilde W_n-H_n)/G_n^2
       -B_{1,n}[D\widetilde W_n]/(\sqrt{t_n}G_n)
       -Q_{1,n}[D\widetilde W_n,D\widetilde W_n]/t_n\}\longrightarrow0.
\]
The domains exhaust the plane.  Local Poisson compactness preserves the
gradient witness in \(C^1\), whereas the bounded harmonic six-vector
limit is constant: its derivative is at most \(C\|W\|_\infty/R\)
on arbitrary disks.  This is a contradiction.  A finite threshold follows
also directly from
\(|DU(0)|\le C_P(\|U\|_\infty/R+R\|\Delta U\|_\infty)\):
first choose a fixed \(R\) to make the value term less than one quarter
of the witness, then large \(G,t\) to make each of the three source terms
less than one quarter.  Thus \(\|Dw\|_\infty\le C_G/\sqrt t\).

On first-scale disks bounded values, gradients and Laplacians give
\(C^{1,\nu}\) bounds for any \(\gamma<\nu<1\), and then
\(C^{2,\gamma}\) bounds by the Hessian cancellation estimate.
Rescaling and the finite atlas give
\[
 \|w\|_\gamma\le C_St^{-1+\gamma/2},\quad
 \|Dw\|_\gamma\le C_St^{-1/2+\gamma/2},\quad
 \|D^2w\|_\gamma\le C_St^{\gamma/2}.
\]
Farther point pairs are controlled by the corresponding supremum divided
by the first-scale radius to power \(\gamma\).  The complete original
remainder consequently satisfies
\(\|\mathcal R_t(w)\|_\gamma\le C_Rt^{-(1-\gamma)}\):
the coefficient--Hessian, gradient-square and mass terms have precisely
this order, and linear-gradient and zeroth remainders decay faster.
The mass-uniform inverse then puts every solution in the same contraction
ball for a finite \(T_{\rm all}\).  Choose once
\[
 \mu_s\ge\max(t_{\rm req},T_{\rm all})/\tau_b.
\]
At the anchor all positive-symbol solutions therefore coincide with the
constructed one.  Differentiating the exact total identity there gives
\(\int\operatorname{tr}_{\rm tot}(D_w\mathcal E\,U)=tDN(w)U\).
The unrestricted inverse restricts to an isomorphism
\(\ker DN(w)\to Y\).  No simplicity premise or kernel classification
has been used.  For the boundary setting all constants are recomputed
from its smooth zero-safe coefficients; no primary uniformity across
rank loss is asserted.
\end{proof}

\begin{theorem}[Positive-parameter continuation at one fixed mass]
\label{common-positive-continuation}
For each fixed primary setting, and separately for each fixed
lower-boundary setting, the mass chosen in Lemma~\ref{common-anchor}
has a connected subset of smooth normalized positive-symbol auxiliary
solutions containing \((1,0)\) and projecting onto all \((0,1]\).
There is no assertion of a single-valued curve below the anchor,
uniqueness there, a nested choice of components, or compactness at zero.
\end{theorem}
\begin{proof}
First fix a positive lower cutoff and use degree invariance to obtain a
solution on every parameter slice. Slice-wise existence alone does not
give the connected set in the statement. A compact zero-set argument
then shows that the component through the anchor reaches the lower
cutoff. Finally, the components for different cutoffs can be joined
through that same anchor, without requiring them to be nested or
single-valued.

Fix \(0<\tau_-<\tau_b\), without changing \(\mu_s\).
The independent estimate \(B_{\tau_-}=C_{\rm size}/(\tau_-\mu_s)\) applies
to every solution on this interval, even before existence is known.
Theorem~\ref{common-gradient-gaps} with \(c_1=A_0/2\) gives all
stage gradients and gaps.  The zero-safe boundary localization, rather
than the primary positive-lower-norm step, supplies this inference in
the boundary setting.

Let \(E=C^{1,\gamma}(X;\mathbb R\oplus\operatorname{Herm}(2)
\oplus\mathbb R)\), and choose one fixed finite \(\kappa\ge1\)
large enough for the scalar Gaussian/localization construction above.
Applied componentwise, that construction gives the bounded scalar inverse
\(S_\kappa=(-\Delta+\kappa)^{-1}\) on all six fields.  Define
\[
 K_\tau(w)=S_\kappa[(\kappa-\tau\mu_s)w+G(x,w,Dw,\tau)].
\]
Its fixed points are exactly the full equations
\eqref{common-scalar-principal}; a \(C^{1,\gamma}\) fixed point gains
\(C^{2,\gamma}\), and normalization follows afterwards.  Stage values,
gradients and gaps give uniform \(C^{2,\nu}\) zero-set bounds for every
fixed \(\gamma<\nu<1\): the bounded component source first gives
\(C^{1,\sigma}\), \(\nu<\sigma<1\), then the complete smooth source
is \(C^{0,\nu}\), and the local Hessian estimate gives \(C^{2,\nu}\).
Choose \(R_{\tau_-}\) strictly above the zero-set \(C^{1,\gamma}\) bound and
\(0<\epsilon_{\tau_-}<\frac12\min(\delta_{q,\tau_-},
\delta_{\beta,\tau_-},\delta_{V,\tau_-})\).  The bounded open set
\[
 \Omega_{\tau_-}=\{\|w\|_{C^{1,\gamma}}<R_{\tau_-},
                q>\epsilon_{\tau_-},\ \beta>\epsilon_{\tau_-},
                V>\epsilon_{\tau_-} I\}
\]
has legal closure and no fixed point on its boundary.
On this bounded buffered domain the images of \(K_\tau\) have a uniform
\(C^{2,\gamma}\) bound, hence are jointly precompact in \(E\) by
Ascoli and lower-exponent interpolation.  The forcing is affine in
\(\tau\), so parameter dependence is uniformly Lipschitz there.
The exact zero set is compact even in
\([\tau_-,\tau_b]\times C^{2,\gamma}\).  Convexity and an invariant-domain
property are unnecessary.

We use the compact-map degree and local index of Leray and
Schauder~\cite{LS1934}, in their cited 1934 journal edition.
Sections 4--8, 10--11, and 12--16 supply the degree for
\(I-K\) on bounded open real Banach domains with compact maps and no
boundary zeros, its existence, excision, additivity and compact-homotopy
invariance, and the nondegenerate local index \(\pm1\).
The application here is to the bounded legal domains just checked, not
to a global map on the whole Banach space; neither the weak-compactness
discussion in section 9 nor the global-map discussion in section 17 is
an additional premise.  These source degree/index properties, including
their finite-dimensional degree background, are the external input;
all analytic hypotheses and the component application are proved here.

At the unique anchor, \(DK\) is compact.  A kernel element of
\(I-DK\) gains \(C^{2,\gamma}\) and solves the actual full linearization,
so Lemma~\ref{common-anchor} makes it zero.  The compact Fredholm
alternative required here has a direct proof.  If \(A=I-C\), with
\(C\) compact, is injective but not bounded below, a unit sequence with
\(Ax_n\to0\) has a convergent subsequence through \(Cx_n\), contradicting
injectivity.  Thus its range and all \(R_n=A^nE\) are closed.  If it is
not onto, every inclusion \(R_{n+1}\subset R_n\) is strict: equality
at one stage, canceled using injectivity of \(A^n\), would give \(AE=E\).
Choose unit \(x_n\in R_n\) at distance greater than \(1/2\) from
\(R_{n+1}\): for \(y\) outside that proper closed subspace, put
\(d_0=\operatorname{dist}(y,R_{n+1})>0\), take
\(z\in R_{n+1}\) with \(\|y-z\|<2d_0\), and normalize \(y-z\).
For \(n>m\), both \(Cx_n\) and \(Ax_m\) belong to \(R_{m+1}\), so
\(\|Cx_m-Cx_n\|>1/2\), contradicting compactness.  Hence \(A\) is
onto with bounded inverse, without an approximation-property premise.
The local linear homotopy is zero-free on a sufficiently small sphere:
if the nonlinear error is at most
\(\|h\|/(2\|A^{-1}\|)\), it cannot cancel \(Ah\).
Thus the external local index result gives \(\pm1\).
Anchor uniqueness and excision make the degree on \(\Omega_{\tau_-}\)
nonzero.  The checked compact zero-free homotopy preserves it, yielding
solutions on every stage slice.

For the stronger component conclusion, let \(\mathcal S\) be this compact
zero set and \(p=(\tau_b,w_b)\).  The equivalence class \(P_n\) of
\(p\) under finite chains of steps less than \(1/n\) is clopen in
\(\mathcal S\), compact, and nested.  Their intersection is the component
of \(p\): it is connected, since a separation of the compact intersection
into two disjoint compact neighborhoods with positive distance would,
by nested compactness, eventually contain \(P_n\) in their union, whereas
a sufficiently small-step chain cannot cross that distance.  Conversely
every connected set through \(p\) lies in every clopen \(P_n\).
If this component misses the \(\tau_-\)-slice, nested compactness gives
some \(P_n\) missing that slice.  A small relative product neighborhood
\(\mathcal O\) isolates exactly this selected compact clopen zero set,
with no zero on its relative boundary.

The degree of the moving slice \(\mathcal O_\tau\) is locally constant.
Indeed at any nonempty selected slice choose a bounded open \(U\) whose
closure lies in \(\mathcal O_{\tau_0}\) and contains precisely its compact
selected zeros.  For nearby parameters its closure remains in
\(\mathcal O_\tau\), all selected zeros remain in \(U\), and there are no
boundary zeros; otherwise compactness produces a contradicting limiting
zero.  Excision reduces the moving-slice degree to that on \(U\), and
fixed-domain homotopy invariance proves constancy.  Empty selected slices
stay empty nearby by the same compactness argument.  This degree would
be zero at \(\tau_-\) but \(\pm1\) at \(\tau_b\), a contradiction.
Thus the anchor component \(C_{\tau_-}\) meets \(\tau_-\), and its connected
projection contains \([\tau_-,\tau_b]\).
On the compact zero set the identity between the \(C^{2,\gamma}\) and
\(C^{1,\gamma}\) topologies is a compact-to-Hausdorff homeomorphism,
so this is also a \(C^{2,\gamma}\) component.

Join every \(C_{\tau_-}\) to the already constructed upper curve.  All share
the same anchor at the same mass.  A union of connected sets sharing one
point is connected, and its projection is \((0,1]\).  No nesting,
coherent graph, limit exchange, or compact global union is used.
The smooth scalar-principal bootstrap makes every fixed point smooth.
\end{proof}

\subsection{Scalar propagation and the zero-parameter limit}

\begin{lemma}[Uniform upper radial trace]\label{common-trace-bound}
Every positive-parameter auxiliary solution in either fixed setting has
\(T\le C_T\), independently of mass, parameter, force sign, values,
derivatives and symbol minima.  The same bound holds for an actual
normalized zero-parameter solution, if one exists.
\end{lemma}
\begin{proof}
The total source mean is zero.  Define
\(z_{\rm raw}(x)=\int G_0(x,y)f_*(y)\omega_X(y)\) and subtract its
actual mean to obtain \(z\).  The constant term in the kernel identity
vanishes against the total source mean, so \(\Delta_Xz=f_*\).
Local logarithmic potentials for its smooth source, and their harmonic
remainders, show successively that \(z\) is smooth.
With \(n_\omega=\sup\omega_X/\omega_{\rm FS}\), kernel integrability
gives \(\|z\|_\infty\le2n_\omega\|f_*\|_\infty\).
Set \(Z=(1-\tau)z,D=T-Z\).  The exact trace equation gives
\(\Delta D=tv\ge tT=t(D+Z)\).  At a global maximum of \(D\),
\(D\le\|Z\|_\infty\); hence \(T\le2\|z\|_\infty=:C_T\).
The resulting bound contains no factor \(1/t\).  At zero, if an actual normalized solution
exists, \(T-z\) is a constant \(C\); since
\(\int v=0\), \(v\ge T\), and \(\int z=0\), one has \(C\le0\).
This proves only the stated trace interface, not a full value bound.
\end{proof}

\begin{lemma}[Negative-potential propagation]\label{common-negative-propagation}
For every positive-symbol solution on \(0<\tau\le\tau_b\), at the same
fixed mass in either setting, set \(u=\log\operatorname{tr}(e^{-M})\)
and let bars denote actual \(\omega_X\)-means.  Fixed constants give
\[
 \sup u\le\bar u+C_u,\qquad b\ge\bar b-C_b,
 \qquad \sup(-a)\le\overline{-a}+C_a.
\]
They do not assume any of these means bounded and do not require a
global positive lower-arrow norm.
\end{lemma}
\begin{proof}
The independent size numerator and the positive endpoint rows give
\[
 \Delta(-a)\ge-L_0,\quad\Delta(-b)\ge-L_2,
 \quad L_0=\|H_{*0}-\alpha\|_\infty+C_{\rm size},
 \quad L_2=c+C_{\rm size}.
\]
For \(h=e^{-M},\rho=\operatorname{tr}h\), the ordered inverse row and
positive middle forcing imply
\[
 \Delta\rho\ge g^{-1}\operatorname{tr}(hf_zhf_{\bar z}h)
                       -(d+C_{\rm size})\rho.
\]
Put \(B_2=h^{1/2}f_zh^{1/2}\).  Then
\(\rho_z=-\operatorname{tr}(hB_2)\) and the positive trace term is
\(\operatorname{tr}(hB_2B_2^*)\).  Weighted Hilbert--Schmidt
Cauchy--Schwarz, for \((A,D)_h=\operatorname{tr}(hAD^*)\), gives
\(|\operatorname{tr}(hB_2)|^2\le\rho\operatorname{tr}(hB_2B_2^*)\).
Thus \(\Delta u\ge-L_1\), \(L_1=d+C_{\rm size}\), without dropping
any moving eigendirection or assigning a sign to an unordered product.

For any real \(U\) with \(\Delta U=s\ge-L\), its source has total
integral zero and \(\|s\|_{L^1}\le2L\operatorname{Vol}X\).
The exact kernel identity gives
\(U(x)-\bar U_{\rm FS}\le Ln_\omega\).
Let \(J(y)=\overline{G_0(\cdot,y)}\), with the actual mean.
Then \(|J(y)|\le n_\omega/(2\pi)\), and averaging the exact identity,
not merely its inequality, gives
\(|\bar U-\bar U_{\rm FS}|\le2Ln_\omega\).
Consequently \(\sup U\le\bar U+3Ln_\omega\).
Apply this to \(-a,-b,u\).  The argument uses only fixed reference
coefficients and positivity, including at a lower-arrow zero.
\end{proof}

\begin{lemma}[Finite scalar fluctuation bounds]\label{common-scalar-fluctuations}
For every fixed finite \(p\ge1\) in either setting,
\[
 \|a-\bar a\|_{L^p}+\|b-\bar b\|_{L^p}+\|u-\bar u\|_{L^p}\le C_p.
\]
In particular \(\int|b-\bar b|\omega_X\le K_b\), with a fixed
parameter-uniform constant.  No \(L^\infty\) fluctuation conclusion is
inferred from this kernel mapping.
\end{lemma}
\begin{proof}
For every integer \(k\ge1\), the exact kernel moment is
\(\int|G_0(x,y)|^k\omega_{\rm FS}(x)=(2\pi)^{1-k}k!\), uniformly
in \(y\).  Indeed move \(y\) to \([1:0]\) for this kernel computation;
the squared chordal distance \(s=|z|^2/(1+|z|^2)\) has angularly
integrated measure \(2\pi\,ds\), and
\(\int_0^1(-\log s)^kds=k!\) by \(k\) integrations by parts after
\(s=e^{-v}\).  This does not change the actual source, mean or norm.
For \(H_X(x,y)=G_0(x,y)-\overline{G_0(\cdot,y)}\), Hölder's inequality
for the mean and \(\omega_X\le n_\omega\omega_{\rm FS}\) yield
\[
 \sup_y\|H_X(\cdot,y)\|_{L^k(\omega_X)}
 \le2n_\omega^{1/k}(2\pi)^{1/k-1}(k!)^{1/k}=:C_k'.
\]
The mean-adjusted Green representation and Minkowski give
\(\|U-\bar U\|_{L^k}\le C_k'\|\Delta U\|_{L^1}\).
The source budgets proved in the preceding lemma apply to \(-a,-b,u\).
Every finite real \(p\) follows from an integer \(k\ge p\) and finite
volume Hölder.  In particular one can take
\(K_b=C_1'\,2L_2\operatorname{Vol}X\).
Logarithmic integrability and the distributional kernel identity justify
the same calculation for \(C^{2,\gamma}\) solutions.
\end{proof}

\begin{theorem}[Strong original endpoint and strict calibration]
\label{common-original-endpoint}
In either fixed setting suppose that the positive-parameter solutions at
the one fixed finite mass have uniform full values, full first derivatives,
and positive \(q,\beta,V\) gaps on \(0<\tau\le\tau_b\).
Every sequence of such actual solutions with \(\tau_n\downarrow0\) has
a subsequence converging strongly in \(C^{2,\gamma}\) to smooth positive
metrics on the original bundles, with the original maps fixed, solving
the undamped equation, the single normalization, and strict calibration
in all blocks.  Existence at each positive parameter is separately needed
to deduce existence at zero from this assertion.
For the primary application the value input is
Theorem~\ref{primary-value-bound}; for the lower-boundary application it is
Theorem~\ref{lower-value-bound}.  Theorem~\ref{common-gradient-gaps}
then supplies the other uniform bounds, and
Theorem~\ref{common-positive-continuation} independently supplies positive
parameter existence.
\end{theorem}
\begin{proof}
Let \(\nu=(1+\gamma)/2\) and \(\sigma=(1+\nu)/2\).
In \eqref{common-scalar-principal}, scalar reciprocals are bounded by the
given gaps.  The Lyapunov eigenmultipliers \((v_i+v_j)/2\) are at least
\(\delta_V\), and \(E_M\) and its inverse are smooth with bounded finite
derivatives on the compact bounded value range.  All ordered sources are
smooth in basepoint, values and parameter, and at most quadratic in the
full first derivatives.  The mass is one fixed finite number.  Thus all
six component Laplacians have a uniform two-sided bound, without an
assumed Hölder gradient norm.  Local logarithmic potentials have gradient
modulus \(Ch(1+|\log h|)\); their bounded harmonic remainders have
controlled interior derivatives.  A finite nested atlas gives uniform
\(C^{1,\sigma}\) bounds.  The full smooth source is now uniformly
\(C^{0,\sigma}\), hence \(C^{0,\nu}\).  The zero-circle-mean Hessian
kernel, with the source value subtracted, gives a uniform
\(C^{2,\nu}\) bound as proved in Lemma~\ref{common-local-potentials}.

Ascoli gives a common subsequence converging uniformly through second
derivatives; Hölder interpolation at the strictly lower exponent
\(\gamma\) makes the convergence strong \(C^{2,\gamma}\).
Compactness in the same \(C^{2,\nu}\) norm is not claimed.
The reference chart reconstructs
\(h_0=h_{*0}e^a,h_1=h_{*1}e^M,h_2=h_{*2}e^b\).
Bounded values give uniformly positive upper and lower metric bounds on
the original bundles.  Smooth symbol dependence preserves their gaps.
There is no rank loss or punctured-gauge replacement: the holomorphic maps
and bundle structures have been fixed throughout.

Strong \(C^{2,\gamma}\) passes every curvature, full arrow derivative,
transported adjoint and ordered product.  Bounded values and fixed mass
give \(\tau_n\mu_sB(w_n)w_n\to0\), while \(\tau_nF_*\to0\), so
\(\widehat F(w_0)=0\).  The single integral
\(\int[qa+\operatorname{tr}(VM)+\beta b]\omega_X=0\) passes by uniform
value convergence.  No separate vertex mean is inserted.
In the boundary setting every coefficient remains smooth at the section
zero and at coincidence of the two divisors: \(D\psi\) is retained and
no coefficient divides by \(\psi,r\), the composition, or a normalized
lower arrow.  This is the zero-safe regularity check, not a transfer of
the primary lower-norm argument.

At zero the scalar-principal equation is smooth.  Starting with
\(C^{2,\gamma}\), its source is \(C^{1,\gamma}\); differentiated local
Poisson/cutoff estimates give \(C^{3,\gamma}\).  Repeating at each finite
order makes all fields and metrics smooth.
At the zero-parameter endpoint write \(q^{(0)}=q(w_0)\).
The exact undamped rows are
\[
 q^{(0)}(H_0-\alpha)=A_0+K_0,\quad
 \mathcal L_{V_0}(\widehat H_1-\alpha I)=A_0I+K_1,\quad
 \beta_0(H_2-\alpha)=A_0+K_2.
\]
The endpoint rows imply \(H_0,H_2>\alpha\).  The ordered positive
Lyapunov inverse has \(\mathcal L_{V_0}^{-1}(A_0I)=A_0V_0^{-1}>0\),
so \(\widehat H_1>\alpha I\) without commutation.
The frozen vertical coefficients are
\(S_0=q^{(0)}+\alpha,S_1=V_0+\alpha I,S_2=\beta_0+\alpha\).
Thus the phase coefficient \(\mathcal B_j=H_j+S_j\) satisfies
\(\mathcal B_j>2\alpha I=I/3\) in every block.
For the frozen \(\theta_4=\pi-\arctan6\),
\(\sin\theta_4=6/\sqrt{37}>0\), and the exact phase equation is
\(\cos\theta_4\mathcal B_j=\sin\theta_4\mathcal A_j\).
Consequently
\[
 \cos\theta_4\mathcal A_j+\sin\theta_4\mathcal B_j
 =\mathcal B_j/\sin\theta_4>\sqrt{37}\,I/18.
\]
This proves strict calibration at the smooth undamped endpoint,
not at a forced auxiliary solution.
\end{proof}

\section{Uniform values and existence on the primary family}
\label{primary-section}

Fix an original primary chain
\[
 \mathcal O(1)\xrightarrow{\psi}\mathcal O(2)^{\oplus2}
 \xrightarrow{\phi}\mathcal O(3).
\]
Thus each component pair is a basis of $H^0(\mathcal O(1))$ and
$\phi\psi$ is not identically zero. We use the reference metrics and the
full Hermitian chart of \ref{common-fields}. All constants in this section
refer to this fixed chain, its reference, and the fixed background
$\omega_X$, whose volume is $\mathcal V=2\pi$. Write
\[
 \alpha=\frac16,\qquad c=\frac56,\qquad d=\frac{11}{6},
 \qquad A_0=\frac{37}{36},\qquad W=\frac{10}{3}.
\]
Choose the one finite mass $\mu_s$ of
\ref{common-positive-continuation} before varying $\tau$, and put
$t=\tau\mu_s$. The positive-force interval is $0<\tau\le\tau_b$.
No constant below is asserted to be uniform when the original chain varies.

\begin{theorem}[Fixed-chain value bound]\label{primary-value-bound}
For every normalized $C^{2,\gamma}$ solution in the full chart
$w=(a,M,b)$, where $M$ is a Hermitian $2\times2$ field, with
$0<\tau\le\tau_b$ and $q,\beta,V>0$, there is a bound
\[
 \|a\|_\infty+\|b\|_\infty+\|M\|_{\mathrm{op},\infty}\le B.
\]
Here $B$ depends only on the fixed original primary, its reference,
background, mass, and interval, and is independent of $\tau$, the
solution, and its symbol minima.
\end{theorem}

The issue is the possible drift of the negative middle eigenvalue as
$\tau$ tends to zero. One-sided Laplacian estimates first separate that
drift from spatial fluctuations. The matrix equation then restricts every
normalized escape to two alternatives. A scalar--matrix reciprocal
inequality excludes both alternatives using the original positive source
budgets.

\subsection{Ordered rows and total source budgets}

In the reference Hermitian frame put
\begin{align*}
 f&=e^M,& S&=f^{1/2},\\
 R&=e^aS^{-1}R_*S^{-1},&P&=e^{-b}SP_*S,\\
 p&=\operatorname{tr}R,&r&=\operatorname{tr}P,\\
 q&=p-\alpha,&\beta&=d-r,&V&=cI-R+P.
\end{align*}
The reference arrow matrices are positive of rank one, with
\[
 p_*=\operatorname{tr}R_*=\frac9{14},\qquad
 0<r_{*,\min}\le r_*(x)=\operatorname{tr}P_*(x)\le\frac3{10}.
\]
Their reference symbols satisfy
\[
 q_*=\frac{10}{21},\qquad \beta_*\ge\frac{23}{15},
 \qquad V_*\ge\delta I,\qquad\delta=\frac4{21},
 \qquad q_*+\operatorname{tr}V_*+\beta_*=W.
\]
In particular, $r_{*,\min}>0$ follows from the pointwise nonvanishing of
the original lower arrow, with no division by its composition with $\phi$.

The actual rows are
\begin{align}
 U_0&:=H_0-\alpha+ta=\frac{C_0^\tau+K_0}{q},\nonumber\\
 D&:=\widehat H_1-\alpha I+tM
      =\mathcal L_V^{-1}(C_1^\tau+K_1),\label{primary-rows}\\
 U_2&:=H_2-\alpha+tb=\frac{C_2^\tau+K_2}{\beta},\nonumber
\end{align}
where $\widehat H_1=SH_1S^{-1}$,
$\mathcal L_V(A)=(VA+AV)/2$, and
\[
 C_j^\tau=A_0I_j+\tau F_{*,j},\qquad
 C_0^\tau\ge\frac{55}{63},\qquad
 C_2^\tau\ge A_0,\qquad C_1^\tau\ge\frac{A_0}{2}I.
\]
Every $F_{*,j}$ is the fixed smooth reference residual. The actual
ordered arrow derivatives $X,Y$ of \ref{common-fields} give
\[
 K_0=\frac{XX^\dagger}{2g},\qquad
 K_1=\frac{X^\dagger X+YY^\dagger}{2g},\qquad
 K_2=\frac{Y^\dagger Y}{2g},\qquad
 \operatorname{tr}K_1=K_0+K_2.
\]
For each individual $V>0$ the inverse has the ordered representation
\begin{equation}\label{primary-lyapunov}
 \mathcal L_V^{-1}(A)
   =2\int_0^\infty e^{-sV}Ae^{-sV}\,ds,
 \qquad \mathcal L_V^{-1}(I)=V^{-1}.
\end{equation}
It preserves positive semidefiniteness. A uniform inverse norm is not
needed here.

The size estimate of \ref{common-positive-size} supplies
\begin{equation}\label{primary-size}
 t|a|,\quad t|b|,\quad t\|M\|_{\mathrm{op}}
       \le\mathcal C,\qquad\mathcal C=C_{\mathrm{size}}.
\end{equation}
Positive symbols and rank-one arrow values imply
\begin{equation}\label{primary-symbol-upper}
 \alpha<p<c+r<\frac83,\quad 0<r<d,\quad
 0<q<\frac52,\quad0<\beta<d,\quad
 0<V\le\frac83I,\quad q+\operatorname{tr}V+\beta=W.
\end{equation}
The exact trace comparison and upper bound of
\ref{common-trace-comparison} and \ref{common-trace-bound} are
\begin{align}
 T&=-\alpha a+c\operatorname{tr}M+db+p+r-p_*-r_*(x),\nonumber\\
 T&\le C_T,\qquad
 T\ge q_*a+\operatorname{tr}(V_*M)+\beta_*b.\label{primary-trace}
\end{align}
The normalization is $\int_Xv\,\omega_X=0$, where
$v=qa+\operatorname{tr}(VM)+\beta b$.

All unmarked integrals below use $\omega_X$.
The degree identities in the same curvature convention are
\[
 \int_X H_0=3\mathcal V,\qquad
 \int_X\operatorname{tr}\widehat H_1=4\mathcal V,\qquad
 \int_XH_2=\mathcal V.
\]
Integrating the positive rows and using \eqref{primary-size} gives
\begin{align}
 \int_X U_0&\le J_0=\left(\frac{17}{6}+\mathcal C\right)\mathcal V,
 &\int_X U_2&\le J_2=(c+\mathcal C)\mathcal V,\nonumber\\
 \int_X\operatorname{tr}D&\le
 J_1=\left(\frac{11}{3}+2\mathcal C\right)\mathcal V.\label{primary-source-budgets}
\end{align}
In particular,
\begin{equation}\label{primary-endpoint-reciprocals}
 \int_X\left[\frac{55/63}{q}+\frac{K_0}{q}\right]\le J_0,
 \qquad
 \int_X\left[\frac{A_0}{\beta}+\frac{K_2}{\beta}\right]\le J_2.
\end{equation}
For the middle row, cyclic trace makes $\mathcal L_V$ self-adjoint for
the real trace pairing. Since $\mathcal L_V(V^{-1})=I$,
\[
 \operatorname{tr}\mathcal L_V^{-1}(A)=\operatorname{tr}(V^{-1}A).
\]
This is a trace identity; the factors in the matrix inverse remain ordered.
Thus
\begin{equation}\label{primary-middle-reciprocals}
 \int_X\left[\frac{A_0}{2}\operatorname{tr}(V^{-1})
                 +\operatorname{tr}(V^{-1}K_1)\right]\le J_1.
\end{equation}
The estimates bound all bare reciprocal symbols and all arrow kinetics:
\[
 \int_XK_0\le\frac52J_0,\quad \int_XK_2\le dJ_2,\quad
 \int_X\operatorname{tr}K_1\le\frac83J_1.
\]

Put $C_H=(\mathcal C-\alpha)_+$. The rows imply that each actual
metric-unitary curvature is bounded below by $-C_HI$. For a Hermitian
matrix $H$ of rank $k$ with this lower bound,
$\operatorname{tr}|H|\le\operatorname{tr}H+2kC_H$. Consequently
\begin{equation}\label{primary-curvature-budgets}
 \int_X|H_0|\le(3+2C_H)\mathcal V,\quad
 \int_X\operatorname{tr}|\widehat H_1|\le(4+4C_H)\mathcal V,\quad
 \int_X|H_2|\le(1+2C_H)\mathcal V.
\end{equation}
These are total-mass estimates in the current metric. They do not assert
uniform integrability of a curvature multiplied by an arrow coefficient.

\subsection{Negative potentials and the remaining mean}

Set
\[
 \nu_0=-a,\qquad\nu_2=-b,\qquad
 \nu_1=\log\operatorname{tr}(e^{-M}).
\]
Apply Lemmas~\ref{common-negative-propagation} and
\ref{common-scalar-fluctuations} at this fixed reference, mass and
positive-force slab, with $u=\nu_1$ and $\mathcal C=C_{\mathrm{size}}$.
Their endpoint curvature and ordered inverse-metric calculations give
\begin{equation}\label{primary-one-sided-endpoints}
 \Delta\nu_0\ge-L_0,\quad
 L_0=\|H_{*,0}-\alpha\|_\infty+\mathcal C,\qquad
 \Delta\nu_2\ge-L_2,\quad L_2=c+\mathcal C.
\end{equation}
For the middle field, let $h=f^{-1}$. The ordered inverse-metric identity is
\begin{equation}\label{primary-inverse-metric}
 \Delta h=g^{-1}hf_zhf_{\bar z}h-dh-tMh
       +S^{-1}\mathcal L_V^{-1}(C_1^\tau+K_1)S^{-1}.
\end{equation}
The last term is positive. The weighted Hilbert--Schmidt inequality
proved in Lemma~\ref{common-negative-propagation}, with
$\rho_h=\operatorname{tr}h$, controls the full derivative term.
Together with \eqref{primary-size}, it gives
\begin{equation}\label{primary-one-sided-middle}
 \Delta\log\rho_h
 =\frac{\Delta\rho_h}{\rho_h}
   -g^{-1}\frac{|(\rho_h)_z|^2}{\rho_h^2}
 \ge-L_1,\qquad L_1=d+\mathcal C.
\end{equation}
This uses commutation only between $M$ and its exponential.

For a real scalar $u$, write
$\overline u=\mathcal V^{-1}\int_Xu\,\omega_X$.
Use the chordal kernel $G_0$ of \eqref{common-green-kernel}.
Put $n_\omega=\max_X(\omega_X/\omega_{\rm FS})$ and
\[
 H_X(x,y)=G_0(x,y)-\frac1{\mathcal V}
                       \int_XG_0(z,y)\,\omega_X(z).
\]
For $s=\Delta_Xu$, its mean-adjusted Green identity on the actual
background is
\begin{equation}\label{primary-green-identity}
 u(x)-\overline u
    =\int_XH_X(x,y)s(y)\,\omega_X(y).
\end{equation}
Lemma~\ref{common-negative-propagation} therefore yields, for
$\Delta_Xu\ge-L$,
\begin{equation}\label{primary-propagation}
 \sup_Xu\le\overline u+3Ln_\omega.
\end{equation}
Both the equation and the mean use $\omega_X$, not the Fubini--Study
background used to express the kernel.

The kernel-moment calculation in Lemma~\ref{common-scalar-fluctuations}
also gives, for each integer $k\ge1$,
\[
 \sup_y\|H_X(\cdot,y)\|_{L^k(\omega_X)}
 \le C_k=2n_\omega^{1/k}(2\pi)^{1/k-1}(k!)^{1/k}.
\]
Its Green mapping estimate reads
\begin{equation}\label{primary-green-lp}
 \|u-\overline u\|_{L^k(\omega_X)}
       \le C_k\|\Delta_Xu\|_{L^1(\omega_X)}.
\end{equation}
Every fixed finite real $p\ge1$ follows by H\"older's inequality with
an integer $k\ge p$.

Equations \eqref{primary-one-sided-endpoints} and
\eqref{primary-one-sided-middle} bound the Laplacian total variations of
$a,b,\nu_1$. The flat reference also gives
\[
 \Delta\operatorname{tr}M=4-\operatorname{tr}\widehat H_1,\qquad
 \|\Delta\operatorname{tr}M\|_1\le(8+4C_H)\mathcal V.
\]
Applying the same Green mapping to this additional scalar identity,
each of $a,b,\operatorname{tr}M,\nu_1$ has a uniformly bounded
$L^p$ fluctuation about its actual mean, for every fixed finite $p$.

One mean controls all negative values. At a point let
$m=\lambda_{\min}(M)$. Functional calculus for the single matrix $M$
and cyclic trace give
\[
 p=e^a\operatorname{tr}(R_*e^{-M})\le p_*e^{a-m},\qquad
 r=e^{-b}\operatorname{tr}(P_*e^M)\ge r_{*,\min}e^{m-b}.
\]
Since $p>\alpha$ and $r<d$,
\begin{equation}\label{primary-endpoint-lower}
 a>m-\log(p_*/\alpha),\qquad
 b>m-\log(d/r_{*,\min}).
\end{equation}
If $\overline{\nu_1}\le B_1$, then
\eqref{primary-propagation} gives
$m\ge-D_1$, where $D_1=B_1+3L_1n_\omega$.
For
\[
 C_-=\max\bigl(0,D_1+\max(0,\log(p_*/\alpha),
                                  \log(d/r_{*,\min}))\bigr)
\]
we have $a,b\ge-C_-$ and $M\ge-C_-I$. Shift all three fields by
this same constant. The reference comparison \eqref{primary-trace} gives
\begin{equation}\label{primary-value-reconstruction}
 \delta\bigl[(a+C_-)+\operatorname{tr}(M+C_-I)+(b+C_-)\bigr]
       \le C_T+C_-W.
\end{equation}
All summands and eigenvalues on the left are nonnegative, so this bounds
every upper value as well. For example,
$C_-+(C_T+C_-W)/\delta$ is a sufficient bound for each absolute value.
Conversely $\|M\|_{\mathrm{op},\infty}\le B$ implies
$\nu_1\le B+\log2$. We have therefore proved, on this actual solution
family,
\begin{equation}\label{primary-one-mean}
 \text{uniform full values}\quad\Longleftrightarrow\quad
          \sup\overline{\log\operatorname{tr}(e^{-M})}<\infty.
\end{equation}

\subsection{Normalized escape and its matrix projection}

Suppose that the value bound fails. By \eqref{primary-one-mean}, choose
actual solutions with
$\Lambda_n=\overline{\nu_{1,n}}\longrightarrow+\infty$.
Propagation, \eqref{primary-endpoint-lower}, and
\eqref{primary-value-reconstruction} imply
\begin{equation}\label{primary-normalized-values}
 \|a_n\|_\infty+\|b_n\|_\infty+\|M_n\|_{\mathrm{op},\infty}
            \le C(1+\Lambda_n).
\end{equation}
The size estimate gives
$t_n(\Lambda_n-\log2)\le\mathcal C$. It follows that
$t_n\to0$ and $\tau_n\to0$, since the mass is fixed and finite.
After a subsequence,
\begin{equation}\label{primary-mass-limit}
 t_n\Lambda_n\longrightarrow\rho,\qquad0\le\rho\le\mathcal C.
\end{equation}

Extract limits of the three bounded real mean ratios for $a_n,b_n$ and
$\operatorname{tr}M_n$. Their finite-$L^p$ fluctuations, together with
\[
 \nu_{1,n}=-m_n+\epsilon_n^{\rm eig},\qquad
 m_n=\lambda_{\min}M_n,\quad
 \ell_n=\lambda_{\max}M_n,\quad
 0\le\epsilon_n^{\rm eig}=\log(1+e^{-(\ell_n-m_n)})\le\log2,
\]
give constants $A,B,\vartheta$ such that
\begin{equation}\label{primary-scalar-limits}
 \frac{a_n}{\Lambda_n}\to A,\quad
 \frac{b_n}{\Lambda_n}\to B,\quad
 \frac{m_n}{\Lambda_n}\to-1,\quad
 \frac{\ell_n}{\Lambda_n}\to\vartheta
\end{equation}
strongly in every finite $L^p$ and almost everywhere after a further
subsequence. The same subsequence works for all finite $p$: only finitely
many real means were extracted, and the fluctuations are bounded for
every integer exponent. In addition to \eqref{primary-endpoint-lower},
\[
 p\ge p_*e^{a-\lambda_{\max}M},\qquad p<\frac83
\]
gives $a\le\lambda_{\max}M+\log(112/27)$. Hence
\begin{equation}\label{primary-scalar-range}
 -1\le A\le\vartheta,\qquad B\ge-1,\qquad\vartheta\ge-1.
\end{equation}

The eigenvalues alone do not control the middle eigendirections. The
flat reference End connection yields
\[
 H_1=2I-g^{-1}\partial_{\bar z}(f^{-1}f_z),\qquad
 \widehat H_1=SH_1S^{-1}.
\]
Integration by parts, retaining $M_z$ as a full complex matrix, gives
\begin{equation}\label{primary-log-energy}
 \mathcal E(M):=\int_Xg^{-1}\operatorname{tr}(M_{\bar z}f^{-1}f_z)
       =\int_X\operatorname{tr}\bigl(M(\widehat H_1-2I)\bigr).
\end{equation}
At a point diagonalize $M$ alone, with eigenvalues $m_i$. Then
\begin{equation}\label{primary-energy-coefficients}
 \operatorname{tr}(M_{\bar z}f^{-1}f_z)
    =\sum_{i,j}k_{\exp}(m_j-m_i)|(M_z)_{ij}|^2,\qquad
 k_{\exp}(s)=\frac{e^s-1}{s},\quad k_{\exp}(0)=1.
\end{equation}
All coefficients are positive, including repeated spectra, and
$k_{\exp}(s)\ge(1+|s|)^{-1}$. The row implies
$\widehat H_1-2I\ge-(d+\mathcal C)I$, while its trace has total integral
zero. Therefore
\[
 \int_X\operatorname{tr}|\widehat H_1-2I|
       \le4(d+\mathcal C)\mathcal V.
\]
Writing $\Gamma=\|M\|_{\mathrm{op},\infty}$ in
\eqref{primary-log-energy}, we obtain
\begin{equation}\label{primary-normalized-energy}
 0\le\mathcal E(M)\le C\Gamma,\qquad
 \int_X|DM|^2\le C\Gamma(1+2\Gamma).
\end{equation}
It follows from \eqref{primary-normalized-values} that
$M_n/\Lambda_n$ is bounded in $H^1$.

For clarity, the compactness needed here can be seen directly in fixed
coordinate cutoffs. Local mollification gives
$\|u-\eta_\varepsilon*u\|_2\le C\varepsilon\|Du\|_2$.
For fixed $\varepsilon$ the convolutions have bounded derivatives, and
finite-grid extraction gives a uniformly convergent subsequence. A
diagonal choice as $\varepsilon\downarrow0$ gives strong $L^2$
compactness. Consequently, after a subsequence,
\begin{equation}\label{primary-matrix-limit}
 M_n/\Lambda_n\longrightarrow U_\infty
       \quad\text{strongly in }L^2\text{ and weakly in }H^1.
\end{equation}
Its eigenvalues are the constants $-1,\vartheta$ of
\eqref{primary-scalar-limits}.

\subsection{Excluding isotropic escape}

First suppose $\vartheta=-1$. Then $U_\infty=-I$ and
$A=-1$ by \eqref{primary-scalar-range}. The exact radial identity is
\begin{align}
 Q:=v-T
  &=\operatorname{tr}\bigl(R_*k(aI-M)\bigr)
     +\operatorname{tr}\bigl(P_*k(M-bI)\bigr)\ge0,\nonumber\\
 k(y)&=ye^y-e^y+1.\label{primary-entropy}
\end{align}
In this regime $Q_n/\Lambda_n\to0$ almost everywhere. For each upper
spectral term, $(a_n-m_{i,n})/\Lambda_n\to0$ and its exponential sum
is $p_n<8/3$. For each lower term the normalized argument tends either
to zero or to a strictly negative constant, and its exponential sum is
$r_n<d$. These facts make every reference-weighted spectral summand in the
displayed entropy traces, divided by $\Lambda_n$, tend to zero.
Equation \eqref{primary-normalized-values} and the bounded arrow values
also bound $Q_n/\Lambda_n$ uniformly, so its mean tends to zero.
Using $\overline v_n=0$ and \eqref{primary-trace}, we get
\[
 0=\lim_n\overline{T_n}/\Lambda_n=\alpha-2c+dB,\qquad
 B=\frac9{11}>0.
\]

This forces uniform collapse of the lower arrow. The pointwise upper
trace estimate and $a_n\le\ell_n+\log(112/27)$ give
\begin{equation}\label{primary-isotropic-trace}
 (c-\alpha)\ell_n+cm_n+db_n\le C.
\end{equation}
Propagation gives $m_n\ge-\Lambda_n-C$ and
$b_n\ge\overline b_n-C$. Since
$\overline b_n=(9/11)\Lambda_n+o(\Lambda_n)$,
\eqref{primary-isotropic-trace} implies uniformly
$\ell_n\le-\Lambda_n+o(\Lambda_n)$. Therefore
\begin{equation}\label{primary-isotropic-arrow-collapse}
 r_n\le r_{*,\max}e^{\ell_n-b_n}
 \le C\exp\left[-\left(\frac{20}{11}+o(1)\right)\Lambda_n\right]
       \longrightarrow0\quad\text{uniformly}.
\end{equation}
The exact lower Bochner identity is
\[
 2\int_XK_{2,n}
       =\int_X\operatorname{tr}(P_n\widehat H_{1,n})
           -\int_Xr_nH_{2,n}.
\]
The current-metric curvature budgets
\eqref{primary-curvature-budgets}, together with the uniform $r_n$
bound in \eqref{primary-isotropic-arrow-collapse}, show that
$\int_XK_{2,n}\to0$.
Integrate the original bottom row after multiplying it by
$\beta_n=d-r_n$. The $r_n$ terms tend to zero by the same curvature
budgets and $t_n|b_n|\le\mathcal C$, and
$C_2^{\tau_n}\to A_0$. Thus
\[
 dc\mathcal V+d\,t_n\int_Xb_n=A_0\mathcal V+o(1),\qquad
 t_n\overline b_n\longrightarrow\frac{A_0-dc}{d}=-\frac3{11}.
\]
But $\overline b_n>0$ eventually and $t_n>0$. This contradiction
excludes isotropic escape. Hence
\begin{equation}\label{primary-anisotropic}
 \vartheta>-1.
\end{equation}

\subsection{The moving spectral test and exhaustive alternatives}

The limiting eigenspace identifies the possible escaping line, but the
curvature estimate must still be tested against the current matrix.
Testing first and passing to the limit afterwards retains the kinetic
terms that would be lost by replacing a product with the product of its
weak limits. The smooth spectral cutoff also avoids choosing an
eigenvector at repeated eigenvalues.

Write $U_\infty=-\Pi+\vartheta(I-\Pi)$. Then $\Pi$ is a rank-one
Hermitian $H^1$ projection almost everywhere. Choose one smooth
nonincreasing cutoff $\chi$, with $0\le\chi\le1$, equal to $1$ near
$-1$ and $0$ near $\vartheta$, in disjoint neighborhoods, and set
\[
 E_n=\chi(M_n/\Lambda_n).
\]
Divided differences bound $|DE_n|$ by
$C|D(M_n/\Lambda_n)|$. Thus $E_n$ is $H^1$ bounded and converges
strongly in $L^2$ to $\Pi$. It is globally defined at repeated spectra,
$0\le E_n\le I$, and commutes with $M_n$ and $S_n$.

Let the good set be where the normalized eigenvalues lie in the two
flat cutoff neighborhoods. There $E_n$ is the exact low projection and
$\Delta_n=\ell_n-m_n\ge\kappa\Lambda_n$ for a fixed $\kappa>0$.
The $12$ term of \eqref{primary-energy-coefficients} and
\eqref{primary-normalized-energy} give
\[
 \int_{\rm good}g^{-1}|(E_n)_z{}_{12}|^2
       \le\frac{C}{\kappa(e^{\kappa\Lambda_n}-1)}\longrightarrow0.
\]
Hermiticity turns this into vanishing of $(I-E_n)(E_n)_{\bar z}$ on
the good set. The bad set has measure tending to zero; its $L^1$
contribution is bounded by its square-root measure times the uniform
$H^1$ bound. Strong $L^2$ convergence of $E_n$ and weak convergence of
its derivative then give
\begin{equation}\label{primary-weak-holomorphic}
 (I-\Pi)\bar\partial\Pi=0
       \quad\text{distributionally}.
\end{equation}

The arrow values impose two further constraints:
\begin{equation}\label{primary-arrow-alignment}
 A>-1\ \Longrightarrow\ \phi\Pi=0,\qquad
 B<\vartheta\ \Longrightarrow\ (I-\Pi)\psi=0.
\end{equation}
For example, on the good set,
\[
 \|\phi E_n\|_*^2
       \le p_n e^{-(a_n-m_n)},\qquad
 \|(I-E_n)\psi\|_*^2
       \le r_n e^{-(\ell_n-b_n)}.
\]
The appropriate exponent has a strictly positive normalized limit in
each implication. Passing along the almost-everywhere limits proves
\eqref{primary-arrow-alignment}. Both implications cannot hold at once:
they would give $\phi\psi=0$ almost everywhere and hence identically.
This uses the nonzero holomorphic composition, including its possible
pointwise zeros.
If $A>-1$, we therefore have $B\ge\vartheta$ and $\Pi$ is the smooth
reference-orthogonal projection onto $\ker\phi$. If $B<\vartheta$,
$\Pi$ is the smooth projection onto $\operatorname{im}\psi$.
Both line bundles have degree one: the first is the kernel of the
surjection $\mathcal O(2)^2\to\mathcal O(3)$ and the second is the
injected $\mathcal O(1)$.

The curvature must be tested with the moving $E_n$. Its commutation with
$S_n$ gives the exact integration-by-parts identity
\begin{equation}\label{primary-spectral-test}
 \int_X\operatorname{tr}\bigl(E_n(\widehat H_{1,n}-2I)\bigr)
    =\int_Xg^{-1}\operatorname{tr}\bigl((E_n)_{\bar z}
                                         f_n^{-1}(f_n)_z\bigr).
\end{equation}
In an $M_n$ eigenbasis the coefficient of $|(M_n)_z{}_{ij}|^2$ is
the divided difference of $x\mapsto\chi(x/\Lambda_n)$ multiplied by
$k_{\exp}(m_j-m_i)$. It is nonpositive. Repeated spectra use its
continuous diagonal value. In particular, for every $n$,
\begin{equation}\label{primary-spectral-coarse}
 \int_X\operatorname{tr}(E_n\widehat H_{1,n})
       \le2\int_X\operatorname{tr}E_n.
\end{equation}
On the good set the $21$ term can be retained to give
\[
 \int_X\operatorname{tr}\bigl(E_n(\widehat H_{1,n}-2I)\bigr)
 \le-\int_{\rm good}g^{-1}(1-e^{-\Delta_n})|(E_n)_z{}_{21}|^2.
\]
The omitted bad-set terms have the same nonpositive sign. To justify
the lower semicontinuity with moving good sets, let
$Z_n=(I-E_n)(E_n)_zE_n$. This is bounded in $L^2$. Strong $L^2$
convergence of the bounded factors and weak convergence of the
derivatives give $Z_n\rightharpoonup(I-\Pi)\Pi_z\Pi$. Multiplication
by the good-set indicator preserves this weak limit: its bad-set
pairing with a fixed $L^2$ test is bounded by the uniform $L^2$ norm
times the test's $L^2$ norm on a set of vanishing measure. On the good
set $Z_n$ is exactly the $21$ derivative component, and
$1-e^{-\Delta_n}\ge1-e^{-\kappa\Lambda_n}\to1$. Weak $L^2$ lower
semicontinuity therefore yields
\begin{equation}\label{primary-spectral-energy}
 \limsup_n\int_X\operatorname{tr}(E_n\widehat H_{1,n})
       \le2\mathcal V-\mathfrak E_\Pi,\qquad
 \mathfrak E_\Pi=\int_Xg^{-1}|(I-\Pi)\Pi_z\Pi|^2.
\end{equation}
For either smooth degree-one line above, the projected reference
curvature formula is
$H_L=2-g^{-1}|(I-\Pi)\Pi_z\Pi|^2$.
Its degree integral gives $\mathfrak E_\Pi=\mathcal V$, so
\begin{equation}\label{primary-spectral-degree-one}
 \limsup_n\overline{\operatorname{tr}(E_n\widehat H_{1,n})}\le1
       \quad\text{for these two line projections}.
\end{equation}
No regularity or degree quantization for a general weak projection is
needed in this step.

Suppose now $B>-1$. On the good set,
$\operatorname{tr}(E_nP_n)\le r_{*,\max}e^{m_n-b_n}\to0$
almost everywhere; on the bad set use $0\le\operatorname{tr}(E_nP_n)<d$.
Dominated convergence gives
$\int_X\operatorname{tr}(E_nP_n)\to0$.
Since $V_n\le cI+P_n$, inverse order and the rank-one inverse formula give
\[
 V_n^{-1}\ge(cI+P_n)^{-1}
       =\frac1cI-\frac{P_n}{c(c+r_n)},\qquad
 \liminf_n\int_X\operatorname{tr}(E_nV_n^{-1})\ge\mathcal V/c.
\]
The fixed force obeys
\[
 \mathcal L_{V_n}^{-1}(C_1^{\tau_n})
  =A_0V_n^{-1}+\tau_n\mathcal L_{V_n}^{-1}(F_{*,1}).
\]
By order preservation the absolute value of the integrated $E_n$-tested
error is at most
\[
 \tau_n\|F_{*,1}\|_\infty\int_X\operatorname{tr}(V_n^{-1}),
\]
which tends to zero by \eqref{primary-middle-reciprocals}.
The kinetic inverse is positive and can be discarded in a lower bound.
Finally,
\[
 \overline{\operatorname{tr}E_n}\to1,\qquad
 \overline{\operatorname{tr}(E_nM_n/\Lambda_n)}\to-1.
\]
The row and \eqref{primary-mass-limit} consequently imply
\begin{equation}\label{primary-spectral-source-lower}
 \liminf_n\overline{\operatorname{tr}(E_n\widehat H_{1,n})}
       \ge\alpha+\rho+A_0/c=\frac75+\rho.
\end{equation}
This contradicts \eqref{primary-spectral-degree-one} whenever $\Pi$
is either of those degree-one projections.
In particular $A>-1$ is impossible, since then
$B\ge\vartheta>-1$ and $\Pi=\ker\phi$. Likewise
$-1<B<\vartheta$ is impossible, since then
$\Pi=\operatorname{im}\psi$.
Together with \eqref{primary-scalar-range} and
\eqref{primary-anisotropic}, these exclusions leave exactly
\begin{equation}\label{primary-residual-alternatives}
 \vartheta>-1,\quad A=-1,\qquad
 \begin{cases}
  B=-1,\quad\Pi=\operatorname{im}\psi,&\text{case I},\\
  B\ge\vartheta,&\text{case II}.
 \end{cases}
\end{equation}

\subsection{Reciprocal allocations exclude both alternatives}

The two remaining alternatives differ in the limiting position of the
lower arrow. Each must nevertheless satisfy the same total source
budgets. The following scalar and matrix Cauchy inequalities compare
those budgets with the available symbol means, producing a contradiction
in each case without discarding the moving projection.

We continue to test the actual moving $E_n$. Set
\[
 \varepsilon_n=\tau_n\max_j\|F_{*,j}\|_{\infty,\mathrm{op}}\to0.
\]
For all sufficiently large $n$, $\varepsilon_n<A_0$, and the forces
can be bounded before integration. From
\eqref{primary-rows} and \eqref{primary-lyapunov},
\begin{equation}\label{primary-common-source-lower}
 U_{0,n}\ge\frac{A_0-\varepsilon_n}{q_n},\qquad
 D_n\ge(A_0-\varepsilon_n)V_n^{-1},\qquad
 U_{2,n}\ge\frac{A_0-\varepsilon_n}{\beta_n}.
\end{equation}
The exact endpoint degree identities give
\[
 \overline{U_{0,n}}=\frac{17}{6}+t_n\overline a_n
        \longrightarrow\frac{17}{6}-\rho,\qquad
 \overline{U_{2,n}}=c+t_n\overline b_n\longrightarrow c+\rho B.
\]
The moving curvature estimates
\eqref{primary-spectral-coarse} and
\eqref{primary-spectral-degree-one} give respectively
\begin{equation}\label{primary-moving-source-upper}
 \limsup_n\overline{\operatorname{tr}(E_nD_n)}
 \le\begin{cases}c-\rho,&\text{case I},\\d-\rho,&\text{case II}.
 \end{cases}
\end{equation}
These follow directly from
$D_n=\widehat H_{1,n}-\alpha I+t_nM_n$ and the preceding bounded
value limits. They never replace a moving curvature probe by its weak
limit multiplied by curvature.

Since $A=-1<\vartheta$, on the good set
\[
 \operatorname{tr}((I-E_n)R_n)
       \le p_*e^{a_n-\ell_n}\longrightarrow0
       \quad\text{almost everywhere}.
\]
Its global bound is $p_n<8/3$, and the bad set has vanishing measure.
Thus
\begin{equation}\label{primary-upper-high-value}
 \overline{\operatorname{tr}((I-E_n)R_n)}\longrightarrow0.
\end{equation}
In case II, $B\ge\vartheta>-1$, so the same argument with
$e^{m_n-b_n}$ and $r_n<d$ gives
\begin{equation}\label{primary-lower-low-value}
 \overline{\operatorname{tr}(E_nP_n)}\longrightarrow0
       \quad\text{in case II}.
\end{equation}
In case I the high component of $P_n$ is retained, with its sign,
in the allocation below.

For a positive scalar field $s$, a positive Hermitian matrix field $V$,
and a positive semidefinite field $E$, integrated Cauchy--Schwarz gives
\begin{equation}\label{primary-reciprocal-allocation}
 (1+\overline{\operatorname{tr}E})^2
 \le\overline{s+\operatorname{tr}(EV)}\,
       \overline{s^{-1}+\operatorname{tr}(EV^{-1})}.
\end{equation}
To see the matrix order explicitly, use the direct sum of scalar and
matrix $L^2$ spaces with
\[
 \mathcal A=(\sqrt s,V^{1/2}E^{1/2}),\qquad
 \mathcal B=(s^{-1/2},V^{-1/2}E^{1/2}).
\]
Their inner product has integrand $1+\operatorname{tr}E$, and their
squared norms have integrands
$s+\operatorname{tr}(EV)$ and
$s^{-1}+\operatorname{tr}(EV^{-1})$. After division by $\mathcal V^2$
this is \eqref{primary-reciprocal-allocation}.
The matrices $E,V$ need not commute and $E$ need not be a projection.
Scalar Cauchy--Schwarz also gives
\begin{equation}\label{primary-scalar-reciprocal}
 1\le\overline s\;\overline{s^{-1}}.
\end{equation}
All integrals used here are finite by the original budgets
\eqref{primary-endpoint-reciprocals} and
\eqref{primary-middle-reciprocals}.

In case II the exact value identity is
\[
 q_n+\operatorname{tr}(E_nV_n)
   =c\operatorname{tr}E_n-\alpha
      +\operatorname{tr}((I-E_n)R_n)+\operatorname{tr}(E_nP_n).
\]
Equations \eqref{primary-upper-high-value} and
\eqref{primary-lower-low-value} imply
\[
 \overline{q_n+\operatorname{tr}(E_nV_n)}\longrightarrow
                 c-\alpha=\frac23.
\]
Apply \eqref{primary-reciprocal-allocation} with $s=q_n$ and then
\eqref{primary-common-source-lower}. We obtain
\[
 \overline{U_{0,n}+\operatorname{tr}(E_nD_n)}
 \ge(A_0-\varepsilon_n)
       \frac{(1+\overline{\operatorname{tr}E_n})^2}
       {\overline{q_n+\operatorname{tr}(E_nV_n)}}.
\]
Its lower limit is at least $4A_0/(c-\alpha)=37/6$.
The actual degree and moving-source upper bounds instead give
\[
 \limsup_n\overline{U_{0,n}+\operatorname{tr}(E_nD_n)}
       \le\frac{17}{6}+d-2\rho
       =\frac{14}{3}-2\rho\le\frac{28}{6}<\frac{37}{6}.
\]
This excludes case II without assigning a degree to its weak projection.

In case I the combined value identity is
\[
 q_n+\beta_n+\operatorname{tr}(E_nV_n)
  =c\operatorname{tr}E_n+d-\alpha
    +\operatorname{tr}((I-E_n)R_n)-\operatorname{tr}((I-E_n)P_n).
\]
The last trace is nonnegative since $I-E_n,P_n\ge0$, whether or not
they commute. Thus
\begin{equation}\label{primary-case-i-allocation-upper}
 \limsup_n\overline{q_n+\beta_n+\operatorname{tr}(E_nV_n)}
       \le c+d-\alpha=\frac52.
\end{equation}
Before dividing, note that $\beta_n<d$ and
\eqref{primary-common-source-lower} imply
\[
 c-\rho=\lim_n\overline{U_{2,n}}
       \ge A_0/d>0.
\]
Hence $0\le\rho<c$. The actual source upper bounds now imply
\[
 \limsup_n\overline{q_n^{-1}}\le\frac{17/6-\rho}{A_0},\qquad
 \limsup_n\overline{\beta_n^{-1}+\operatorname{tr}(E_nV_n^{-1})}
       \le\frac{2(c-\rho)}{A_0}.
\]
The common coefficient $A_0-\varepsilon_n$ in
\eqref{primary-common-source-lower} controls the force error before
these limits are taken. Both denominators are positive.
Apply \eqref{primary-scalar-reciprocal} to $q_n$ and
\eqref{primary-reciprocal-allocation} to $\beta_n$ to obtain
\begin{align*}
 \liminf_n\overline q_n&\ge\frac{A_0}{17/6-\rho},\\
 \liminf_n\overline{\beta_n+\operatorname{tr}(E_nV_n)}
        &\ge\frac{2A_0}{c-\rho}.
\end{align*}
Both right sides increase with $\rho\in[0,c)$. Already at $\rho=0$
their sum is
\begin{equation}\label{primary-case-i-contradiction}
 \frac{A_0}{17/6}+\frac{2A_0}{c}
   =\frac{37}{102}+\frac{37}{15}
   =\frac{481}{170}=\frac52+\frac{28}{85}>\frac52.
\end{equation}
This contradicts \eqref{primary-case-i-allocation-upper}.
The proof uses only bounded value limits and finite source budgets;
the high-$P$ term has its original favorable sign throughout.

\begin{proof}[Conclusion of Theorem \ref{primary-value-bound}]
Every unbounded sequence would admit the normalized extraction above.
Isotropic escape is excluded by the bottom row, the two degree-one
spectral exclusions reduce anisotropic escape exactly to
\eqref{primary-residual-alternatives}, and the reciprocal allocations
exclude both alternatives. Thus the middle mean in
\eqref{primary-one-mean} is bounded uniformly. The reconstruction
\eqref{primary-value-reconstruction} bounds both endpoints and both
middle eigenvalues, proving the full field bound.
\end{proof}

\subsection{The zero-parameter limit for each primary chain}

\begin{theorem}[Primary existence]\label{primary-existence}
Every original primary chain admits a smooth positive Hermitian metric
triple solving the frozen reduced dHYM equation at
$\theta_4=\pi-\arctan6$ and satisfying its strict calibration inequality
in all three blocks.
\end{theorem}

\begin{proof}
Fix the chain. Keep its same reference, one exponent $\gamma\in(0,1)$,
and the single prospectively chosen finite strong mass $\mu_s$.
The construction of \ref{common-positive-continuation} gives actual
normalized smooth positive-symbol solutions at every $\tau\in(0,1]$.
Theorem \ref{primary-value-bound} bounds every such solution on
$(0,\tau_b]$. On precisely this interval the middle force satisfies
$C_1^\tau\ge A_0I/2$. Therefore the values-conditioned estimate
\ref{common-gradient-gaps}, with $B$ from
Theorem \ref{primary-value-bound} and $c_1=A_0/2$, gives a uniform full
reference gradient bound and positive constants
\[
 q\ge\delta_q,\qquad\beta\ge\delta_\beta,\qquad V\ge\delta_VI
\]
for these same solutions. All its fields, reference residuals, and
normalization agree with \eqref{primary-rows}. The lower reference norm
is positive because the original lower arrow is nonvanishing; a
pointwise nonzero composition is not required.

Choose any sequence $\tau_n\downarrow0$ in $(0,\tau_b]$ and a solution
at each parameter. At the same fixed reference and finite mass, the
preceding values, gradients and symbol gaps satisfy
Theorem~\ref{common-original-endpoint}. It gives a subsequence
converging strongly in $C^{2,\gamma}$ in the original chart,
reconstructs smooth positive metrics on the original bundles with their
maps fixed, and preserves the single normalization. The damping and
reference-forcing terms vanish because the mass is fixed and the values
are bounded. The theorem's ordered endpoint rows and phase calculation
then give the undamped equation and strict calibration in all three
blocks.
\end{proof}

The construction is performed separately for every original primary
coefficient pair. It requires neither a common metric witness for
different chains nor a discriminant or orbit restriction.

\section{The rank-one lower arrow}
\label{lower-setup}

Fix an original chain for which the upper coefficient matrix has rank two,
the lower coefficient matrix has rank one, and the composition is nonzero.
Write the arrows in vertex order as
\[
 \mathcal O(1)\xrightarrow{\psi}\mathcal O(2)^2
 \xrightarrow{\phi}\mathcal O(3).
\]
There is a nonzero constant vector $v$ and a nonzero section
$\sigma\in H^0(\mathcal O(1))$ such that $\psi=v\sigma$.
Its saturated image is the line $L=\mathcal O(2)v$.
Let $p$ be the simple zero of $\sigma$, and let $q_0$ be the simple
zero of $\phi|_L$.  Thus the composition has divisor $p+q_0$.
The arguments below allow $p=q_0$; neither separation nor a change of
the original coefficients is used.  We use $q_0$ for this point to
distinguish it from the positive scalar symbol $q$ and its endpoint
value $q^{(0)}$.

The aim is to bound all six metric logarithms as $\tau\downarrow0$ on
the positive-force slab.  Lemma~\ref{lower-mean-transfer} reduces this
to a bound for the single mean
$\overline{\log\operatorname{tr}(e^{-M})}$.  If it escapes,
Lemma~\ref{lower-residual-extraction} leaves a fixed-line residual along
$L$. The original restriction divisor forces atomic source mass;
subtracting that mass from the source budgets before estimating their
diffuse parts gives
$k^\circ+\rho\le88/233$, where $k^\circ$ is the normalized punctured
upper-kernel kinetic mass and $\rho\ge0$ is the limiting damping magnitude.
The punctured derivative calculation, original angular flux and full
lower-arrow derivative instead give $k^\circ\ge2/3$; see
\eqref{lower-kernel-allocation-bound} and
\eqref{lower-final-incompatibility}.  These incompatible bounds exclude
escape in both divisor configurations.  The resulting global value bound,
not merely a punctured limit, permits the original smooth endpoint.

\subsection{A reference metric that is regular at the lower zero}
\label{lower-reference}

Use the lower-boundary instance of the construction
\eqref{common-reference}, computed for this fixed chain, and write
$h'_j=h_{*,j}$ and $F'_j=F_{*,j}$. The upper components form a basis,
so the quotient denominator is positive everywhere. The scale $s$ uses
the lower norm $\rho_{\mathrm{ref}}=|\psi|^2$ and covariant derivative
$k_{\mathrm{ref}}=|D\psi|^2/2$ in the preliminary metric
$h_{2,\mathrm{base}}=\ell$, including $D\psi(p)$.
The reference has middle curvature $2I$, lower curvature $1$, and flat
endomorphism connection. The symbol bounds of the common reference construction
specialize to
\[
 p_*:=\operatorname{tr}R_* =\frac9{14},\qquad
 r_*:=\operatorname{tr}P_* =\frac{\rho_{\mathrm{ref}}}{s},\qquad
 0\le r_*\le\frac3{10},\qquad r_*(p)=0.
\]
In particular,
\[
 q_*:=p_*-\frac16=\frac{10}{21},\qquad
 \beta_*:=\frac{11}{6}-r_*\ge\frac{23}{15},\qquad
 V_*:=\frac56I-R_*+P_*\ge\frac4{21}I.
\]
For $A_0=37/36$, the lower residual of this reference is
\[
 F'_2=\left(\frac{11}{6}-r_*\right)\frac56
          -\frac{37}{36}-\frac{k_{\mathrm{ref}}}{s}
       =\frac12-
          \frac{(5/6)\rho_{\mathrm{ref}}+k_{\mathrm{ref}}}{s}
       \ge\frac14.
\]
The upper reference calculation gives $C_0^\tau\ge55/63$, and
$C_2^\tau\ge A_0$ for $0\le\tau\le1$.
If $F'_1=0$, put $T_+=1$; otherwise put
\[
 T_+=\min\{1,A_0/(2\|F'_1\|_{\infty})\}.
\]
Then $C_1^\tau=A_0I+\tau F'_1\ge A_0I/2$ on this slab.
Fix $\tau_b=T_+/2$ and choose the single finite mass $\mu_s$
in the order required by the common analytic construction, before
choosing any symbol floors.  All estimates below concern this fixed
chain, reference, mass and slab $0<\tau\le\tau_b$.
There is no positive global lower bound for $r_*$.

\subsection{Recovering all values from one scalar mean}
\label{lower-mean-reduction}

For a normalized positive-symbol solution, let
\[
 f=e^M,\quad S=f^{1/2},\quad
 R=e^aS^{-1}R_*S^{-1},\quad P=e^{-b}SP_*S,
\]
\[
 p_1=\operatorname{tr}R,\quad r=\operatorname{tr}P,\quad
 q=p_1-\alpha,\quad \beta=d-r,\quad V=cI-R+P,
 \qquad \alpha=\frac16,\quad c=\frac56,\quad d=\frac{11}{6}.
\]
In particular $p_1>\alpha$ and $r<d$, while $r(p)=0$.
For a scalar function $g$ write
$\overline g=(2\pi)^{-1}\int_Xg\,\omega_X$.
The parameter-uniform Green and radial estimates used here are
\begin{align}
 u&=\log\operatorname{tr}(e^{-M}),\qquad L_0=\overline u,
       &\sup_Xu&\le L_0+C_u,\label{lower-green-u}\\
 b&\ge\overline b-C_b,
       &\int_X|b-\overline b|\,\omega_X&\le K_b,
          \label{lower-green-b}\\
 q_*a+\operatorname{tr}(V_*M)+\beta_*b&\le T,
       &T&\le C_T.\label{lower-radial-upper}
\end{align}
For solutions of the forced equations,
Lemmas~\ref{common-negative-propagation} and
\ref{common-scalar-fluctuations} give \eqref{lower-green-u} and
\eqref{lower-green-b}; Lemmas~\ref{common-trace-comparison} and
\ref{common-trace-bound} give
\eqref{lower-radial-upper}. Their inputs are the same fixed reference,
mass, positive-force slab and normalization. Neither scalar lemma
divides by $\psi$ or requires a positive minimum of $r_*$.
Here $L_0$ denotes the mean of $u$, not a Laplacian-bound constant.
In particular, the bottom row gives
$\Delta(-b)\ge-(c+C_{\mathrm{size}})$.
Since its Laplacian has integral zero, the mean-adjusted Green
kernel gives both assertions in \eqref{lower-green-b}, not just
the pointwise lower bound.

\begin{lemma}[Positive-area transfer of the lower mean]
\label{lower-mean-transfer}
Under \eqref{lower-green-u}--\eqref{lower-radial-upper}, all six
real field values are bounded uniformly on the fixed positive
parameter slab if and only if $L_0$ is bounded above there.
More precisely, there is a fixed $C_-$ such that
\[
 a,b\ge-L_0-C_-,\qquad M\ge-(L_0+C_-)I.
\]
If $L_0$ tends to $+\infty$ along a sequence, then
\[
 \|a\|_\infty+\|b\|_\infty+\|M\|_{\mathrm{op},\infty}
       \le C(1+L_0)
\]
along its tail.  Constants may depend on the fixed chain and
reference; no uniformity across varying chains is asserted.
\end{lemma}

\begin{proof}
Let $m=\lambda_{\min}(M)$.  Since $u\ge-m$,
\eqref{lower-green-u} gives $m\ge-L_0-C_u$.
Functional calculus and the positive trace pairing give
\[
 p_1\le p_*e^{a-m}.
\]
Using $p_1>\alpha$ we obtain
\[
 a\ge-L_0-C_u-\log(p_*/\alpha).
\]

The reference lower arrow is not identically zero.  Consequently
\[
 r_\Omega=\frac12\max_Xr_*>0,\qquad
 \Omega=\{r_*>r_\Omega\},\qquad
 m_\Omega=\int_\Omega\omega_X>0.
\]
This set is fixed independently of the solution and parameter.
On $\Omega$,
\[
 r=e^{-b}\operatorname{tr}(P_*e^M)
       \ge r_\Omega e^{m-b}.
\]
The inequality $r<d$ therefore implies
\[
 b\ge-L_0-C_\Omega\quad\hbox{on }\Omega,
 \qquad C_\Omega=C_u+\log(d/r_\Omega).
\]
Integrating over $\Omega$ and using the fluctuation estimate gives
\begin{align*}
 m_\Omega\overline b
 &=\int_\Omega b\,\omega_X
       -\int_\Omega(b-\overline b)\,\omega_X\\
 &\ge m_\Omega(-L_0-C_\Omega)-K_b.
\end{align*}
It follows from \eqref{lower-green-b} that
\[
 b\ge-L_0-C_\Omega-K_b/m_\Omega-C_b
       \quad\hbox{everywhere on }X.
\]
Enlarge a nonnegative $C_-$ to cover all three lower bounds.
No estimate has divided by the lower arrow at $p$.

Suppose $L_0\le B_1$ and put $D=\max\{0,B_1+C_-\}$.
The shifted scalars and matrix are nonnegative.  The reference
weights satisfy
\[
 q_*,\beta_*\ge\delta,\quad V_*\ge\delta I,
 \quad \delta=\frac4{21},\qquad
 q_*+\operatorname{tr}V_*+\beta_*=W=\frac{10}{3}.
\]
Thus \eqref{lower-radial-upper} implies
\[
 \delta\bigl((a+D)+\operatorname{tr}(M+DI)+(b+D)\bigr)
       \le C_T+DW.
\]
Each nonnegative summand, including each shifted eigenvalue of
$M$, is bounded by this sum.  This proves the full value bound.
Conversely, $\|M\|_{\mathrm{op},\infty}\le B$ gives
$u\le B+\log2$ and hence $L_0\le B+\log2$.
Finally use $D=L_0+C_-$ on a tail where it is nonnegative in the
same calculation to obtain the asserted linear escape bound.
\end{proof}

\subsection{Budgets on original solutions and exhaustive escape extraction}
\label{lower-original-budgets}

Write $\mathcal V=\int_X\omega_X=2\pi$, $t=\tau\mu_s$, and
$A_0=37/36$.  All unmarked integrals below use $\omega_X$.
The actual positive rows are
\begin{equation}\label{lower-rows}
\begin{aligned}
 U_0&=H_0-\alpha+ta=(C_0^\tau+K_0)/q,\\
 D&=\widehat H_1-\alpha I+tM
       =\mathcal L_V^{-1}(C_1^\tau+K_1),\\
 U_2&=H_2-\alpha+tb=(C_2^\tau+K_2)/\beta,
\end{aligned}
\qquad C_j^\tau=A_0I_j+\tau F'_{*,j}.
\end{equation}
Here $X,Y$ are the actual covariant arrow derivatives in current
unitary frames, with the same middle Chern connection, and
\[
 K_0=XX^*/(2g),\qquad K_1=(X^*X+YY^*)/(2g),\qquad
 K_2=Y^*Y/(2g),\qquad \operatorname{tr}K_1=K_0+K_2.
\]
The reference derivatives belong to $X,Y$, including $D\psi(p)$.
Positive symbols imply the global, zero-safe inequalities
\begin{equation}\label{lower-symbol-upper}
 \alpha<p_1<\frac83,\quad 0\le r<d,\quad
 0<q<\frac52,\quad 0<\beta\le d,\quad
 0<V\le\frac83I,\qquad q+\operatorname{tr}V+\beta=W=\frac{10}{3}.
\end{equation}
In particular $\beta(p)=d$ is allowed.

Lemma~\ref{common-positive-size} gives
$t\max(|a|,|b|,\|M\|_{\rm op})\le\mathcal C$, with
$\mathcal C=C_{\rm size}$ for this boundary reference.
Integrating \eqref{lower-rows}, using the degree integrals
$\int H_0=3\mathcal V$, $\int\operatorname{tr}\widehat H_1=4\mathcal V$,
$\int H_2=\mathcal V$, yields
\begin{equation}\label{lower-source-budgets}
\begin{aligned}
 \int U_0&\le J_0=(17/6+\mathcal C)\mathcal V,\\
 \int U_2&\le J_2=(c+\mathcal C)\mathcal V,\\
 \int\operatorname{tr}D&\le J_1=(11/3+2\mathcal C)\mathcal V.
\end{aligned}
\end{equation}
The ordered inverse is
\[
 \mathcal L_V^{-1}N=2\int_0^\infty e^{-sV}Ne^{-sV}\,ds,
 \qquad \operatorname{tr}(\mathcal L_V^{-1}N)
            =\operatorname{tr}(V^{-1}N).
\]
Its positivity and the force bounds give
\begin{equation}\label{lower-reciprocal-budgets}
\begin{aligned}
 \int[(55/63)/q+K_0/q]&\le J_0,\\
 \int[A_0/\beta+K_2/\beta]&\le J_2,\\
 \int[(A_0/2)\operatorname{tr}V^{-1}
                  +\operatorname{tr}(V^{-1}K_1)]&\le J_1.
\end{aligned}
\end{equation}
Consequently $\int K_0\le(5/2)J_0$, $\int K_2\le dJ_2$.
All current curvatures are bounded below by $-C_HI$, where
$C_H=(\mathcal C-\alpha)_+$.  For a rank-$k$ Hermitian matrix with
this lower bound, $\operatorname{tr}|H|\le\operatorname{tr}H+2kC_H$.
Thus
\begin{equation}\label{lower-curvature-budgets}
 \int|H_0|\le(3+2C_H)\mathcal V,\quad
 \int\operatorname{tr}|\widehat H_1|\le(4+4C_H)\mathcal V,\quad
 \int|H_2|\le(1+2C_H)\mathcal V.
\end{equation}
These are total-mass bounds; no uniform integrability of moving
projection--curvature products is asserted.

The Green estimates of Lemmas~\ref{common-negative-propagation} and
\ref{common-scalar-fluctuations} give all finite $L^s$ fluctuations of
$a,b,u$.  The identity
$\Delta\operatorname{tr}M=4-\operatorname{tr}\widehat H_1$ and
\eqref{lower-curvature-budgets} give the same bound for
$\operatorname{tr}M$.  For the logarithmic eigenvalues $m\le n$ of $M$,
\[
 m=-u+\epsilon,\qquad
 0\le\epsilon=\log(1+e^{-(n-m)})\le\log2,\qquad
 n=\operatorname{tr}M-m.
\]
They therefore also have bounded finite $L^s$ fluctuations.

\begin{lemma}[The only possible normalized escape]\label{lower-residual-extraction}
If the original solution family has unbounded full values, a common
subsequence has $L_n=\overline u_n\to\infty$, $t_n\to0$,
$\tau_n\to0$, and $t_nL_n\to\rho\ge0$, with
\begin{equation}\label{lower-residual-profiles}
 a_n/L_n,\ b_n/L_n,\ m_n/L_n\to-1,\qquad
 n_n/L_n\to\vartheta>-1
\end{equation}
strongly in every finite $L^s$ and almost everywhere.  Moreover
$\|w_n\|_\infty\le C(1+L_n)$, and smooth nonincreasing low spectral
probes $E_n=\chi(M_n/L_n)$ converge strongly in $L^2$, weakly in $H^1$,
to the constant reference projection $\Pi_L$ onto $L$.
The escape parameter $\vartheta$ is distinct from the phase $\theta_4$.
\end{lemma}

\begin{proof}
Lemma~\ref{lower-mean-transfer} gives $L_n\to\infty$ and the full
linear value bound after selecting solutions with increasing values.
The size estimate gives $t_n(L_n-\log2)\le\mathcal C$; hence
$t_n\to0$ and $\tau_n=t_n/\mu_s\to0$ because the mass is fixed.
Extract $t_nL_n\to\rho$ and the finitely many normalized means.
The fluctuation bounds give
\[
 a_n/L_n\to A,\quad b_n/L_n\to B,\quad
 m_n/L_n\to-1,\quad n_n/L_n\to\vartheta\ge-1.
\]
All convergence is strong in finite $L^s$ and almost everywhere.
The lower bounds give $A,B\ge-1$.  Since
$p_1\ge p_*e^{a-n}$ and $p_1<8/3$,
$a\le n+\log(112/27)$, so $A\le\vartheta$.

The flat reference gives the ordered energy identity
\begin{equation}\label{lower-logarithmic-energy}
 \int g^{-1}\operatorname{tr}(M_{\bar z}f^{-1}f_z)
    =\int\operatorname{tr}[M(\widehat H_1-2I)]\ge0.
\end{equation}
In an $M$-eigenbasis its integrand has coefficients
$k_{\exp}(s)=(e^s-1)/s$, with value $1$ at zero and
$k_{\exp}(s)\ge(1+|s|)^{-1}$.  The curvature budget bounds the
energy by $C\|M\|_\infty$, whence
\[
 \int|DM|^2\le C\|M\|_\infty(1+2\|M\|_\infty).
\]
Thus $M_n/L_n$ is $H^1$ bounded.  Local mollification has error
$C\varepsilon\|D(M_n/L_n)\|_2$; fixed mollifications have a
finite-grid convergent subsequence.  A diagonal choice as
$\varepsilon\downarrow0$ gives strong $L^2$, weak $H^1$ compactness.

If $\vartheta=-1$, then $A=-1$.  The normalized radial entropy
\[
 (v-T)/L_n=L_n^{-1}\{
 \operatorname{tr}[R_*k(aI-M)]+
 \operatorname{tr}[P_*k(M-bI)]\},\quad
 k(y)=ye^y-e^y+1,
\]
tends to zero almost everywhere.  The upper normalized arguments
tend to zero, and the lower ones tend to $-1-B\le0$.
The bounded exponential sums $p_1<8/3$, $r<d$, together with the
normalized value bound, dominate these terms, also where $P_*=0$.
Since $\overline v=0$, the limiting trace equation yields
$0=\alpha-2c+dB$, or $B=9/11$.
The upper bound for $T$ and $a\le n+\log(112/27)$ give
$(c-\alpha)n+cm+db\le C$.  The global lower bounds for $m$ and
$b\ge\overline b-C_b$ imply $n\le-L_n+o(L_n)$ uniformly.
Therefore
\[
 r_n\le\max r_*\,e^{n_n-b_n}
       \le C e^{-(20/11+o(1))L_n}\to0
\]
uniformly.  The actual lower Bochner identity
\[
 2\int K_2=\int\operatorname{tr}(P\widehat H_1)-\int rH_2
\]
and the current-curvature budgets give $\int K_2\to0$.
Integrating the undivided bottom row gives
$dc\mathcal V+dt_n\int b_n=A_0\mathcal V+o(1)$, hence
$t_n\overline b_n\to-3/11$.  This contradicts $B=9/11$ and $t_n>0$.

Now $\vartheta>-1$.  The matrix limit is
$-\Pi+\vartheta(I-\Pi)$ for a rank-one Hermitian $H^1$ projection.
Choose a fixed smooth nonincreasing $\chi$ equal to $1$ near $-1$
and $0$ near $\vartheta$.  Smooth divided differences at repeated
spectra make $E_n=\chi(M_n/L_n)$ $H^1$ bounded and strongly $L^2$
convergent to $\Pi$.  The exact moving test is
\begin{equation}\label{lower-spectral-test}
 \int\operatorname{tr}(E_n\widehat H_{1,n})
 =2\int\operatorname{tr}E_n+
      \int g^{-1}\operatorname{tr}[(E_n)_{\bar z}f_n^{-1}(f_n)_z].
\end{equation}
Every coefficient in the last integrand is a nonpositive divided
difference of $\chi$, multiplied by a positive exponential divided
difference.  On the good set where the normalized eigenvalues lie
in their separated plateaus, retain its $21$ term:
\[
 \int g^{-1}\operatorname{tr}[(E_n)_{\bar z}f_n^{-1}(f_n)_z]
 \le-\int_{\rm good}g^{-1}(1-e^{-(n_n-m_n)})
                         |(E_n)_z{}_{21}|^2.
\]
The bad set has vanishing measure; its terms are discarded with
their nonpositive sign.  The $L^2$-bounded fields
$(I-E_n)(E_n)_zE_n$ converge weakly to $(I-\Pi)\Pi_z\Pi$.
Multiplication by the good-set indicator preserves this weak limit:
pair a fixed $L^2$ test with the bad set and use its vanishing
$L^2$ norm there.  Weak lower semicontinuity therefore gives
\begin{equation}\label{lower-spectral-degree}
 \limsup\overline{\operatorname{tr}(E_n\widehat H_{1,n})}
       \le2-\mathfrak E_\Pi/\mathcal V,
 \qquad \mathfrak E_\Pi=\int g^{-1}|(I-\Pi)\Pi_z\Pi|^2.
\end{equation}

Bounded arrow values on the good set give the implications
\[
 A>-1\Longrightarrow\phi\Pi=0,\qquad
 B<\vartheta\Longrightarrow(I-\Pi)\psi=0.
\]
For example a strict normalized exponential gap forces the corresponding
reference arrow coefficient to zero; strong $L^2$ probe convergence
then passes that bounded coefficient.  Both constraints cannot hold,
because $\phi\psi$ is nonzero.  If $A>-1$, then
$\Pi=\Pi_{\ker\phi}$ and $B\ge\vartheta>-1$.
The kernel is the smooth holomorphic $\mathcal O(1)$ line; its
reference Gauss--Codazzi degree identity gives $\mathfrak E_\Pi=\mathcal V$.
Thus \eqref{lower-spectral-degree} has upper bound $1$.
On the other hand its low-$P$ values tend to zero in $L^1$, and
\[
 V^{-1}\ge(cI+P)^{-1}=c^{-1}I-P/[c(c+r)].
\]
The positive bare force and reciprocal budget give the lower bound
$\alpha+\rho+A_0/c=7/5+\rho$, a contradiction.  Therefore $A=-1$.

If $B>-1$, bounded high-$R$ and low-$P$ values tend to zero in $L^1$.
Hence $\overline{q+\operatorname{tr}(E_nV)}\to c-\alpha=2/3$.
The spectral test and the top degree give
\[
 \limsup\overline{\operatorname{tr}(E_nD_n)}\le d-\rho,
 \qquad\overline U_{0,n}\to17/6-\rho.
\]
For any positive $E$, scalar--matrix Hilbert Cauchy gives
\begin{equation}\label{lower-hilbert-cauchy}
 (1+\overline{\operatorname{tr}E})^2
 \le\overline{q+\operatorname{tr}(EV)}\,
       \overline{q^{-1}+\operatorname{tr}(EV^{-1})}.
\end{equation}
Use the Hilbert vectors $E^{1/2}V^{1/2}$ and $E^{1/2}V^{-1/2}$
and cyclic trace, together with $q^{1/2},q^{-1/2}$ in the scalar
slot.  No $EV=VE$ is required.  Since $C_j^\tau\to A_0I_j$,
the reciprocal budgets absorb the force errors.  The lower
top-plus-middle source bound is $4A_0/(2/3)=37/6$, whereas its
upper bound is $17/6+d-2\rho\le28/6$.  This contradiction excludes
every $B>-1$.  We have $B=-1<\vartheta$, so the lower-arrow
constraint identifies $\Pi=\Pi_L$ off $p$, hence almost everywhere
with the constant reference projection onto the saturated $\mathcal O(2)$
line.  This proves the stated exhaustive residual.
\end{proof}

\subsection{Physical line potentials and mandatory divisor mass}
\label{lower-physical-potentials}

Assume henceforth the residual of Lemma~\ref{lower-residual-extraction}.
Choose a constant reference-unitary basis with the fixed lower coefficient
vector $v=e_1$, absorbing its fixed length in $\sigma$.  This basis
choice imposes no restriction on the full metric.  Let $\Pi$ denote
the actual metric-orthogonal projection onto $L$, in current unitary
coordinates, and put $Q_L=I-\Pi$, $J=R/p_1$.
The projection $J$ is the upper quotient projection; none of these
physical projections is identified with the reference probe $E_n$.
Set
\[
 \ell_L=\log f_{11},\qquad \ell_Q=\operatorname{tr}M-\ell_L,
 \qquad H_L=2-\Delta\ell_L,\qquad H_Q=2-\Delta\ell_Q.
\]
Both physical line bundles have degree two.  With the current
second-fundamental energy $\kappa_L\ge0$, Gauss--Codazzi gives
\begin{equation}\label{lower-gauss-codazzi}
 \operatorname{tr}(\Pi\widehat H_1)=H_L+\kappa_L,
 \quad \operatorname{tr}(Q_L\widehat H_1)=H_Q-\kappa_L,
 \quad 2K_{2,\perp}=r\kappa_L.
\end{equation}
The last identity is valid at $p$ without division by $r$.
The lower curvature bound and $\int H_Q=2\mathcal V$ imply
$\int\kappa_L\le(d+\mathcal C)\mathcal V$.
For the actual holomorphic upper kernel,
$2K_{0,\perp}=p_1\kappa_K$; $p_1\ge\alpha$ bounds $\int\kappa_K$.

The fixed rank-one inverse potential retains the full connection.
Let $h=f^{-1}$ and $s=e_2^*he_2=f_{11}/\det f$, so
$\log s=-\ell_Q$ and $\ell_L=\operatorname{tr}M+\log s$.
The exact inverse row is
\[
 \Delta h=g^{-1}hf_zhf_{\bar z}h-dh-tMh+S^{-1}DS^{-1}.
\]
For $k=he_2$, $s_z=-k^*f_zk$, and ordered Cauchy gives
\[
 |k^*f_zk|^2\le(k^*h^{-1}k)(k^*f_zhf_{\bar z}k)
                =s\,e_2^*hf_zhf_{\bar z}he_2.
\]
Thus $\Delta\log s\ge-(d+\mathcal C)$.  The same mean-adjusted
Green kernel supplies $\ell_Q\ge\overline{\ell_Q}-C$ and all its
finite $L^s$ fluctuations.  It also supplies fluctuations of $\ell_L$.
The exact lower norm relation
\begin{equation}\label{lower-physical-lower-norm}
 r=r_*e^{\ell_L-b}
\end{equation}
and $\ell_L\ge m$ show
$\ell_{L,n}/L_n\to-1$, $\ell_{Q,n}/L_n\to\vartheta$ strongly
in finite $L^s$ and almost everywhere; $r_*>0$ is used only off $p$.

Take a common subsequence of the bounded scalar symbol means and write
\[
 s_\ell=\lim\overline{\operatorname{tr}(E_nP_n)},\qquad
 s_h=\lim\overline{\operatorname{tr}((I-E_n)P_n)}.
\]
The high-$R$ values tend to zero in $L^1$, and the original degree
and spectral identities yield
\begin{equation}\label{lower-residual-source-totals}
 \overline U_{0,n}\to17/6-\rho,\qquad
 \overline U_{2,n}\to c-\rho,\qquad
 \limsup\overline{\operatorname{tr}(E_nD_n)}\le d-\rho.
\end{equation}
Equation~\eqref{lower-hilbert-cauchy} therefore gives
\begin{equation}\label{lower-positive-low-mass}
 s_\ell\ge\frac{4A_0}{14/3-2\rho}-\frac23\ge\frac3{14}.
\end{equation}
The denominator is positive by the finite reciprocal-source budget;
the bottom row independently gives $0\le\rho\le3/11$.

These positive low values prevent a relative-mean escape.  For the actual
low spectral projection $J_{\rm low}$,
\[
 P_{\rm low}
  =r\frac{e^m|J_{\rm low}e_1|^2}{e_1^*e^Me_1}
       \le d e^{m-\ell_L}.
\]
If $\overline{\ell_L-m}\to\infty$, its bounded $L^1$ fluctuation
would make $P_{\rm low}\to0$ in measure, then in $L^1$ by boundedness,
contrary to \eqref{lower-positive-low-mass}.  Thus
$0\le\overline{\ell_L-m}\le C$.
The integrable inequality $\ell_L-b\le\log(d/r_*)$ off $p$ bounds
its mean above.  A mean tending to $-\infty$ would make $r\to0$
in measure and in $L^1$, the same contradiction; hence
$|\overline{\ell_L-b}|\le C$.
Finally $p_1\le p_*e^{a-m}$ bounds $a-m$ below.
On a fixed positive-area region with $\phi|_L\ne0$, low-projection
convergence supplies a subset of uniformly positive area on which
$a-m\le C$; its bounded $L^1$ fluctuation bounds its mean above.
Consequently
\begin{equation}\label{lower-relative-means}
 |\overline{a-m}|+|\overline{b-m}|
             +|\overline{\ell_L-m}|\le C,
 \qquad \overline{n-\ell_Q}=\overline{\ell_L-m}.
\end{equation}

Let $u_*=|\phi|_L|_*^2$.  Its sole simple-section zero is $q_0$.
The actual restriction norm is $U=e^{a-\ell_L}u_*\le p_1<8/3$.
The trace quantity obeys
\[
 \Delta T=(1-\tau)f_*+tv,\qquad |tv|\le W\mathcal C,
\]
so its Green representation gives $\operatorname{osc}T\le C$.
The norm terms in $T$ are bounded; hence
\[
 | -\alpha(a-\bar a)+c(\ell_L-\bar\ell_L)
       +c(\ell_Q-\bar\ell_Q)+d(b-\bar b)|\le C.
\]
The centered lower bounds for $a,b,m,\ell_Q$ and
\eqref{lower-relative-means} also bound $\ell_L$ below.
Insert $a-\bar a\le\ell_L-\bar\ell_L+C-\log u_*$ in this
bounded linear combination.  Since $c-\alpha=2/3$, the resulting
upper bounds are
\begin{equation}\label{lower-logarithmic-values}
\begin{aligned}
 a-\bar a&\le C-\tfrac54\log u_*,&
 \ell_L-\bar\ell_L&\le C-\tfrac14\log u_*,\\
 \ell_Q-\bar\ell_Q&\le C-\tfrac15\log u_*,&
 b-\bar b&\le C-\tfrac1{11}\log u_*.
\end{aligned}
\end{equation}
They give full relative value bounds on each compact subset of
$X^\circ=X\setminus\{q_0\}$.  If $p\ne q_0$, this domain includes
the lower zero with its actual derivative.

In the next lemma and its proof, a subscript infinity on a logarithmic
metric potential denotes its centered subsequential limit: each
potential's own mean is subtracted before taking the subsequence.
This convention includes the minimum eigenvalue of the middle logarithm.
The bounded differences of means are extracted separately.

The original divisor masses use the curve form of Poincar\'e--Lelong
\cite[Ch.~V, \S13, equation (13.2)]{Demailly2012}: with our Laplacian
normalization, a simple zero contributes mass \(2\pi\).
The subsequent coupled atomic comparisons are proved here, not supplied
by that identity.

\begin{lemma}[Mandatory original divisor atoms]\label{lower-divisor-atoms}
On a common measure subsequence, let
\[
 \mathfrak a_0,\quad\mathfrak a_L,\quad
 \mathfrak a_Q,\quad\mathfrak a_2
\]
be the scalar curvature atoms at $q_0$, and let $\mathfrak a_m$ be the
atom of the distribution $2-\Delta m_\infty$.  Then
\begin{equation}\label{lower-atom-laws}
\begin{gathered}
 \mathfrak a_m=\mathfrak a_0-\mathcal V\ge0,\qquad
 \mathfrak a_L\ge\mathfrak a_m,\qquad
 \mathfrak a_L,\mathfrak a_Q,\mathfrak a_2\ge0,\\
 \alpha\mathfrak a_0=c(\mathfrak a_L+\mathfrak a_Q)
                                      +d\mathfrak a_2,\qquad
 \mathcal V\le\mathfrak a_0\le\tfrac54\mathcal V.
\end{gathered}
\end{equation}
For $p\ne q_0$, also $\mathfrak a_2\ge\mathfrak a_L$ and
$\mathfrak a_0\le(16/15)\mathcal V$.
For $p=q_0$, the valid comparison is
$\mathfrak a_L\le\mathfrak a_2+\mathcal V$.
\end{lemma}

\begin{proof}
The uniform scalar Laplacian total variations give strong compactness of
centered potentials in every finite $L^s$.  To see this for the actual
Green kernel, truncate the chordal logarithm to a continuous bounded
kernel.  Its omitted $L^s$ tail tends to zero uniformly in the source
point, because
$\int|G_0|^k\omega_{\rm FS}=(2\pi)^{1-k}k!$.
Minkowski controls this omitted potential for bounded-variation measures;
weak measure convergence gives uniform convergence for the truncated
kernel.  Subtracting the actual means preserves the argument.
For $m=-u+\epsilon$, the normalized spectral separation makes
$\epsilon_n\to0$ in finite $L^s$; thus the same compactness applies
to $m$.  Extract the bounded relative-mean differences as well.

The bounded high-$R$ values disappear.  The limiting upper norm is
therefore
\[
 p_{1,\infty}=e^{a_\infty-m_\infty+\delta_a}u_*,\qquad
 \alpha\le p_{1,\infty}\le8/3,
\]
where $\delta_a=\lim(\bar a-\bar m)$.
Its exact logarithmic Bochner identity is
\begin{equation}\label{lower-upper-log-bochner}
 \Delta\log p_1=\operatorname{tr}(J\widehat H_1)-H_0+\kappa_K.
\end{equation}
The curvature and kinetic budgets bound its Laplacian variation.
Its bounded limiting logarithm has no Laplacian atom.  Indeed a
finite-measure Laplacian atom of mass $\lambda$ produces circular
average $(\lambda/2\pi)\log r^2+o(|\log r|)$ by the logarithmic
kernel; dominated convergence handles the non-atomic remainder.
A bounded potential therefore has no atom, and a potential bounded
below has nonpositive Laplacian atoms, or nonnegative curvature atoms.

The identity
$m_\infty=a_\infty+\delta_a+\log u_*-\log p_{1,\infty}$
has the original simple divisor $\mathcal V\delta_{q_0}$ and gives
$\mathfrak a_m=\mathfrak a_0-\mathcal V$.
The centered lower bounds make all displayed curvature atoms
nonnegative.  Since $\ell_L\ge m$ with bounded relative means,
the same circular-average comparison gives
$\mathfrak a_L\ge\mathfrak a_m$.
The bounded trace combination has no atom, giving the linear atom law.
It implies $(c-\alpha)\mathfrak a_0\le c\mathcal V$, hence the
upper bound $5\mathcal V/4$.

If $p\ne q_0$, $r_*>0$ near $q_0$, and
$b-\ell_L\ge\log(r_*/d)$ there.  Thus
$\mathfrak a_2\ge\mathfrak a_L\ge\mathfrak a_0-\mathcal V$;
the linear law gives
$(c+d-\alpha)\mathfrak a_0\le(c+d)\mathcal V$,
or $\mathfrak a_0\le16\mathcal V/15$.
If $p=q_0$, the simple lower divisor instead permits
$\mathfrak a_L\le\mathfrak a_2+\mathcal V$.
The estimates \eqref{lower-logarithmic-values} give bounded centered
potentials off $q_0$, so the positive scalar-source atoms are localized
there.  This does not yet exclude non-atomic concentration.

For later comparisons retain the actual product atom
$\mathfrak b_R$ of $\operatorname{tr}(J\widehat H_1)\omega_X$.
Positivity of $D$ and bounded damping imply
$0\le\mathfrak b_R\le\mathfrak a_L+\mathfrak a_Q$.
Taking atoms in \eqref{lower-upper-log-bochner} gives
\begin{equation}\label{lower-kernel-atom}
 \mathfrak k_K=\mathfrak a_0-\mathfrak b_R
 \ge\tfrac45\mathfrak a_0+\tfrac{11}{5}\mathfrak a_2
 \ge\tfrac45\mathcal V.
\end{equation}
This is a moving-product measure identity, not a product of limits.
\end{proof}

\subsection{Two-scale matrix normalization and the transported coefficients}
\label{lower-two-scale}

The two eigenvalue scales need separate constant rescalings. These
rescalings give bounded metric values on compact subsets of the
punctured curve while retaining the full middle matrix. The transformed
equations must retain the damping as well: bounded transformed values do
not imply that the original mass terms tend to zero.

Put $\mathcal B_n=\overline{\ell_{Q,n}}$ and define the constant gauges
\[
 G_n=\operatorname{diag}(e^{L_n/2},e^{-\mathcal B_n/2}),\quad
 g_n=G_n^{-1},\quad F_n=G_nf_nG_n,\quad
 A_n=a_n+L_n,\quad B_n=b_n+L_n.
\]
Since $\bar m_n+L_n=\bar\epsilon_n$, the relative-mean bounds give
$|\bar a_n+L_n|+|\bar b_n+L_n|+|\bar\ell_{L,n}+L_n|\le C$.
On every fixed compact buffer $K\Subset X^\circ$, the relative
value bounds imply
\[
 |A_n|+|B_n|+|\ell_{L,n}+L_n|
                    +|\ell_{Q,n}-\mathcal B_n|\le C_K,
 \quad 0\le\ell_{L,n}-m_n\le C_K.
\]
The two diagonal entries of $F_n$ are therefore bounded above:
$(F_n)_{11}=e^{L_n+\ell_{L,n}}\le C_K$ and
$(F_n)_{22}\le e^{n_n-\mathcal B_n}\le C_K$.
Moreover
\[
 \det F_n=e^{L_n+\ell_{L,n}+\ell_{Q,n}-\mathcal B_n}\ge c_K>0.
\]
For a positive $2\times2$ matrix, $\lambda_{\max}\le\operatorname{tr}F$
and $\lambda_{\min}=\det F/\lambda_{\max}$.  Hence
\begin{equation}\label{lower-punctured-values}
 c_KI\le F_n\le C_KI,\qquad |A_n|+|B_n|\le C_K.
\end{equation}
All off-diagonal entries remain.  The global bounds also retain
\begin{equation}\label{lower-punctured-growth}
 (F_n)_{11}\le Cu_*^{-1/4},\quad
 (F_n)_{22}\le Cu_*^{-9/20},\quad\det F_n\ge c,
 \quad\|F_n\|+\|F_n^{-1}\|\le Cu_*^{-9/20}.
\end{equation}
Indeed $\ell_L-m\le C-\tfrac14\log u_*$ and
$n=\ell_Q+\ell_L-m$ give the exponent $1/5+1/4=9/20$.
The estimates explicitly degenerate at $q_0$.

With endpoint gauges $g_{0,n}=g_{2,n}=e^{-L_n/2}$, the transformed
arrows are exactly
\begin{equation}\label{lower-transformed-arrows}
 \phi_n=(\phi_L,\varepsilon_n\phi_Q),\qquad
 \psi_n=(\sigma,0)^T,\qquad
 \varepsilon_n=e^{-(L_n+\mathcal B_n)/2}\to0.
\end{equation}
Every finite transformed chain is isomorphic over $\operatorname{id}_X$
to the original chain.  Its limiting rank drop is a device for excluding
escape, not a replacement theorem.  The transformed damping is
$N_n=g_nM_ng_n^{-1}$, which is $F_n$-selfadjoint and is not $\log F_n$.

We give its coefficient formula before applying spatial derivatives.
For $F=\left(\begin{smallmatrix}a_F&b_F\\\bar b_F&d_F\end{smallmatrix}\right)$
in a compact positive-metric value set, let
$A_x=\operatorname{diag}(x,1)F$, $x=\varepsilon_n^2$.
Its positive eigenvalues satisfy
\[
 \lambda_+=\frac{d_F+xa_F+
       \sqrt{(d_F+xa_F)^2-4x\det F}}2,\qquad
 \lambda_-=x\det F/\lambda_+.
\]
The discriminant and gap stay positive on a fixed neighborhood including
$x=0$.  Thus the low idempotent
\[
 J_x(F)=\frac{\lambda_+I-A_x}{\lambda_+-\lambda_-}
\]
extends smoothly to
\begin{equation}\label{lower-upper-projection-coefficient}
 J(F)=J_0(F)=
 \begin{pmatrix}1&0\\-\bar b_F/d_F&0\end{pmatrix}.
\end{equation}
It is $F$-orthogonal, not generally standard-orthogonal, and is the
upper quotient projection rather than the lower-line projection.
Similarity gives $g_ne^{M_n}g_n^{-1}=e^{\mathcal B_n}A_x$.
The exact logarithm is consequently
\begin{equation}\label{lower-damping-coefficient}
\begin{aligned}
 \Psi_n(F):=t_nN_n={}&-t_nL_nJ_x+t_n\mathcal B_n(I-J_x)\\
 &+t_n\{\log(\det F/\lambda_+)J_x
                         +\log\lambda_+(I-J_x)\}.
\end{aligned}
\end{equation}
Integrate the smooth $x$ derivative, also after any two $F$ derivatives.
This gives
\[
 \|J_x-J\|_{C^2}+\|\log\lambda_+-\log d_F\|_{C^2}\le C_Kx,
\]
and
\begin{equation}\label{lower-damping-limit}
 \|\Psi_n-[-\rho J+\rho\vartheta(I-J)]\|_{C^2}
 \le C_K\{|t_nL_n-\rho|+|t_n\mathcal B_n-\rho\vartheta|
       +t_n+t_n(L_n+|\mathcal B_n|)\varepsilon_n^2\}\to0.
\end{equation}
These are derivatives in metric-value variables, not spatial derivatives
of an as-yet unbounded field.  The exact actual upper quotient is
\[
 \bar J_n(F,z)=F^{-1}\phi_n^*\phi_n/(\phi_nF^{-1}\phi_n^*).
\]
On a fixed buffer divide its direction by the nonzero $\phi_L$ and use
$(1,\varepsilon_n\phi_Q/\phi_L)$.  Its denominator is uniformly
positive and its joint coordinate/value $C^2$ difference from $J(F)$
is $O_K(\varepsilon_n)$.  The scalar dampings
$t_n(A-L_n),t_n(B-L_n)$ tend as coefficient functions to $-\rho$.

The transported low probe also converges uniformly on the fixed buffer
to $J(F_n)$, by its exact spectral formula and the separated plateaus
$m_n=-L_n+O_K(1)$, $n_n=\mathcal B_n+O_K(1)$.
Thus
\begin{equation}\label{lower-buffered-projection-comparison}
 \sup_K\|\bar E_n-\bar J_n\|_{F_n}\to0,
 \quad
 \int\zeta\operatorname{tr}[(\bar E_n-\bar J_n)\bar D_n]\to0
\end{equation}
for every fixed bounded cutoff $\zeta$ supported in $K$; the second
claim follows by multiplying the uniform coefficient error by
$\int\operatorname{tr}\bar D_n\le J_1$.
The same argument uses the current-curvature $L^1$ budget for curvature
pairings. Thus, on fixed punctured buffers, the transported spectral
probe and the upper quotient approach the same coefficient projection
$J(F_n)$. Neither is identified with the physical lower-line projection
$\Pi$. The comparison uses no source uniform integrability and does
not extend to $q_0$.

The original force also has to be transported.  With $S_F=F^{1/2}$,
the polar unitary relating the old and new current-unitary frames is
\begin{equation}\label{lower-polar-force}
 U_\varepsilon(F)=
 \frac{S_F\operatorname{diag}(\varepsilon,1)
       +\sqrt{\det F}\,S_F^{-1}\operatorname{diag}(1,\varepsilon)}
 {\sqrt{\varepsilon^2F_{11}+F_{22}+2\varepsilon\sqrt{\det F}}},
 \qquad
 \widetilde C_{1,n}=U_{\varepsilon_n}(F)
       [A_0I+\tau_nF'_{*,1}]U_{\varepsilon_n}(F)^*.
\end{equation}
The product order in the numerator matters.  The formula extends smoothly
to $\varepsilon=0$ on a compact positive value set; the denominator
is bounded below.  It proves uniform joint coefficient derivatives and
$\widetilde C_{1,n}\ge A_0I/2$.
The rotated residual is the original residual, rather than a reference
residual reconstructed from the limiting arrows.

\subsection{Subtracting divisor mass from the source budgets}
\label{lower-source-allocation}

The divisor mass is part of each relevant total source budget, not an
additional contribution available for the punctured estimate. We first
derive a preliminary bound from the uncut budgets, then use cutoffs to
subtract the atom before applying Cauchy inequalities to the remaining
sources. Applying the two reciprocal inequalities separately gives the
sharper bound needed in the final contradiction.

Normalize the original curvature atoms by
$A_{\rm at}=\mathfrak a_0/\mathcal V$,
$A_L=\mathfrak a_L/\mathcal V$, $A_Q=\mathfrak a_Q/\mathcal V$,
$A_2=\mathfrak a_2/\mathcal V$.
Dividing the single normalization by $L_n$ passes only bounded symbols
against strongly convergent fields.  Since high-$R$ values vanish,
\begin{equation}\label{lower-high-normalization}
 h:=c+s_h=\frac W{\vartheta+1}.
\end{equation}
Indeed the high-$V$ mean is $c+s_h=h$ and the other three slots have
mean $W-h$, so $0=-(W-h)+\vartheta h$.
Scalar bottom Cauchy and the actual source total give
\begin{equation}\label{lower-bottom-allocation}
 c-\rho\ge\frac{A_0}{d-s_\ell-s_h}>0,
 \qquad s_\ell+s_h\le d-A_0/c=\frac35.
\end{equation}
Before subtracting the atom, the uncut two-slot Cauchy inequality and
the full top/low-middle totals give the weaker, useful check
\begin{equation}\label{lower-uncut-allocation}
 s_\ell\ge\frac{4A_0}{14/3-2\rho}-\frac23\ge\frac3{14},
 \qquad
 \frac W{\vartheta+1}+\frac{A_0}{c-\rho}
                      +\frac{4A_0}{14/3-2\rho}\le W.
\end{equation}
The first denominator is positive by the actual reciprocal budget;
independently the bottom row gives $\rho\le3/11$.
Thus $s_h\le27/70$, $h\le128/105$ and
$d-s_\ell\le34/21$.  Equations
\eqref{lower-high-normalization}--\eqref{lower-bottom-allocation}
already imply the preliminary necessary region
$111/64\le\vartheta\le3$, $0\le\rho\le27/136$.
These are necessary inequalities for an extracted escape sequence;
satisfying them alone does not construct such a sequence.

Take smooth cutoffs $0\le\chi_\delta\le1$ vanishing near $q_0$
and tending to $1$ off $q_0$ as $\delta\downarrow0$.
Apply \eqref{lower-hilbert-cauchy} with this fixed cutoff, including
it in both Hilbert slots.  The top source has mass outside the atom
at most $(17/6-\rho-A_{\rm at})\mathcal V$, while the tested middle
source has total mass at most $(d-\rho)\mathcal V$.
Take $n\to\infty$ first, then $\delta\downarrow0$.
The omitted bounded-symbol integrals are at most a fixed constant
times the shrinking disk volume.  Therefore
\[
 4A_0\le(2/3+s_\ell)(14/3-2\rho-A_{\rm at})
           \le(2/3+s_\ell)(11/3-2\rho).
\]
Its denominators are positive as consequences of this inequality.
In conjunction with \eqref{lower-bottom-allocation}, it gives
\begin{equation}\label{lower-escape-parameter-region}
 s_\ell\ge\frac5{11},\quad s_h\le\frac8{55},\quad
 \frac{777}{323}\le\vartheta\le3,\qquad
 0\le\rho\le\frac8{91}.
\end{equation}
Here $h\le c+8/55=323/330$ gives the lower $\vartheta$ bound;
$d-s_\ell\le91/66$ gives $c-\rho\ge407/546$.
The coupled inequality is also retained:
\[
 \frac W{\vartheta+1}+\frac{A_0}{c-\rho}
                    +\frac{4A_0}{11/3-2\rho}\le W.
\]

For the sharper separate allocation, substitute the actual rows in
\eqref{lower-upper-log-bochner}:
\begin{equation}\label{lower-upper-source-accounting}
 \Delta\log p_1=\operatorname{tr}(JD)-U_0
                 +t(a-\operatorname{tr}(JM))+\kappa_K.
\end{equation}
The bounded high-$R$ weight and $p_1\ge\alpha$ give
$\operatorname{tr}(J_nM_n)/L_n\to-1$ in $L^1$.
The damping difference in \eqref{lower-upper-source-accounting}
therefore tends to zero in $L^1$.
Let $\mu_0,\mu_J,\mu_E,\mu_K$ be the common weak limits of
$U_{0,n}\omega_X$, $\operatorname{tr}(J_nD_n)\omega_X$,
$\operatorname{tr}(E_nD_n)\omega_X$, $\kappa_{K,n}\omega_X$.
The bounded upper logarithm has no Laplacian atom.  The total and atom
identities are
\[
 \mu_J(X)+\mu_K(X)=(17/6-\rho)\mathcal V,
 \qquad \mu_J(\{q_0\})+\mu_K(\{q_0\})=\mathfrak a_0.
\]
The buffered comparison \eqref{lower-buffered-projection-comparison}
gives $\mu_E|_{X^\circ}=\mu_J|_{X^\circ}$.
Writing
\[
 k^\circ=\mu_K(X^\circ)/\mathcal V,
 \qquad T_0=17/6-\rho-A_{\rm at},
\]
we obtain the exact remaining masses
\begin{equation}\label{lower-punctured-masses}
 \mu_0(X^\circ)/\mathcal V=T_0,\qquad
 \mu_E(X^\circ)/\mathcal V=T_0-k^\circ.
\end{equation}
Applying the combined two-slot inequality on these same punctured
tests now subtracts the top atom from both source slots:
\begin{equation}\label{lower-double-precharge}
\begin{gathered}
 4A_0\le(2/3+s_\ell)
             (17/3-2\rho-2A_{\rm at}-k^\circ),\\
 \frac W{\vartheta+1}+\frac{A_0}{c-\rho}
       +\frac{4A_0}{17/3-2\rho-2A_{\rm at}-k^\circ}\le W.
\end{gathered}
\end{equation}
The denominator is positive as a conclusion.  Since
$2/3+s_\ell\le19/15$, $A_{\rm at}\ge1$ and $\rho\ge0$,
$4A_0\le(19/15)(11/3-k^\circ)$, giving $k^\circ\le8/19$.
We retain this combined bound, then improve it without combining
the two reciprocal slots.
Separate scalar and matrix Cauchy inequalities with the same fixed
cutoffs now give, with $x=\lim\bar q_n$ and
$y=\lim\overline{\operatorname{tr}(E_nV_n)}$,
\begin{equation}\label{lower-separate-allocation}
 x\ge A_0/T_0,\qquad y\ge A_0/(T_0-k^\circ),\qquad
 h+\frac{A_0}{c-\rho}+\frac{A_0}{T_0}
                      +\frac{A_0}{T_0-k^\circ}\le W.
\end{equation}
Both denominators are proved positive by the finite source budgets.
Since $x+y=2/3+s_\ell\le19/15$,
$x\ge A_0/d=37/66$, and $T_0+\rho\le d$, we have
$y\le233/330$, $T_0-k^\circ\ge2035/1398$.  Hence
\begin{equation}\label{lower-kernel-allocation-bound}
 \boxed{k^\circ+\rho\le\frac{88}{233}.}
\end{equation}
The atom was subtracted before either source allocation.  No reciprocal
force mass is counted again at that atom.

\subsection{Further identities for line, quotient and concentration measures}
\label{lower-retained-correlations}

The preceding source allocation concerns one common escaping sequence.
We now record further identities for its physical line and quotient,
bounded tests at the divisor, and covariant mixed terms. The diagonal
middle rows contain a common off-diagonal term with no prescribed sign,
so they give no direct sign contradiction. The identities describe
additional constraints on concentration; the final contradiction uses
the full lower-arrow derivative estimate in
Section~\ref{lower-covariant-payment}.

Here $\Pi$ projects onto the physical lower image line $L$,
$Q_L=I-\Pi$ onto its physical quotient, and $J=R/p_1$ onto the upper
quotient. Their overlap is $\delta=\operatorname{tr}(\Pi J)$.
The spectral tests $E_n=\chi(M_n/L_n)$ are the probes used in the escape
extraction; they are not the physical lower-line projections.

\paragraph{Line and quotient curvature.}

First, the saturated lower section obeys the smooth Chern--Bochner
identity and divisor formulas
\[
 \int r(H_L-H_2)=\int|\nabla\sigma|^2,
 \quad i\partial\bar\partial\log r+iF_{L\otimes E_2^{-1}}
             =2\pi[p],
 \quad i\partial\bar\partial\log U+iF_{E_0\otimes L^{-1}}
             =2\pi[q_0].
\]
No divisor is inserted into the smooth unlogged identity, and the
logarithmic identities are distributional.  In a $\Pi$-adapted frame
write $R=\left(\begin{smallmatrix}u_R&z_R\\\bar z_R&w_R\end{smallmatrix}\right)$,
$P=\operatorname{diag}(r,0)$.  The exact diagonal middle rows are
\[
 (c-u_R+r)D_{LL}-\operatorname{Re}(z_RD_{QL})=(C_1)_{LL}+(K_1)_{LL},
\]
\[
 (c-w_R)D_{QQ}-\operatorname{Re}(z_RD_{QL})=(C_1)_{QQ}+(K_1)_{QQ}.
\]
The common mixed term has no sign.  The coupled divisor row can instead
be tested with $\mathcal L_V^{-1}(\Pi-J)$; its kinetic contribution is
\[
 \operatorname{tr}(\mathcal L_V^{-1}(\Pi-J)K_1)
   =[X\mathcal L_V^{-1}(\Pi-J)X^*
                    +Y^*\mathcal L_V^{-1}(\Pi-J)Y]/(2g).
\]
It is bounded in total variation by
$\operatorname{tr}(V^{-1}K_1)$, because
$-V^{-1}\le\mathcal L_V^{-1}(\Pi-J)\le V^{-1}$.
Bounded measure size does not make this term vanish.
For an undamped calibrated solution, positivity in this same splitting
gives $w_R<5/6$ and
$25-30p_1+30r-36rw_R>0$.  At $p$ or $q_0$ the nonzero upper
component can still lie in the compatible interval $(1/6,5/6)$.
The bottom row at the lower zero is exactly
$H_2(p)=(8+6K_2(p))/11$; it is not a pointwise contradiction.
The endpoint Chern--Bochner identities
\eqref{geo-endpoint-integrals} retain respectively the factors
$-18\pi/\sqrt{37}$ and $6\pi/\sqrt{37}$ and their
$3/\sqrt{37}$ kinetic terms.  Their signs do not remove the displayed
off-diagonal term or contradict these saturated-line divisor identities.

The physical quotient retains its complete nonnegative excess
\begin{equation}\label{lower-quotient-excess}
\begin{aligned}
 \mathfrak Q={}&\kappa_L
  +\operatorname{tr}(Q_L\mathcal L_V^{-1}K_1)
  +A_0[\operatorname{tr}(Q_LV^{-1})-1/c]\\
 ={}&H_Q-7/5+t\operatorname{tr}(Q_LM)
               -\tau\operatorname{tr}(Q_L\mathcal L_V^{-1}F'_{*,1}).
\end{aligned}
\end{equation}
Indeed $\operatorname{tr}(Q_LV)=c-\operatorname{tr}(Q_LR)\le c$,
and vector Cauchy gives $\operatorname{tr}(Q_LV^{-1})\ge1/c$.
The force error tends to zero in variation by the reciprocal budget.
Its weak limit $\nu_Q$ has atom $\mathfrak a_Q$ at $q_0$, and
$\nu_Q(X)/\mathcal V=3/5+\sigma_Q$, where
$-\rho\le\sigma_Q\le\rho\vartheta$.
The atom laws give $A_Q\le1-(4/5)A_{\rm at}-(11/5)A_2\le1/5$.
Thus
\[
 \nu_Q(X^\circ)/\mathcal V\ge2/5-\rho\ge142/455,
 \qquad \nu_Q(X^\circ)/\mathcal V-k^\circ\ge26/1165.
\]
This cost includes the full parallel component of the lower-arrow derivative $Y$.

\paragraph{Moment, composition and restriction tests.}

The positive upper and signed lower moment identities are
\begin{equation}\label{lower-moment-identities}
\begin{aligned}
 K_0+\operatorname{tr}(RD)
  &=C_0^\tau+\alpha U_0
        +t[\operatorname{tr}(RM)-p_1a]+\Delta p_1,\\
 \operatorname{tr}(PD)-K_2
  &=dU_2-C_2^\tau
        +t[\operatorname{tr}(PM)-rb]-\Delta r.
\end{aligned}
\end{equation}
The first positive measure has total limiting mean $3/2-\rho/6$,
atom $\alpha\mathfrak a_0$, and punctured mean
$3/2-\rho/6-A_{\rm at}/6\le4/3-\rho/6$.
The second has total limiting mean
$1/2-d\rho+\rho(1+\vartheta)s_h$ and atom $d\mathfrak a_2$.
Its sign is not assumed favorable.

For the composition and restriction identities, write the upper arrow
as $u^*$ and the lower arrow as $v$ in current unitary coordinates.
Then $R=uu^*$, $P=vv^*$, and $X,Y$ are their full covariant derivatives
in the same middle Chern connection. The probes retain the cross energies.
For $C=|\phi\psi|^2=\operatorname{tr}(RP)$,
$\Xi=Xv+u^*Y$, $K_C=|\Xi|^2/(2g)$, one has
\[
 2K_C+CU_2=CU_0+tC(b-a)+\Delta C,
 \qquad K_C\le2(rK_0+p_1K_2).
\]
The bounded relative means and scalar Green compactness imply that
$a-b$ converges almost everywhere and strongly in every finite $L^p$.
The uniformly bounded $C_n$ therefore converge to $C_\infty$ in those
senses.
The damping $tC(b-a)$ tends to zero in $L^1$; $\Delta C_\infty$
is a finite signed measure and has no point atom.
Extract the actual positive moving-product measures
\[
\begin{gathered}
 \nu_{C0}=\lim C_nU_{0,n}\omega_X,\qquad
 \nu_{C2}=\lim C_nU_{2,n}\omega_X,\\
 \nu_{KC}=\lim K_{C,n}\omega_X.
\end{gathered}
\]
Then
\begin{equation}\label{lower-composition-measures}
 2\nu_{KC}+\nu_{C2}=\nu_{C0}+(\Delta C_\infty)\omega_X,
 \qquad
 2\nu_{KC}(\{x\})+\nu_{C2}(\{x\})=\nu_{C0}(\{x\}).
\end{equation}
Their global masses satisfy the same equality.  These measures are
not products of $C_\infty$ with weak source limits.
For the upper restriction $U=p_1\delta=e^{a-\ell_L}u_*$, its
induced derivative is $\Xi_L=Xe+u^*\beta_L(e)$ and
$K_\gamma=|\Xi_L|^2/(2g)$, where $e$ spans $L$.
Then
\begin{equation}\label{lower-restriction-probe}
\begin{gathered}
 2K_\gamma\le4K_0+2(p_1-U)\kappa_L,
 \qquad |\partial U|^2/g=2UK_\gamma,\\
 2K_\gamma+U\operatorname{tr}(\Pi D)
   =UU_0+U\kappa_L-tU[a-\operatorname{tr}(\Pi M)]+\Delta U.
\end{gathered}
\end{equation}
The cross in $\Xi_L$ is kept.  In the coincident case the composition
has the uniform bound $C_n\le C u_*^{3/4}$, since
$C_*=u_*r_*$, $r_*\le Cu_*$, and the relative mean and value
bounds apply to $a-b$.  Thus its weighted source atoms vanish at
$q_0$, but this says nothing about unweighted upper-kernel atoms.
For distinct zeros this bound does not assert vanishing at $q_0$.
At a composition root with both arrows nonzero the root-jet identity
is $\nabla(\phi\psi)=\phi(\beta_L-\beta_K)\psi$.
The lower zero is instead treated by the undivided Leibniz identity;
no angle condition is inferred there.

\paragraph{Bounded tests at the divisor.}
The test value can remain correlated with a concentrating source even
when both have separate weak limits. We therefore retain their joint
distribution before taking its fiber over the divisor point.

For each $F_0\in C^2([0,8/3])$, the bounded restriction chain rule is
\[
 \Delta F_0(U)=UF_0'(U)
  [\operatorname{tr}(\Pi D)-U_0-\kappa_L
       +t(a-\operatorname{tr}(\Pi M))]
       +2[F_0'(U)+UF_0''(U)]K_\gamma.
\]
Lift the four positive actual measures $U_0$, $\operatorname{tr}(\Pi D)$,
$\kappa_L$, $K_\gamma$ by $x\mapsto(x,U_n(x))$ before restricting
their weak limits to the fiber over $q_0$.
Call the fiber value measures $\eta_0,\eta_D,\eta_L,\eta_G$.
Bounded potentials have atomless finite-measure Laplacians, and the
damping has bounded base density.  Consequently the full bounded-test
relation is
\begin{equation}\label{lower-restriction-lifts}
 2\int[F_0'(u)+uF_0''(u)]\,d\eta_G
     =\int uF_0'(u)\,d(\eta_0+\eta_L-\eta_D).
\end{equation}
These value measures preserve moving source correlations; they are not
obtained by factoring a source limit against $U_\infty$.

Similarly decompose $K_0=K_{0,\parallel}+K_{0,\perp}$ relative to $J$.
For bounded $F_0\in C^2([\alpha,8/3])$ the upper-norm chain rule is
\[
\begin{aligned}
 \Delta F_0(p_1)
 &=p_1F_0'(p_1)\operatorname{tr}(JD)
   +(p_1-2\alpha)F_0'(p_1)U_0\\
 &\quad+2p_1F_0''(p_1)K_{0,\parallel}
   -2F_0'(p_1)C_0^\tau\\
 &\quad+t p_1F_0'(p_1)[a-\operatorname{tr}(JM)].
\end{aligned}
\]
Lift its positive measures by $x\mapsto(x,p_{1,n}(x))$.
At every point their fiber measures obey
\begin{equation}\label{lower-upper-norm-lifts}
\begin{gathered}
 \int sF_0'(s)d\eta_J+\int(s-2\alpha)F_0'(s)d\eta_0
       +2\int sF_0''(s)d\eta_\parallel=0,\\
 (s-\alpha)\eta_0=\eta_{K_0}
       =\eta_\parallel+(s/2)\eta_K.
\end{gathered}
\end{equation}
The convex test $F_0(s)=(s-2\alpha)_+^3$ makes every integrand
nonnegative; all these lifted atom measures are therefore supported
in $[\alpha,2\alpha]=[1/6,1/3]$.
This does not impose a pointwise upper norm bound on the original solution.

\paragraph{Covariant mixed terms.}
Finally the same actual connection gives a covariant correlation
identity.  Set $A_{\rm com}=[\Pi,J]$,
$B_{\rm com}=[\Pi,[\Pi,J]]$, $\delta=\operatorname{tr}(\Pi J)$.
For real $h\in C^1([0,1])$ and a smooth cutoff $\chi$,
\begin{equation}\label{lower-weighted-commutator}
\begin{aligned}
 \int\chi h(\delta)\operatorname{tr}(B_{\rm com}\widehat H_1)\omega_X
 ={}&-i\int h(\delta)d\chi\wedge\operatorname{tr}(A_{\rm com}\nabla\Pi)\\
 &-i\int\chi h(\delta)\operatorname{tr}(\nabla A_{\rm com}\wedge\nabla\Pi)\\
 &-i\int\chi h'(\delta)d\delta\wedge
                           \operatorname{tr}(A_{\rm com}\nabla\Pi).
\end{aligned}
\end{equation}
This follows from $\nabla^2\Pi=[F_1,\Pi]$ and trace cyclicity.
Finite-rank Cauchy bounds each bulk variation on $U$ by
$C[\int_U\kappa_L\int_U(\kappa_L+\kappa_K)]^{1/2}$;
the cutoff term is at most
$C\|h\|_\infty\|d\chi\|_2(\int_U\kappa_L)^{1/2}$.
Logarithmic point cutoffs, with $n$ first and radius second, give
\[
 |\operatorname{atom}_x\operatorname{tr}(B_{\rm com}\widehat H_1)|
 \le C\sqrt{\mathfrak k_L(x)[\mathfrak k_L(x)+\mathfrak k_K(x)]}.
\]
More fully, after lifting by the angle, the fiber measures satisfy
$\beta_{q_0}(h)=\theta_{0,q_0}(h)+\theta_{1,q_0}(h')$.
The $h'$ term is retained.  In an adapted interior-angle frame with
$(\nabla'_z\Pi)=b e_\perp e^*$ and $(\nabla'_zJ)=w jk^*$,
their densities are
\[
 \Theta_0/\omega_X=2(1-2\delta)\kappa_L
               -2(1-\delta)\operatorname{Re}(b\bar w)/g,
\]
\[
 \Theta_1/\omega_X=2\delta(1-\delta)
                  [\kappa_L-\operatorname{Re}(b\bar w)/g].
\]
The coordinate-free identity covers endpoint angles without division.
Positivity of the actual quotient-tested source also bounds
$\mathfrak k_L(q_0)\le\mathfrak a_Q$;
$\mathfrak k_K(q_0)\le\mathfrak a_0$ follows from upper source accounting.

For this particular traceless mixed probe,
$c<\operatorname{tr}V<W$ and
$\operatorname{tr}(VB_{\rm com})=-2p_1\delta(1-\delta)$ give
\begin{equation}\label{lower-mixed-lyapunov}
 \operatorname{tr}(B_{\rm com}D)=\frac2{\operatorname{tr}V}
  [\operatorname{tr}(B_{\rm com}K_1)
      +\tau\operatorname{tr}(B_{\rm com}F'_{*,1})
      +p_1\delta(1-\delta)\operatorname{tr}D].
\end{equation}
The full raw kinetic cross in a $\Pi$-adapted frame is
$[\operatorname{Re}(x_L\zeta\bar x_Q)
+\operatorname{Re}(\bar y_L\zeta y_Q)]/g$, with
$|\zeta|^2=\delta(1-\delta)$.
Its lower part is bounded by
$\sqrt{2d\delta(1-\delta)K_2\kappa_L}$ using
$|y_Q|^2/g=r\kappa_L$, including $r=0$.
In a $J$-adapted frame write $X=(x_J,x_K)$ and let $\xi$ be the
off-diagonal entry of $\Pi$, $|\xi|^2=\delta(1-\delta)$.
Direct projection algebra gives the complete upper part
\[
 \frac{XB_{\rm com}X^*}{2g}
  =2\delta(1-\delta)(K_{0,\parallel}-K_{0,\perp})
       +(1-2\delta)\operatorname{Re}(x_J\xi\bar x_K)/g,
 \qquad K_{0,\perp}=(p_1/2)\kappa_K.
\]
It has no asserted sign.  Since $t\operatorname{tr}(B_{\rm com}M)$
is a bounded density, common weak extraction in
\eqref{lower-mixed-lyapunov} also identifies the mixed-curvature atom
with the atom of its entire bounded-coefficient source expression.
No angle coefficient is factored out of a concentrating product.
No component is dropped when $\kappa_L$ is positive.
These identities retain both concentrated and remaining correlations.
The final contradiction below uses a bound on the full covariant
derivative, not a classification of all possible lifted atom distributions.

\subsection{Local gradients, symbol gaps and the punctured limit}
\label{lower-punctured-derivatives}

The two-scale gauges give bounded values on compact subsets of $X^\circ$.
To obtain a strong PDE limit, we next prove local gradients and symbol
gaps for the transported equations. The coefficient bounds of
Section~\ref{lower-two-scale} include the possibly persistent damping.
Fixed interior buffers support the derivative argument; circular means
then exclude small symbols at their separate centers. This is a local
estimate, not a smooth extension across $q_0$. All buffers are fixed
before $n$ varies.
Write $S_n=F_n^{1/2}$; in new current-unitary frames let
$u^*$ be the upper row and $v$ the lower column.  The exact
common-connection derivatives are
\begin{equation}\label{lower-transformed-jets}
 X=x_n+(A_n)_zu^*-u^*Z,\qquad
 Y=y_n+Zv-(B_n)_zv,\qquad Z=S_n^{-1}(F_n)_zS_n^{-1}.
\end{equation}
The smooth $x_n,y_n$ retain the reference derivatives and are uniformly
bounded, including at $p$.  Define
$a_n^d=t_n(A_n-L_n)$, $b_n^d=t_n(B_n-L_n)$,
$\widehat\Psi_n=S_n\Psi_n(F_n)S_n^{-1}$.
The transformed equations are
\begin{equation}\label{lower-transformed-equations}
\begin{aligned}
 \Delta A_n&=H'_{*0}-\alpha+a_n^d-(C_0^{\tau_n}+K_0)/q,\\
 \Delta B_n&=c+b_n^d-(C_2^{\tau_n}+K_2)/\beta,\\
 \Delta F_n&=g^{-1}(F_n)_{\bar z}F_n^{-1}(F_n)_z+dF_n
      +F_n\Psi_n(F_n)
      -S_n\mathcal L_V^{-1}(\widetilde C_{1,n}+K_1)S_n.
\end{aligned}
\end{equation}
The product $F_n\Psi_n$ is kept in this order.  The coefficient
bounds in \eqref{lower-damping-limit}--\eqref{lower-polar-force}
show that all regular coordinate/value derivatives through order two
are uniform on the compact metric value set.

\begin{lemma}[Full local first derivatives and all symbol gaps]
\label{lower-local-gradient-gaps}
For every $K\Subset X^\circ$, the fields of
\eqref{lower-transformed-equations} satisfy
\[
 |\nabla A_n|+\|\nabla F_n\|+|\nabla B_n|\le C_K,
 \qquad q_n\ge\delta_{q,K}>0,\quad
 \beta_n\ge\delta_{\beta,K}>0,\quad V_n\ge\delta_{V,K}I.
\]
Constants depend on the fixed data and nested buffers, but not on the
escaping means or the original symbol minima.
\end{lemma}

\begin{proof}
Choose $K\Subset K^+\Subset X^\circ$ and on $K^+$ let
$c_FI\le F_n\le C_FI$, $|A_n|+|B_n|\le B_F$.
Put $h_n=F_n^{-1}$, $\mathcal U=\operatorname{tr}V^{-1}$,
$\mathcal W_1=\operatorname{tr}(V^{-1}K_1)$.
Ordered inverse differentiation gives
\[
 \Delta h_n=g^{-1}h_n(F_n)_zh_n(F_n)_{\bar z}h_n
       -dh_n-\Psi_nh_n
       +S_n^{-1}\mathcal L_V^{-1}(\widetilde C_{1,n}+K_1)S_n^{-1}.
\]
Its trace obeys
\[
 \Delta\operatorname{tr}h_n\ge
       a_F|\nabla F_n|^2+C_F^{-1}\mathcal W_1
           +(A_0/2)C_F^{-1}\mathcal U-C.
\]
For $|\nabla F|^2=2\|F_z\|_{\rm HS}^2/g$, take
$a_F=1/(2C_F^3)$.
The bounded scalar potential
$\Phi_n=\operatorname{tr}F_n^{-1}-A_n+e^{-B_n}$ controls every
gradient.  Indeed
\[
 \Delta e^{-B_n}=e^{-B_n}
 \{ -c-b_n^d+(C_2^{\tau_n}+K_2)/\beta+|(B_n)_z|^2/g\},
\]
and \eqref{lower-transformed-jets}, with $p_1>\alpha$, gives
$|\nabla A_n|^2\le C(K_0+|\nabla F_n|^2+1)$.
Since $V\le(c+d)I$,
$\mathcal W_1\ge(K_0+K_2)/(c+d)$.
The top row and these inequalities yield
\begin{equation}\label{lower-local-density}
 \Delta\Phi_n\ge a\mathcal F_n-C,\quad a>0,\qquad
 \mathcal F_n=|\nabla A_n|^2+|\nabla F_n|^2+|\nabla B_n|^2
     +\mathcal U+\mathcal W_1+q^{-1}+K_0/q+\beta^{-1}+K_2/\beta.
\end{equation}
The full bottom gradient is controlled even when $v=0$.
Dirichlet Green on the same fixed-radius disk $D_{r_0}(z)$ for every
center in a smaller buffer gives
\begin{equation}\label{lower-all-center-log-bound}
 \int_{D_{r_0}(z)}\log\frac{r_0}{|y-z|}\mathcal F_n\,\omega_X\le H,
 \qquad \int_{D_s(z)}\mathcal F_n\,\omega_X
                   \le\frac H{\log(r_0/s)}\to0.
\end{equation}
The full component rows also give
$|\Delta A_n|+|\Delta B_n|+\|\Delta F_n\|\le C(1+\mathcal F_n)$.

Suppose gradients fail to be bounded.  In a fixed interior disk maximize
\[
 \operatorname{dist}(z,\partial D)
       |\nabla(A_n,F_n,B_n)(z)|.
\]
At the chosen point $z_n$ let the gradient be $G_n^\nabla\to\infty$.
Then the distance times $G_n^\nabla$ tends to infinity, and gradients
are at most $2G_n^\nabla$ on the half-distance disk.
The first-scale fields at $z_n+y/G_n^\nabla$ have bounded values and
Lipschitz constants on expanding disks.  Equation
\eqref{lower-all-center-log-bound} makes their energy on every fixed
disk tend to zero; every locally uniform limit is constant.
Derivative convergence still needs the complete face equations.

Set $T_R=\mathcal L_V^{-1}R$, $T_P=\mathcal L_V^{-1}P$, and
$N=\widetilde C_1+K_1$.  The scalar symbol rows are
\begin{equation}\label{lower-local-scalar-rows}
\begin{aligned}
 \Delta q&=(2-p_1/q)K_0+\operatorname{tr}(T_RN)-p_1C_0/q
                       +\mathcal R_q,\\
 \Delta\beta&=-(2+r/\beta)K_2+\operatorname{tr}(T_PN)-rC_2/\beta
                       +\mathcal R_\beta,\\
 \mathcal R_q&=p_1a^d-\operatorname{tr}(R\widehat\Psi),\qquad
 \mathcal R_\beta=rb^d-\operatorname{tr}(P\widehat\Psi).
\end{aligned}
\end{equation}
All remainders are uniformly bounded coefficient functions.
At every point $V\ge(2/3-q)I$ and $\lambda_{\max}V\ge c$.
This leaves precisely strict interior, strict-$V$ scalar faces, pure
matrix face, or mixed matrix/bottom face.

In strict interior all component first-scale Laplacians are two-sided
bounded.  The local Poisson compactness of
Lemma~\ref{common-local-potentials} gives $C^1$ convergence to the
constant value limit, contradicting the unit gradient witness.
If $V$ is strict, the middle derivatives first vanish; strict endpoints
also have vanishing normalized derivative.  A collapsing scalar therefore
preserves $q_z/G_n^\nabla=p_1(A_n)_z/G_n^\nabla+o(1)$ or
$\beta_z/G_n^\nabla=r(B_n)_z/G_n^\nabla+o(1)$.
For a beta face $r=d-\beta\to d$, so no global lower norm is needed.
For $q\le1/18$, $V\ge11I/18$.  The transforms with
\[
 F_q'(q)=\frac{\sqrt q}{\alpha+q},\qquad
 F_b'(\beta)=\frac{\sqrt\beta}{d-\beta}
\]
cancel the singular parallel kinetics in
\eqref{lower-local-scalar-rows}.  The upper perpendicular coefficient
is at most $-18/11$; the beta perpendicular terms also have the
required nonpositive sign.  Thus
\[
 \Delta F_q\le C\sqrt q(1+(G_n^\nabla)^2),\qquad
 \Delta F_b\le C\sqrt\beta(1+(G_n^\nabla)^2).
\]
For the collapsing symbol $s$, use its amplitude scale
$\ell_n=s(z_n)/G_n^\nabla$.  The positive Lipschitz ratio
$Q_n(y)=s(z_n+\ell_ny)/s(z_n)$ has $Q_n(0)=1$.
The normalized transform converges to a positive multiple of
$Q^{3/2}$ with nonpositive distributional Laplacian on the plane.
A nonnegative entire superharmonic function is constant: convolve,
use the circular mean to exclude positive $-\Delta$ mass, and then
use differentiated Poisson estimates on arbitrarily large disks.
Hence $Q=1$.  On the derived disk $Q_n\ge1/2$ the exact row has
bounded two-sided normalized Laplacians, so Poisson $C^1$ compactness
contradicts the preserved symbol derivative.  Simultaneous scalar
collapse is covered by the same argument for either witness.

For the matrix face let $\lambda=\lambda_{\min}V<c/4$,
$\nu=\lambda_{\max}V$.  Then $q\ge11/24$,
$\nu\ge c$, $\nu-\lambda\ge3c/4$.
In its separated eigenframe write
$\rho_1=R_{ee}$, $\pi_1=P_{ee}$,
$\zeta_1=R_{ef}=P_{ef}$, $s_1=\rho_1+\pi_1$.
The full eigenvalue row is
\begin{equation}\label{lower-local-matrix-row}
\begin{aligned}
 \Delta\lambda={}&(|y_1|^2-|x_1|^2)/g
       +\rho_1(C_0+K_0)/q+\pi_1(C_2+K_2)/\beta\\
 &-s_1N_{ee}/\lambda
       -4\operatorname{Re}(\zeta_1N_{fe})/(\lambda+\nu)
       +\mathcal R_\lambda-\mathcal C_{\rm eig},\\
 \mathcal R_\lambda={}&
 [\mathcal L_{R+P}\widehat\Psi-a^dR-b^dP]_{ee},\\
 \mathcal C_{\rm eig}={}&
 \frac{|u_1x_2-y_1\bar v_2|^2+|u_2x_1-y_2\bar v_1|^2}
                  {g(\nu-\lambda)}\ge0.
\end{aligned}
\end{equation}
Both eigenvalue-concavity squares and the common-connection mixed
term are retained.
When $\beta$ is strict the whole singular image of the six-component
row is $-m_fN_f$, with
$m_f=N_{ee}/\lambda$ and $N_f=(0,S_n\Pi_eS_n,0)$.
Five smooth covectors annihilate this moving image exactly.
Differentiating their pairings with $w_z$ kills the singular source
before it is estimated; their coefficient derivatives give bounded
quadratics on the first scale.  This is the moving-annihilator
calculation of \eqref{common-moving-annihilator}, with the bounded
transformed remainders just displayed.  Cauchy-kernel compactness makes
all five projected derivatives tend uniformly to zero, since the
first-scale value limit is constant.  As
$d\lambda(N_f)=s_1\ge c-\lambda$, a nonzero lambda derivative
witness remains.  Equation~\eqref{lower-local-matrix-row} gives
\[
 \Delta\lambda\le-\frac{|\lambda_z|^2}{2g\lambda}
                        +C(1+(G_n^\nabla)^2).
\]
The $\lambda^{3/2}$ transform at scale $\lambda(z_n)/G_n^\nabla$
again gives the constant amplitude limit; the exact two-sided row
contradicts the witness.  Nothing here requires $P>0$ at $p$.

On the mixed face the two singular images are
$-m_fN_f-m_bN_b$, where $m_b=(C_2+K_2)/\beta$,
$N_b=(0,0,1)$.  Four smooth covectors kill both images at every
moving state; the bounded non-exact quadratic terms remain in their
$\bar\partial$ sources.  Their projected derivatives tend to zero.
The remaining differential frame is
\[
 \begin{pmatrix}d\lambda(N_f)&d\lambda(N_b)\\
                 d\beta(N_f)&d\beta(N_b)\end{pmatrix}
 =\begin{pmatrix}s_1&-\pi_1\\-\pi_1&r\end{pmatrix},
 \qquad\det\ge(3c/4)(d/2)=55/96.
\]
For $1<j<2$ put $s_j=\lambda+j\beta$.
The full combined row is
\begin{equation}\label{lower-local-combined-row}
\begin{aligned}
 \Delta s_j={}&-(s_1-j\pi_1)N_{ee}/\lambda
             -(jr-\pi_1)(C_2+K_2)/\beta\\
 &+(|y_1|^2-|x_1|^2)/g-2jK_2+\rho_1(C_0+K_0)/q\\
 &+jP_{ff}N_{ff}/\nu
       +4(j-1)\operatorname{Re}(\zeta_1N_{fe})/(\lambda+\nu)\\
 &+\mathcal R_\lambda+j\mathcal R_\beta-\mathcal C_{\rm eig}.
\end{aligned}
\end{equation}
For $j=5/4,7/4$, $s_1-j\pi_1\ge5/8$ and
$jr-\pi_1\ge(j-1)d/2$.  The derivative formula and three-term
square inequality therefore give the common sufficient constant
\[
 \Delta s_j\le-\frac{|(s_j)_z|^2}{60g s_j}
                       +C(1+(G_n^\nabla)^2),\qquad
 \Delta s_j^{61/60}\le C s_j^{1/60}(1+(G_n^\nabla)^2).
\]
Use the shared scale
$\delta_n=\lambda(z_n)+\beta(z_n)$,
$\ell_n=\delta_n/G_n^\nabla$.  Entire superharmonic Liouville makes
both transformed amplitudes constant, hence
$\lambda/\delta_n\to L$, $\beta/\delta_n\to B$, $L+B=1$.
The four annihilator projections and the invertible frame also give
$(w(z_n+\ell_ny)-w(z_n))/\delta_n\to0$ locally uniformly.
If $L,B>0$, both relative denominators are buffered and two-sided
Poisson compactness contradicts the gradient witness.
If $L=0,B=1$, five covectors killing $N_f$ preserve a lambda witness.
On its smaller scale $\lambda(z_n)/G_n^\nabla$, beta is comparable
to $\delta_n$ and the transformed lambda error carries
$\lambda(z_n)/\beta\to0$; its exact two-sided row gives the contradiction.
If $B=0,L=1$, the $A,F$ derivatives vanish by their bounded sources;
the witness is in $B_n$, hence $\beta_z/G_n^\nabla$ because $r\to d$.
On the smaller beta scale lambda is comparable to $\delta_n$, and the
error carries $\beta(z_n)/\lambda\to0$.  The beta row gives the
contradiction.  No unproved relative-error rate is divided by a vanishing
symbol.  This exhausts the faces and proves the full interior gradient bound.

To obtain gaps without a global minimizing center, use the elementary
circular-mean lemma.  If positive $C^2$ functions on a fixed disk have
$|\nabla s_n|\le L\ge1$, $\Delta_{\mathbb R^2}s_n\le C-c/s_n$,
and $s_n(0)=\epsilon_n\to0$, then $s_n(y)\le\epsilon_n+L|y|$.
Their circular means satisfy
\[
 m_n(r)\le\epsilon_n+Cr^2/4
       -c\int_0^r\log(r/u)\frac u{\epsilon_n+Lu}\,du.
\]
For fixed $r$ its upper limit is $Cr^2/4-cr/L<0$ when $r$ is
sufficiently small, contradicting positivity.  The fixed conformal
factor converts the $\Delta_X$ inequalities to this form.
Insert nested buffers $K\Subset K_1\Subset K_2\Subset K^+$.
The proved gradient bound on $K_2$ bounds every kinetic and all symbol
Lipschitz constants.  For $q\le1/18$ the scalar row gives
$\Delta q\le C-\alpha(55/63)/q$, proving a q-gap on $K_1$.
First prove the same-point combined gap for $s=\lambda+\tfrac32\beta$:
on its small separated chart the combined row gives
\[
 \Delta s\le C-\frac{(5/8)(A_0/2)}\lambda
                -\frac{(d/4)A_0}\beta\le C-c_s/s.
\]
Thus $s\ge\delta_s>0$ on $K_1$.
A small lambda at a center in $K$ stays small on a fixed Lipschitz
disk in $K_1$; there $\beta\ge\delta_s/3$, and
$\Delta\lambda\le C-(3c/4)(A_0/2)/\lambda$ gives the matrix gap.
A small beta center instead has $\lambda\ge\delta_s/2$ and
$r\ge d/2$ on its disk; the beta row gives
$\Delta\beta\le C-(d/2)A_0/\beta$.
Different centers have never been synchronized; repeated matrix spectra
are avoided only on the small-lambda separated chart where they cannot
occur.  All three local gaps follow.
\end{proof}

\begin{lemma}[Strong actual punctured PDE limit]\label{lower-strong-punctured-limit}
A common subsequence converges strongly in $C^{2,\gamma'}$ on every
compact subset of $X^\circ$, for each fixed $0<\gamma'<1$, to full
positive Hermitian $F$ and scalar fields $A,B$.
The limiting arrows are $(\phi_L,0)$ and $(\sigma,0)^T$ and their
actual covariant derivatives are retained.  In current unitary frames
the limit satisfies
\begin{equation}\label{lower-punctured-rows}
\begin{gathered}
 R=p_1J,\quad P=vv^*=r\Pi,\quad q=p_1-\alpha,\quad
 \beta=d-r,\quad V=cI-R+P,\quad K=I-J,\\
 U_0=(A_0+K_0)/q,\quad U_2=(A_0+K_2)/\beta,\quad
 D=\mathcal L_V^{-1}(A_0I+K_1),\\
 H_0=\alpha+\rho+U_0,\quad H_2=\alpha+\rho+U_2,\quad
 \widehat H_1=\alpha I+D+\rho J-\rho\vartheta K.
\end{gathered}
\end{equation}
No extension of the metric at $q_0$ is asserted.
\end{lemma}

\begin{proof}
On the compact value/gradient/gap set established above, square root,
matrix inverse and scalar reciprocals have uniformly bounded finite
derivatives.  The Lyapunov eigenmultipliers are
$(v_i+v_j)/2\ge\delta_{V,K}$; its coefficient derivative is
$D(\mathcal L_V^{-1})[\dot V]=
-\mathcal L_V^{-1}\mathcal L_{\dot V}\mathcal L_V^{-1}$.
Thus the complete six real component equations have scalar principal
part and bounded right sides, smooth and uniformly Lipschitz in
coordinate/value/first-derivative arguments.
The local logarithmic potential of a bounded source has gradient
modulus $Ch(1+|\log h|)$: split the $1/r$ kernel at distance $2h$,
then integrate the far $h/r^2$ difference.  Bounded harmonic remainders
have interior differentiated Poisson estimates.  This gives uniform
$C^{1,\gamma}$ bounds for every $\gamma<1$.
The full right sides are now uniformly $C^{0,\gamma}$.  In their
Hessian potentials subtract the source value from the zero-angular-mean
$r^{-2}$ kernel.  The near integral is bounded by
$\int_0^r s^{\gamma-1}ds$; the far comparison is
$h\int_{2h}^R s^{\gamma-2}ds=O(h^\gamma)$.
The fixed-cutoff and harmonic terms are smooth.  These are the actual
local Poisson estimates of Lemma~\ref{common-local-potentials}; they
give uniform $C^{2,\gamma}$ bounds on smaller buffers.
Ascoli and lower-exponent interpolation, followed by a countable
exhaustion and diagonal subsequence, give the stated strong local limit.
All ordered sources, curvatures and covariant products then converge
strongly locally.

The explicit coefficient convergence of \eqref{lower-damping-limit}
passes the persistent damping
$-\rho J(F)+\rho\vartheta(I-J(F))$.
Since $\tau_n=t_n/\mu_s\to0$ and the mass is fixed, the rotated
original residual in \eqref{lower-polar-force} vanishes, leaving exactly
$A_0I$.  In coefficient coordinates the limiting equations are
\[
 \Delta A=H'_{*0}-\alpha-\rho-(A_0+K_0)/q,
 \quad\Delta B=c-\rho-(A_0+K_2)/\beta,
\]
\[
 \Delta F=g^{-1}F_{\bar z}F^{-1}F_z+dF
       +F[-\rho J(F)+\rho\vartheta(I-J(F))]
       -F^{1/2}\mathcal L_V^{-1}(A_0I+K_1)F^{1/2}.
\]
These give \eqref{lower-punctured-rows} after current-unitary transport.
All $X,Y$ reference terms are the derivatives of the displayed limiting
arrows with the limiting metrics; they remain at $p$ when $p\ne q_0$.
Local strong convergence identifies all source densities on fixed
punctured tests; inherited global budgets and monotone exhaustion
give finite punctured masses.  No global uniform integrability, removal,
or interchange of $n$ with puncture removal has been used.
\end{proof}

\subsection{The angular potential and the original divisor flux}
\label{lower-angular-flux}

The strong punctured limit identifies source and curvature densities
away from $q_0$. To recover their global angular flux, we retain the
wedge divisor of the finite chains. For every finite $n$, the upper
kernel is $\mathcal O(1)$, while the
punctured limiting kernel is the constant $\mathcal O(2)e_2$.
The original wedge divisor records this degree difference.
Let
\[
 \delta_n=\operatorname{tr}(\Pi_nJ_n)
         =u_*e^{a_n-\ell_{L,n}}/p_{1,n}.
\]
Its only zero is the simple-section norm zero at $q_0$; it is positive
elsewhere.  The centered lower bound for $a$ and
\eqref{lower-logarithmic-values} give the integrable envelope
\begin{equation}\label{lower-angle-envelope}
 -C+\tfrac54\log u_*\le\log\delta_n\le0.
\end{equation}
Rescale $u_*$ by a fixed harmless constant so $u_*\le1$ in this
estimate.  The strong punctured limit therefore gives global $L^1$
convergence of $\log\delta_n$, by dominated convergence.
In the limit
\[
 \delta=\operatorname{tr}(\Pi J)
       =\det F/(F_{11}F_{22})>0\quad\hbox{on }X^\circ.
\]
The growth bounds also give $-C+\tfrac7{10}\log u_*\le\log\delta\le0$;
this is only a growth estimate, not metric removal.

The finite wedge Gauss--Codazzi formula, with the original kernel, is
\begin{equation}\label{lower-original-wedge}
 (\Delta\log\delta_n)\omega_X
   =\operatorname{tr}[(\Pi_n-J_n)\widehat H_{1,n}]\omega_X
      -(\kappa_{L,n}+\kappa_{K,n})\omega_X
      +\mathcal V\delta_{q_0}.
\end{equation}
All measures have bounded total variation.  Taking $n$ first identifies
the distributional limit from the global $L^1$ convergence; the locally
strong PDE identifies its density off $q_0$.
Thus
\begin{equation}\label{lower-angle-distribution}
 (\Delta\log\delta)\omega_X
     =f_\delta\omega_X+\gamma_{\rm fl}\mathcal V\delta_{q_0},
 \qquad
 \gamma_{\rm fl}=1+A_L-A_{\rm at},\quad 0\le\gamma_{\rm fl}\le\frac15.
\end{equation}
The coefficient follows directly from
$\log\delta=\log u_*+A-\log F_{11}-\log p_1$.
The atoms of $\Delta A,\Delta\log F_{11}$ are
$-\mathfrak a_0,-\mathfrak a_L$; $\log p_1$ is bounded with
finite-measure Laplacian and has no atom.  The lower bound for
$\gamma_{\rm fl}$ is $A_L\ge A_{\rm at}-1$.
The atom law gives its upper bound
$\gamma_{\rm fl}\le1-(4/5)A_{\rm at}\le1/5$.
All original atoms remain explicit.

For completeness, the full punctured angular equation retains the
noncommuting kinetic bracket.  Let $B_\delta=\Pi-J$,
$T_V=\operatorname{tr}V$.  Curvature subtraction for the two constant
holomorphic limit lines gives
\begin{equation}\label{lower-angular-equation}
 \Delta\log\delta=\operatorname{tr}(B_\delta D)
       -\rho(1+\vartheta)(1-\delta)-\kappa_L-\kappa_K.
\end{equation}
For any traceless Hermitian $2\times2$ matrix $B$,
$VB+BV-T_VB=\operatorname{tr}(VB)I$.  It follows, by substitution in
the Lyapunov equation, that
\[
 \mathcal L_V^{-1}B
  =\frac2{T_V}\{B-\tfrac12\operatorname{tr}(VB)V^{-1}\}.
\]
For $B_\delta$, $\operatorname{tr}(VB_\delta)=(p_1+r)(1-\delta)$.
Selfadjointness of the ordered inverse gives
\begin{equation}\label{lower-full-angular-bracket}
\begin{aligned}
 \Delta\log\delta={}&\frac2{T_V}
   \{\operatorname{tr}(B_\delta K_1)
     -\tfrac12(p_1+r)(1-\delta)\operatorname{tr}(V^{-1}K_1)\}\\
 &-\{A_0(p_1+r)/\det V+\rho(1+\vartheta)\}(1-\delta)
       -\kappa_L-\kappa_K.
\end{aligned}
\end{equation}
The complete kinetic bracket is not assigned a sign.  Its parallel and
perpendicular $X/Y$ components remain.  The undecomposed equation
has $|f_\delta|\le2\operatorname{tr}D+C+\kappa_L+\kappa_K$,
an integrable bound.  Applying \eqref{lower-angle-distribution} to
the constant test on compact $X$ therefore gives the global flux identity
\begin{equation}\label{lower-flux-compatibility}
 \mathcal V^{-1}\int_{X^\circ}
  [\operatorname{tr}(B_\delta D)-\rho(1+\vartheta)(1-\delta)
                             -\kappa_L-\kappa_K]
       =-\gamma_{\rm fl}.
\end{equation}
This preserves the finite $\mathcal O(1)$ versus limiting
$\mathcal O(2)$ kernel charge.

Let $H_K=2-\Delta\log F_{22}$ be the limiting kernel's induced
curvature on $X^\circ$.  Since
$\log F_{22}=\ell_Q^\infty-\log\delta$, its distributional extension
is $H_K\omega_X=H_Q\omega_X+(\Delta\log\delta)\omega_X$.
The physical quotient has total curvature $2\mathcal V$ and original
atom $\mathfrak a_Q$.  Thus
\begin{equation}\label{lower-kernel-curvature-mass}
 \mathcal V^{-1}\int_{X^\circ}H_K
   =2-A_Q-\gamma_{\rm fl}
   =1+\frac45A_{\rm at}+\frac{11}{5}A_2.
\end{equation}
The last equality uses
$\alpha A_{\rm at}=c(A_L+A_Q)+dA_2$.
Gauss--Codazzi and the persistent kernel damping give
$H_K=\alpha-\rho\vartheta+\operatorname{tr}(KD)-\kappa_K$.
Local strong convergence identifies
$k^\circ=\mathcal V^{-1}\int_{X^\circ}\kappa_K$.
Consequently
\begin{equation}\label{lower-kernel-source-mass}
 \mathcal V^{-1}\int_{X^\circ}\operatorname{tr}(KD)
 =1+\frac45A_{\rm at}+\frac{11}{5}A_2-\alpha
                                  +\rho\vartheta+k^\circ.
\end{equation}

\subsection{Controlling the full lower-arrow derivative}
\label{lower-covariant-payment}

The angular-flux calculation fixes the punctured upper-kernel curvature
mass. To compare it with the source bound, we still need to control the
mixed derivative term without assigning it a sign. We test the lower
arrow against its bounded kernel projection, remove the cutoff errors
using finite intrinsic energy, and use the full bottom-row source to
bound the cross term. This is a global integral estimate on the
punctured curve, not an extension theorem for the limiting metric.

The limiting lower arrow is a holomorphic section $v$ of the lower
Hom-bundle, intrinsically bounded by $|v|^2=r<d$, with full
common-connection derivative $Y$. Here $K=I-J$ projects onto the upper
kernel, whereas $\Pi$ projects onto the lower line.
In its actual common Chern connection the curvature is
\[
 \widehat H_1-H_2I=D-U_2I-m_\rho K,\qquad
 m_\rho=\rho(1+\vartheta)\ge0.
\]
Because $K$ is the holomorphic kernel on $X^\circ$,
$W_K=\nabla'_zJ=JW_KK$, $\nabla'_zK=-W_K$.
Differentiating $\phi_\infty K=0$ gives
$XK=u^*W_K$, so
$|XK|^2/(2g)=(p_1/2)\kappa_K$ and
$\|W_K\|_{\rm HS}^2/g=\kappa_K$.
Use the bounded section $\zeta=Kv$, and put
\[
 \begin{aligned}
 u_K&=|\zeta|^2=r(1-\delta)\le d,\quad
 K_{2,K}=|KY|^2/(2g),\quad K_{2,J}=|JY|^2/(2g),\\
 \mathscr X&=\operatorname{Re}\langle JY,W_KKv\rangle/g.
 \end{aligned}
\]
The product rule gives
$\nabla'_z\zeta=KY-W_KKv$ and
$\nabla''\zeta=-W_K^*Jv$.
Hence $|\nabla'\zeta|^2/g\le C(K_2+d\kappa_K)$ and
$|\nabla''\zeta|^2/g\le d\kappa_K$.
All these intrinsic energies are finite by the inherited budgets.

For a compactly supported cutoff $\chi$ on $X^\circ$, the bilinear
Bochner formula for the holomorphic section $v$ is
\[
 \int\chi\,\operatorname{Re}\langle Y,\nabla'_z\zeta\rangle/g
 =\int\chi\,\operatorname{Re}
     \langle(\widehat H_1-H_2I)v,\zeta\rangle+\mathcal E_\chi.
\]
In a holomorphic frame it is obtained from
$\nabla''\nabla'v=-g(\widehat H_1-H_2I)v$ by integration by parts.
Choose logarithmic capacity cutoffs vanishing near $q_0$, tending to
one off $q_0$, with $\|d\chi_\epsilon\|_2\to0$.
A smooth profile between radii $\epsilon^2$ and $\epsilon$ gives
$\|d\chi_\epsilon\|_2^2\le C/|\log\epsilon|$ by polar integration.
The cutoff error is bounded by
\[
 |\mathcal E_{\chi_\epsilon}|\le C\sqrt d\,
    \|d\chi_\epsilon\|_2(\int_{X^\circ}|Y|^2/g)^{1/2}\to0.
\]
Both bulk sides are integrable.  First take the established strong limit
on fixed punctured tests, then remove the cutoff.  This gives the
complete identity
\begin{equation}\label{lower-bilinear-bochner-payment}
 \int_{X^\circ}(2K_{2,K}-\mathscr X)
  =\int_{X^\circ}
       [\operatorname{Re}\langle Dv,Kv\rangle-u_K(U_2+m_\rho)].
\end{equation}
Bounded intrinsic norms suffice; no coordinate metric removal is needed.

Probe the ordered row $\mathcal L_VD=A_0I+K_1$ with $K$.
The $R=p_1J$ terms vanish under this trace pairing, whereas the
$P=vv^*$ terms are
$\operatorname{tr}(K\mathcal L_PD)=\operatorname{Re}\langle Dv,Kv\rangle$.
Therefore
\begin{equation}\label{lower-ordered-kernel-probe}
 c\operatorname{tr}(KD)+\operatorname{Re}\langle Dv,Kv\rangle
                =A_0+\tfrac{p_1}{2}\kappa_K+K_{2,K}.
\end{equation}
Combining its integral with \eqref{lower-bilinear-bochner-payment} gives
\[
 c\int_{X^\circ}\operatorname{tr}(KD)-A_0\mathcal V
  =\int_{X^\circ}
      [\tfrac{p_1}{2}\kappa_K-K_{2,K}
                    +\mathscr X-u_K(U_2+m_\rho)].
\]
The parallel $JY$ has no asserted favorable sign.  Instead use
$|W_KKv|^2/g=u_K\kappa_K$ and the ordinary square inequality
with the actual positive $\beta$:
\begin{equation}\label{lower-full-y-payment}
\begin{aligned}
 \mathscr X&\le(u_K/\beta)K_{2,J}+(\beta/2)\kappa_K,\\
 \mathscr X-u_KU_2-K_{2,K}
 &\le(\beta/2)\kappa_K-A_0u_K/\beta
                           -(1+u_K/\beta)K_{2,K}.
\end{aligned}
\end{equation}
All terms are integrable, since $u_KK_2/\beta\le dU_2$.
This remains valid at $u_K=0$ and never divides by $r$ or $\sigma$.
The positive $J$-Rayleigh quotient of $V$ gives
$0<c-p_1+r\delta$, hence
$p_1+\beta=p_1+d-r<c+d=8/3$.
Since $m_\rho\ge0$, we obtain
\begin{equation}\label{lower-global-kernel-cost}
 c\int_{X^\circ}\operatorname{tr}(KD)
              \le A_0\mathcal V+\frac43\int_{X^\circ}\kappa_K.
\end{equation}
This is the required global covariant estimate: the full bottom source
controls the cross term.

Insert \eqref{lower-kernel-source-mass} in
\eqref{lower-global-kernel-cost}.  With the frozen constants the result
is exactly
\begin{equation}\label{lower-final-incompatibility}
 \frac12k^\circ\ge\frac23A_{\rm at}+\frac{11}{6}A_2-\frac13
                                  +\frac56\rho\vartheta.
\end{equation}
Since $A_{\rm at}\ge1$, $A_2\ge0$, $\rho\ge0$ and
$\vartheta\ge777/323>0$, it forces
\[
 k^\circ\ge\frac23+\frac53\rho\vartheta\ge\frac23.
\]
But \eqref{lower-kernel-allocation-bound} gives
$k^\circ+\rho\le88/233<2/3$.  This contradiction excludes every
original unbounded solution sequence in both divisor configurations.

\begin{theorem}[Uniform values for every original lower-boundary chain]
\label{lower-value-bound}
Fix the original rank-$(2,1,1)$ coefficient chain, background, reference,
one finite mass and slab as above.  Every normalized positive-symbol
$C^{2,\gamma}$ auxiliary solution with $0<\tau\le\tau_b$ obeys
\[
 \|a\|_\infty+\|b\|_\infty+\|M\|_{{\rm op},\infty}\le B_*.
\]
The bound is independent of $\tau$, the solution and its original
symbol minima; dependence on the fixed chain/reference/mass is allowed.
\end{theorem}
\begin{proof}
If the solution family is empty the bound is vacuous.  Otherwise the
contradiction above gives a finite
$U_*:=\sup\overline{\log\operatorname{tr}(e^{-M})}$: an infinite
supremum would let us choose actual solutions with means at least $n$,
the excluded sequence.
Set $D_-:=\max(0,U_*+C_-)$ and
$H_*=(C_T+D_-W)/\delta$, $\delta=4/21$.
Lemma~\ref{lower-mean-transfer} gives nonnegative shifted fields;
the reference-weight comparison gives
\[
 (a+D_-)+\operatorname{tr}(M+D_-I)+(b+D_-)\le H_*.
\]
Every summand and every shifted middle eigenvalue is bounded.
For instance $B_*=3(D_-+H_*)$ bounds the requested sum.
This reconstruction uses only the one scalar mean and applies to all
six unrestricted fields, including at $p$ and when $p=q_0$.
\end{proof}

\begin{theorem}[Original rank-one-lower existence]\label{lower-existence}
Every original chain $\mathcal O(1)\xrightarrow{\psi}\mathcal O(2)^2
\xrightarrow{\phi}\mathcal O(3)$ with upper coefficient rank two,
lower coefficient rank one and nonzero composition admits a smooth
positive Hermitian metric triple satisfying the frozen reduced dHYM
equation and strict matrix calibration in all three blocks.
\end{theorem}
\begin{proof}
For this zero-safe reference choose once
\[
 \mu_s\ge\max(t_{\rm req},T_{\rm all})/\tau_b,
\]
using the finite inverse, contraction and all-solution anchor constants
of Lemma~\ref{common-anchor}.  This is before choosing a positive stage
floor.  All coefficients there are smooth at $p$, including $D\psi(p)$,
and have no inverse lower norm.  Thus the independent full six-field
construction of Theorem~\ref{common-positive-continuation} supplies an
actual normalized smooth positive-symbol solution at every
$\tau\in(0,1]$, at this same fixed mass.  Its bounded-stage compact
map and component argument do not assume an original zero endpoint.

Theorem~\ref{lower-value-bound} supplies a parameter-uniform value
bound on the fixed positive-force slab.  The hypotheses of
Theorem~\ref{common-gradient-gaps} are now satisfied with this bound,
the same mass/reference and $c_1=A_0/2$.
That theorem gives global gradients and positive $q,\beta,V$ gaps
without a positive lower-arrow norm. Choose $\tau_n=\tau_b/(n+1)$
and a solution at each parameter. The same fixed reference and mass,
these global bounds and positive-parameter existence meet the hypotheses
of Theorem~\ref{common-original-endpoint}. Its ordered inverse is
$E_M^{-1}\mathcal L_V^{-1}$, with the Lyapunov inverse applied first.
All coefficients remain smooth at the lower zero, including $D\psi(p)$,
and at repeated middle spectra; none divides by the lower norm.
The theorem gives strong global $C^{2,\gamma}$ subsequence convergence,
the single normalization, smoothness, the undamped equation and strict
calibration on the original bundles. This application includes $p=q_0$.
Here the original values are bounded, so fixed mass makes both the
damping and reference-forcing terms vanish. In contrast, the escaping
punctured sequence had $t_nL_n\to\rho$, which produces the persistent
damping in \eqref{lower-punctured-rows}.
\end{proof}

\section*{Acknowledgments}

The initial exploration, analytical derivations, and drafting of this manuscript were substantially assisted by OpenAI's GPT-5.6 Sol and GPT-6 Astra under the authors' direction. The authors have checked the arguments and take responsibility for the mathematical claims.

\end{document}